\RequirePackage{fix-cm}
\documentclass[smallextended]{svjour3}       
\smartqed  
\usepackage{graphicx}
\usepackage{mathptmx}      
\usepackage[numbers]{natbib}
\usepackage{amsmath, amssymb, amsfonts,medmath,mathrsfs}
\usepackage[percent]{overpic}
\usepackage{appendix}
\usepackage{placeins}
\usepackage{multirow}
\usepackage{hyperref}

\usepackage[dvipsnames,usenames]{xcolor}

\newcommand{\RR}{\mathbb{R}}

\newcommand{\ZZ}{\mathbb{Z}}

\newcommand{\bD}{\mathbf{D}}

\newcommand{\bH}{\mathbf{H}}
\newcommand{\bI}{\mathbf{I}}
\newcommand{\bJ}{\mathbf{J}}

\newcommand{\bd}{\mathbf{d}}

\newcommand{\bm}{\mathbf{m}}

\newcommand{\bp}{\mathbf{p}}

\newcommand{\br}{\mathbf{r}}
\newcommand{\bs}{\mathbf{s}}

\newcommand{\bv}{\mathbf{v}}
\newcommand{\bw}{\mathbf{w}}
\newcommand{\bx}{\mathbf{x}}

\newcommand{\bz}{\mathbf{z}}

\newcommand{\bseps}{\boldsymbol{\epsilon}}

\newcommand{\bsrho}{\boldsymbol{\rho}}

\newcommand{\bszero}{\mathbf{0}}
\newcommand{\bsone}{\mathbf{1}}

\newcommand{\cC}{\mathcal{C}}

\newcommand{\cN}{\mathcal{N}}

\DeclareMathOperator*{\argmin}{argmin}

\newcommand{\re}{{\rm e}}
\newcommand{\rI}{{\rm I}}
\newcommand{\rF}{{\rm F}}

\journalname{Journal of Scientific Computing}
\begin{document}

\title{A Gauss–Newton Method with No Additional PDE Solves Beyond Gradient Evaluation for Large-Scale PDE-Constrained Inverse Problems}


\author{Cash Cherry \and
       Samy Wu Fung \and
        Luis Tenorio \and 
        Ebru Bozda\u{g} 
}


 \institute{C. Cherry \at
               Department of Applied Mathematics \& Statistics, Colorado School of Mines, Golden, CO \\
               \email{ccherry@mines.edu},           
               \emph{now at the University of California Santa Barbara} 
            \and
            S. Wu Fung \at
               Department of Applied Mathematics \& Statistics, Colorado School of Mines, Golden, CO \\
            \and
            L. Tenorio \at
               Department of Applied Mathematics \& Statistics, Colorado School of Mines, Golden, CO \\ 
            \and
            E. Bozda{\u g} \at
               Departments of Applied Mathematics \& Statistics and Geophysics, Colorado School of Mines, Golden, CO \\
 }

\date{}

\maketitle
\begin{abstract}
Partial Differential Equation (PDE)-constrained optimization problems often take the form of an optimization of an objective function given as a sum of loss terms. 
Each function or gradient evaluation requires one or more PDE solves, which render these problems computationally demanding. While Gauss–Newton methods are well-suited for large-scale PDE-constrained optimization, their application to settings such as Full-Waveform Inversion (FWI) is hindered by the need for additional PDE solves to compute Jacobian–vector products. 
This paper proposes a Gauss–Newton approach that eliminates the need for extra PDE solves beyond those required for gradient computation. 
Our numerical experiments on FWI demonstrate that the proposed method achieves the efficiency of gradient-based schemes while retaining the fast convergence of Gauss–Newton methods.
\end{abstract}

\section{Introduction}
\label{intro}

Large-scale Partial Differential Equation (PDE)-constrained optimization problems arise in a wide variety of scientific and engineering applications, including geophysical inversions~\cite{Tromp05, fung2019multiscale}, fluid dynamics~\cite{cotter2009bayesian, lesnic2021inverse}, medical imaging~\cite{mueller2012linear, ruthotto2012diffeomorphic}, phase retrieval~\cite{shechtman2015phase, parada2025fast, fung2020multigrid}, game theory~\cite{ding2022mean, chow2022numerical}, and machine learning~\cite{liu2025revisiting, ruthotto2020deep, wu2020admm, ye2022adaptive}. In such problems, one seeks to optimize an objective functional subject to physical laws expressed as PDEs, which serve as equality constraints coupling the state and control variables. These formulations are often computationally intensive because each evaluation of the objective function or its derivatives typically requires the solution of large-scale PDE systems.

A common feature of these problems is an
objective function $F : \RR^p \to \RR$  of the form
\begin{equation}\label{eq:emp-min}
    \min_{\bm} \; F(\bm) \quad \mbox{with}\quad F(\bm) =  \sum_{i=1}^N \phi_i(\bm) + R(\bm),
\end{equation}
where $\bm \in \RR^p$ is a vector of the model parameters we wish to recover, each  function $\phi_i\colon \mathbb{R}^p \to \mathbb{R}$ corresponds to a PDE-constrained subproblem associated with a particular experiment or data sample, and $R\colon \mathbb{R}^p \to \mathbb{R}$ is a regularization term. This structure naturally arises in inverse problems such as Full-Waveform Inversion (FWI), where recordings from each seismic measurement contribute an independent loss term \cite{Modrak, GladM15}. The same structure also arises in empirical risk minimization in machine learning~\cite{bottou2018optimization}.

In large-scale settings, the computational cost of repeatedly solving PDEs until convergence of the model parameters prohibits explicit computations of the Jacobians or Hessians, which motivates the use of \emph{matrix-free} optimization algorithms (i.e. no explicit computation of a Jacobian or Hessian). First-order methods, such as conjugate gradient or limited-memory quasi-Newton methods, are often preferred due to their low per-iteration cost and scalability  \cite{Nocedal}. For problems where higher accuracy is desired, {Gauss--Newton} (GN) methods, particularly in their conjugate gradient (CG) variants, are frequently employed~\cite{haber2014computational, fung2019large, fung2019uncertainty}. Matrix-free Gauss-Newton methods can achieve faster local convergence than first order methods without explicitly forming or storing Jacobians \cite{Nocedal}. Instead, they rely on matrix-free Jacobian--vector and vector--Jacobian products that can be computed at the cost of additional PDE solves to iteratively solve the linear Gauss-Newton system at each optimization iteration.

However, a key drawback of Gauss--Newton--CG methods is their computational cost per GN iteration. Each GN iteration typically involves at least one inner CG iteration to solve a linear system, and every CG iteration requires at least one Jacobian--vector and one vector--Jacobian product, each corresponding to one or more additional PDE solves. As a result, even though Gauss--Newton methods converge in fewer outer iterations than first-order schemes, their per-iteration cost may dominate the total runtime for large-scale PDE problems such as FWI.

We propose a Gauss-Newton formulation whose updates do not require additional PDE solves beyond those already required for gradient computations. Our proposed method provides a compromise between  first and second-order methods that retains the efficiency of gradient-based optimization while inheriting the fast local convergence properties of Gauss--Newton methods. We demonstrate the effectiveness of this approach on large-scale FWI problems, highlighting its computational advantages and scalability.

\section{Mathematical Background}
\label{section:math_background}
We consider observations modeled as
\[
    \bd_i = G_i(\bm) + \bseps_i, 
    \quad i = 1, \ldots, N,
\]
where $G_i \colon \RR^p \to \RR^n$ denotes a (nonlinear) forward operator that maps the model parameters 
$\bm$ to predicted data space, and $\bseps_i$ represents measurement error. In many applications (e.g., FWI, Direct Current Resistivity  and medical imaging) it is common  to have $N \ll p$.
The goal of the inverse problem is to recover the model parameters $\bm$ that best explain 
the observed data $\bd_i , i=1,\ldots, N$, by solving a minimization problem similar to \eqref{eq:emp-min}:
\begin{equation}
    \min_{\bm} \; F(\bm)\quad\mbox{with}\quad F(\bm)=
    \Phi(\bm)+ R(\bm),
    \label{eq:objective}
\end{equation}
where $\Phi$ is the sum of squared norms of residuals $\br_i(\bm) \equiv  G_i(\bm) - \bd_i$, so that    
\[
\Phi(\bm) \,=\,  \sum_{i=1}^N \phi_i(\bm) =  \frac{1}{2} \sum_{i=1}^N 
    \| \br_i(\bm) \|^2,
\]
and $R \colon \RR^p \to \RR$ is a function used for regularization (e.g., Tikhonov~\cite{mueller2012linear}, Total Variation~\cite{rudin1992nonlinear} or sparsity-promoting regularization
\cite{tibshirani1996regression}). For ease of presentation, we assume that $R$ is twice continuously-differentiable, but the proposed methodology can also be applied when $R$ is not differentiable, e.g., $R(\bm) = \| \bm \|_1$. Similarly, though we consider least squares misfits $\phi_i ( \bm ) = \frac{1}{2} \| \br_i ( \bm ) \|^2,$ our methodology may be applied to problems of the more general form in \eqref{eq:emp-min}.

To perform gradient-based optimization, one must first compute the gradient of the objective with respect to $\bm$, which takes the form
\begin{equation}
    \label{eq:gradient}
    \nabla F(\bm)
    = \sum_{i=1}^N \nabla \phi_i(\bm) + \nabla R(\bm) =  \sum_{i=1}^N \bJ_i(\bm)^\top 
  \br_i(\bm) + \nabla R(\bm),
\end{equation}
where $\bJ_i(\bm) \in \RR^{n\times p}$ denotes the Jacobian matrix of $G_i$ at $\bm$ (sometimes called a \emph{sensitivity} matrix).
Each gradient $\nabla\phi_i$ typically requires one forward and one adjoint PDE solve. The prototypical first-order algorithm, gradient-descent (GD), has an update given by 
\begin{equation}
    \bm_{k+1} = \bm_k + \alpha_k \,\bp^{\rm GD}_k, \quad \text{with} \quad \bp^{\rm GD}_k = -\nabla F(\bm_k),
    \label{eq:gd_update}
\end{equation}
where the stepsize $\alpha_k > 0$ is usually chosen to either be constant or to satisfy a linesearch condition~\cite{Nocedal}. The CG and LBFGS updates we  use in our tests have similar updates, with the only difference being that $\bp_k^{\rm CG}$ and $\bp_k^{\rm LBFGS}$ are computed as functions of current \emph{and previous} gradients and models to accelerate convergence.

\subsection{Gauss-Newton Method}

{Gauss--Newton} (GN) methods solve \eqref{eq:objective} through making iterative quadratic approximations of the objective function, each of which may be minimized by solving a linear system. At iteration $k$, the (standard) Gauss--Newton method constructs a local quadratic approximation centered at the current iterate $\bm_k$ using a Hessian approximation that neglects second-order terms involving the residuals~\cite{Nocedal}.
The GN approximation to the Hessian of $F$ at the $k$th step is given by
\begin{equation}\label{eq:H-GN}
    \bH^{\rm GN}_{k} = \sum_{i=1}^N \bJ_i(\bm_k)^\top \bJ_i(\bm_k) + \nabla^2 R(\bm_k),
\end{equation}
where 
$\nabla^2 R$ is the Hessian of the regularizer. The first term approximates the Hessian of the data misfit $\Phi(\bm)$ by neglecting second-order derivative terms in the residuals, while the second term provides the exact Hessian contribution from regularization.

The GN step $\bp_k^{\rm GN}$ based on the Hessian approximation \eqref{eq:H-GN}
is defined as the solution of  the quadratic minimization problem
\begin{equation} \label{eq:gn_update_quadratic_rep}
    \bp_k^{\rm GN} = \argmin_{\bp}\, \left(\nabla \Phi(\bm_k) + \nabla R(\bm_k)\right)^\top \bp + \frac{1}{2}\,\bp^\top \bH^{\rm GN}_{k} \bp,
    \end{equation}
which has explicit solution
\begin{equation}
    \bp^{\rm GN}_k =  -\left( \bH_{k}^{\rm GN} \right)^{-1} \nabla F(\bm_k).
    \label{eq:gn_update_explicit}
\end{equation}
The model parameters are then updated as
\(
    \bm_{k+1} = \bm_k + \alpha_k\, \bp_k^{\rm GN},
\)
where $\alpha_k > 0$ is again a step size parameter chosen using linesearch or trust region methods.

In practice, matrix-free variants such as Gauss--Newton--CG are commonly employed to avoid explicitly forming or inverting the Jacobian matrices~\cite{bui2013computational, haber2014computational, petra2014computational}. These methods solve the linear system~\eqref{eq:gn_update_explicit} iteratively using conjugate gradient iterations, requiring only matrix-vector products with $\bH_k^{\rm GN}$. Although these approaches can be effective in many PDE-constrained optimization settings, they become computationally burdensome in applications such as FWI. In such problems, each Jacobian--vector product $\bJ_i(\bm_k) \bv$ or vector--Jacobian product $\bJ_i(\bm_k)^\top \bw$ \emph{requires solving an additional sensitivity or adjoint PDE}. Since multiple such products are required within every conjugate-gradient iteration, the cumulative cost of these PDE solves may dominate the overall runtime. This computational bottleneck motivates the development of Gauss--Newton methods that retain their favorable convergence properties while substantially reducing, or ideally eliminating, the need for additional PDE solves beyond those required for gradient computation.

\section{An Efficient Gradient-Only Gauss-Newton Method (GOGN)}
\label{section:efficient_gn}

To address the computational limitations of traditional Gauss-Newton methods discussed in Section~\ref{section:math_background}, we propose an efficient Gradient--Only Gauss--Newton (GOGN) approach that, in comparison to a Gauss-Newton CG approach, eliminates the need for PDE solves beyond those required to compute the gradient, so that it may be applied with roughly the same computational expense as a first-order algorithm. The key insight is to reformulate the optimization problem in a way that allows us to construct the Gauss-Newton Hessian approximation using only the gradients already computed during each iteration.

\subsection{Problem Reformulation}

We begin by reformulating the objective function~\eqref{eq:objective}. Rather than treating $\Phi$ as a direct sum of norm-squared residuals and using the Jacobian matrices of the operators $G_i$, we focus on the norm of the residuals,
and use their Jacobian matrices. This is the general approach considered in \cite{Nocedal} for nonlinear least-squares with Gauss-Newton. Thus, we define
\[
    \rho_i(\bm) = \sqrt{2\phi_i(\bm)} =  \|\br_i(\bm)\|, 
    \quad i = 1, \ldots, N,
\]
so that $\phi_i(\bm) = \rho_i(\bm)^2/2$, and $\Phi$ is then rewritten as, 
\begin{equation}
    \Phi(\bm) = \sum_{i=1}^N \phi_i(\bm) = \frac{1}{2}\sum_{i=1}^N \rho_i(\bm)^2
    =\frac{1}{2} \,\| \bsrho (\bm)\|^2,
    \label{eq:reformulated_loss}
\end{equation}
where $\bsrho(\bm) = [\rho_1(\bm), \rho_2(\bm), \ldots, \rho_N(\bm)]^\top$ collects the residual norms across the $N$ terms that we are already differentiating separately. 
The key advantage of this reformulation is that we now view $\Phi$ as the objective functional of a nonlinear least-squares problem in terms of the vector-valued function $\bsrho$ that maps from $\RR^p$ to $\RR^N$. This perspective enables us to apply Gauss-Newton methodology to~\eqref{eq:reformulated_loss} in a computationally efficient manner.

\subsection{Constructing the GOGN Jacobian from Available Gradients}

To apply the Gauss-Newton method to the reformulated problem~\eqref{eq:reformulated_loss}, we require the Jacobian of $\bsrho$. The crucial observation is that this Jacobian can be constructed entirely from the gradients $\nabla \phi_i$, which are already available from the gradient computation in~\eqref{eq:gradient}.

Since $\nabla \phi_i(\bm) = \rho_i(\bm) \nabla \rho_i(\bm)$,
we have,
\begin{equation}
    \nabla \rho_i(\bm) = \frac{\nabla \phi_i(\bm)}{\rho_i(\bm)} = \frac{\nabla \phi_i(\bm)}{\sqrt{2 \phi_i(\bm)}}.
    \label{eq:gradient_ri}
\end{equation}
The Jacobian of $\bsrho$ at $\bm$ is therefore given by
\begin{equation}
    \begin{split}
    {\bJ}^{\rm GO}(\bm) 
    &= [\nabla \rho_1(\bm), \nabla \rho_2(\bm), \ldots, \nabla \rho_N(\bm)]^\top 
    \\
    &=  
    \left[\frac{\nabla \phi_1(\bm)}{\sqrt{2 \phi_1(\bm)}} ,
    \frac{\nabla \phi_2(\bm)}{\sqrt{2 \phi_2(\bm)}} ,
    \ldots,
    \frac{\nabla \phi_N(\bm)}{\sqrt{2 \phi_N(\bm)}}
    \right]^\top \in \RR^{N \times p}.
    \end{split}
    \label{eq:jacobian_tilde}
\end{equation}
The advantage of this reformulation  is that constructing $\bJ^{\rm GO}(\bm)$ requires only the gradients $\nabla \phi_i(\bm)$ and the function values $\phi_i(\bm)$, both of which are already computed during the standard gradient evaluation step. \emph{No additional PDE solves are required to form this Jacobian}; in contrast to traditional Gauss-Newton methods where each Jacobian-vector product would necessitate solving sensitivity equations.

\subsection{The Gradient-Only Gauss-Newton Update}

With the GOGN Jacobian at hand, we now derive the GOGN update-step.  The Gauss-Newton approximation to the Hessian of $F=\Phi + R$ at step $k$ is given by
\[
    \bH_{k}^{\rm GO} = {\bJ}^{\rm GO}(\bm_k)^\top {\bJ}^{\rm GO}(\bm_k) + \nabla^2 R(\bm_k),
\]
where the first term approximates the Hessian of the data misfit $\Phi$ and the second term is the exact Hessian of the regularizer.
At iteration $k$, the GOGN step $\bp_k^{\rm GO}$ solves the quadratic minimization problem
\begin{equation}
    \bp_k^{\rm GO} = \argmin_{\bp}  \; \left(\nabla \Phi(\bm_k) + \nabla R(\bm_k)\right)^\top \bp + \frac{1}{2}\,\bp^\top \bH_{k}^{\rm GO} \bp,
    \label{eq:GOGN_update_quadratic_rep}
\end{equation}
which yields the explicit update formula
\begin{equation}
    \bp^{\rm GO}_k =  - \left(\bH_k^{\rm GO}\right)^{-1} \nabla F(\bm_k). 
    \label{eq:GOGN_update_explicit}
\end{equation}
The model parameters are then updated as $\bm_{k+1} = \bm_k + \alpha_k \bp_k^{\rm GO}$, where $\alpha_k > 0$ is a step size determined by linesearch. Importantly, computing this update requires no additional PDE solves beyond those already performed for gradient evaluation, as the matrix ${\bJ}^{\rm GO}(\bm_k)$ is constructed directly from available gradient information. The regularization term $R(\bm_k)$ plays a crucial role in ensuring that $\bH_k^{\rm GO}$ remains positive definite and invertible, which is essential for the method's stability and convergence.

\subsection{Convergence Analysis}

For completeness, we provide the convergence properties of the GOGN method under standard  regularity conditions. The analysis is standard and relies on the well-posedness of the regularizer.

\begin{theorem}
\label{thm:GOGN_convergence}
Let $\cC = \{\bm \: | \: F(\bm) \leq F(\bm_0) \}$ be the sublevel set defined by a starting point  $\bm_0$, and assume $\cC$ is compact. Suppose $F$ has Lipschitz-continuous gradients on an open set containing $\cC$, and the regularizer satisfies
\begin{equation}
    \mu \bI \preceq \nabla^2 R(\bm) \preceq M \bI
    \label{eq:reg_spectral_bounds}
\end{equation}
for all $\bm \in \cC$ with $0 < \mu \leq M < \infty$. 
Consider the GOGN iterates,
\(
    \bm_{k+1} = \bm_k + \alpha_k\, \bp_k^{\rm GO},
\)
where $\bp_k^{\rm GO}$ is given by~\eqref{eq:GOGN_update_explicit}, and $\alpha_k$ is chosen to satisfy the Wolfe conditions~\cite{Nocedal}.
Then 
\begin{equation}
    \lim_{k\to\infty} \|\nabla F(\bm_k)\| = 0.
\end{equation}
\end{theorem}

\begin{proof}
Since 
the gradients $\nabla \phi_i$, $i=1,\ldots,N$, are continuous on $\cC$, so is the GOGN Jacobian $\bJ^{\rm GO}$ defined in~\eqref{eq:jacobian_tilde}, which is therefore  bounded on $\cC$. This implies in turn that for each
$\bm \in \cC$,
the matrix $\bJ^{\rm GO}(\bm)^\top \bJ^{\rm GO}(\bm)$ is positive semi-definite with bounded eigenvalues $\lambda_i(\bm)$: $0\leq \lambda_i(\bm)\leq M_J$, where
\[
M_J = \max_{\bm \in \cC} \|\, \bJ^{\rm GO}(\bm)^\top \bJ^{\rm GO}(\bm) \,\|_2 < \infty
\]

Using the spectral bounds on $\nabla^2 R(\bm)$ from~\eqref{eq:reg_spectral_bounds}, we obtain a bounded spectrum for the regularized GOGN Hessian:
\begin{equation}
    \mu \bI \preceq \bH_{\bm}^{\rm GO} = \bJ^{\rm GO}(\bm)^\top \bJ^{\rm GO}(\bm) + \nabla^2 R(\bm) \preceq (M + M_J) \bI
\end{equation}
for all $\bm \in \cC$. Since $\bH_{\bm}^{\rm GO}$ is positive definite with $\mu > 0$, it is invertible with 
\begin{equation}
    \frac{1}{M + M_J} \bI \preceq \left(\bH_{\bm}^{\rm GO}\right)^{-1} \preceq \frac{1}{\mu} \bI.
\end{equation}
In addition, for any $\bm \in \cC$ with $\nabla F(\bm) \neq \bszero$, we have
\begin{equation}
\nabla F(\bm)^\top \bp^{\rm GO} = - \nabla F(\bm)^\top \left(\bH_{\bm}^{\rm GO}\right)^{-1} \nabla F(\bm)
   \leq - \frac{1}{M + M_J} \| \nabla F(\bm)\|^2  < 0. 
\end{equation}
This shows that $\bp_k^{\rm GO}$ is a descent direction at every iteration. Since $\alpha_k$ satisfies the Wolfe conditions, we have $F(\bm_{k+1}) < F(\bm_k)$, which implies that all iterates remain in the sublevel set $\cC$. To prove the convergence of the gradients, note
that by the Wolfe conditions and the descent property, Zoutendijk's theorem~\cite[Theorem 3.2]{Nocedal} guarantees 
\begin{equation}
    \sum_{k=0}^{\infty} \cos^2(\theta_k) \|\nabla F(\bm_k)\|^2 < \infty,
\end{equation}
where $\theta_k \in (0, \pi)$ is the angle between the steepest descent direction $-\nabla F(\bm_k)$ and the search direction $\bp_k^{\rm GO}$. This implies
\(
    \cos^2(\theta_k) \|\nabla F(\bm_k)\|^2 \to 0,\,\,
\)
and convergence of the gradients follows once we show
 that $\cos(\theta_k)$ is uniformly bounded away from zero. By definition,
\begin{equation}
    \cos(\theta_k) 
    = \frac{- \nabla F(\bm_k)^\top \bp_k^{\rm GO}}{\|\bp_k^{\rm GO}\|\, \|\nabla F(\bm_k)\|}
    = \frac{\nabla F(\bm_k)^\top \left(\bH_k^{\rm GO}\right)^{-1} \nabla F(\bm_k)}{\|\left(\bH_{k}^{\rm GO}\right)^{-1} \nabla F(\bm_k)\| \,\|\nabla F(\bm_k)\|}.
\end{equation}
Using the Cauchy-Schwarz inequality and the spectral bounds on $\left(\bH_{\bm_k}^{\rm GO}\right)^{-1}$, we obtain
\begin{equation}
    \cos(\theta_k) 
    \geq \frac{\frac{1}{M + M_J} \|\nabla F(\bm_k)\|^2}{\frac{1}{\mu} \| \nabla F(\bm_k)\|^2}
    = \frac{\mu}{M+M_J} > 0, \quad k \in \ZZ^+
\end{equation}
Since $\cos(\theta_k) \geq {\mu}/\medmath{(M+M_J)} > 0$ for all $k$, we have $\cos^2(\theta_k) \geq \left(\frac{\mu}{M+M_J}\right)^2 > 0$, which combined with Zoutendijk's condition implies
$\|\nabla F(\bm_k)\| \to 0$. \qed
\end{proof}

This theorem guarantees global convergence of GOGN to a stationary point under mild regularity conditions. 
The lower bound $\mu > 0$ on the regularizer Hessian ensures that the approximate Hessian $\bH_k^{\rm GO}$ remains uniformly positive definite, which is critical for both the descent property and the uniform bound on the cosine of the search angle.


\subsection{Summary and Computational Discussion}
The proposed GOGN method achieves computational efficiency through a systematic reformulation of the optimization problem. By viewing the objective as a least-squares problem in terms of per-source residual norms, we enable the construction of a Jacobian ${\bJ}^{\rm GO}$ directly from the gradients $\nabla \phi_i$ already computed during standard gradient evaluation. This eliminates the need for additional PDE solves that would otherwise be required for traditional Jacobian-vector products in Gauss-Newton methods. The resulting GOGN update can be computed efficiently, particularly when $N \ll p$ as is typical in large-scale inverse problems. This approach provides a favorable compromise between first-order and traditional second-order methods. In particular, it empirically exhibits the fast local convergence properties characteristic of Gauss-Newton methods while maintaining a per-iteration computational cost comparable to that of gradient-based methods. Importantly, the GOGN reformulation \emph{can be applied to general misfit functions $\phi_i$ which are not sums of squares}, as long as we have access to the gradients $\nabla \phi_i,$ making it applicable in more cases than traditional Gauss-Newton methods. As long as the data misfit $\Phi$ is a sum of differentiable functions $\phi_i,$ the GOGN approach can rewrite $\Phi$ as a sum of squares and exploit this structure using only the gradients $\nabla \phi_i.$

We also note that in all our numerical experiments we 
use in the regularization function $R ( \bm ) = \medmath{\frac 12} \| \bD ( \bm - \bm_0 ) \|^2$ a discretization $\bD$ of $\nu \rI - {\rm \Delta}$, where ${\rm \Delta}$ is the Laplace operator (known to be negative semi-definite), and $\nu >0$ controls the length scale of heterogeneity promoted by the regularizer. The choice to use an operator of this form is a common smoothness regularizer in FWI \cite{Bui-Thanh2013-zg, Bessel}. Furthermore, our choice of regularization operator has an invertible Hessian $\bD^\top \bD$ and thus fits within the theoretical assumptions in~\eqref{eq:reg_spectral_bounds} (as opposed to the more standard choice to use a discretization of ${\rm \Delta},$ which leads to a singular Hessian).

\section{Related Work}
\label{section:related_work}

The work most closely related to ours is the source-subspace method of~\cite{SSTape}, which projects the data of a nonlinear least-squares problem onto a low-dimensional subspace. Our approach extends this framework in two key directions. First, our method applies to general sums of misfit functions of the form~\eqref{eq:emp-min}, whereas the source-subspace method is restricted to nonlinear least-squares formulations. Second, our formulation accommodates a broader class of regularization strategies, rather than being tied to the structure inherent in least-squares problems. The source-subspace strategy has demonstrated strong empirical success in large-scale seismic imaging, including regional-scale full waveform inversion (FWI) of the Southern Californian crust by~\cite{tape2010seismic}, which shows its effectiveness in regimes where the number of data samples satisfies $N \ll p$.

More broadly, our work can be viewed as a subspace-method approach for large-scale inverse problems. Subspace methods mitigate the computational burden of forming or inverting full Hessian matrices by restricting optimization updates to carefully constructed low-dimensional subspaces. For example, Kennett and Williamson~\cite{Kennett} proposed projecting the Newton Hessian onto subspaces in model space, with basis vectors chosen to capture dominant features of the inverse problem. These ideas were later extended to multi-parameter joint inversion settings~\cite{SambridgeSS}. Multilevel finite element and finite volume strategies for direct current resistivity inversion~\cite{fung2019multiscale} can also be interpreted within this broader subspace-reduction paradigm.

The Newton–CG method, which has been widely used in PDE-constrained optimization~\cite{haber2014computational, petra2014computational, mang2017lagrangian, Epanomeritakis_2008, Castellanos2015-ma, operto}, can similarly be viewed as implicitly projecting the Hessian onto a Krylov subspace. In practice, constructing and exploring such subspaces requires repeated Jacobian–vector and vector–Jacobian products, each corresponding to additional PDE solves. Although conjugate gradient is more economical than explicitly forming and projecting the Hessian, as in~\cite{Kennett}, the total cost still scales with the dimension of the Krylov subspace.

More recently, randomized sketching techniques have been proposed to construct low-rank Hessian approximations~\cite{gower2019rsn}, providing an alternative mechanism for dimensionality reduction. However, these model-space projection approaches generally require a number of additional PDE solves that grows linearly with the subspace dimension. In contrast, our data-space projection method does not require any additional PDE solves beyond those already needed for gradient computation, yielding a fundamentally different and more economical scaling behavior.



\section{Numerical Experiments}
We demonstrate the effectiveness of GOGN on FWI examples. Our experiments are run using Deepwave~\cite{richardson_alan_2023} on a 2021 MacBook Pro with an Apple M1 Pro Chip and 16 GB of RAM.

\subsection{Experimental Setup}
We consider an acoustic FWI problem associated with a 2D scalar wave equation 
\begin{equation}\label{eq:Weq}
   {\rm \Delta} u_i - {\rm \medmath{\frac{1}{c^2} 
   \frac{\partial^2}{\partial t^2}}\,}u_i = 
   f_i, \quad i=1,\ldots,N,
\end{equation}
on the square domain $\Omega = [0{\rm km}, 480{\rm km}]^2$ with PML (Perfectly Matched Layer) absorbing boundary conditions. Here, $u_i$ is the $i^{th}$ state variable (or wavefield), $c$ is the inhomogeneous wave speed parameter (to be reconstructed), and $f_i ( \bx, t ) = 
s ( t ) \delta ( \bx - \bx_i^{\rm s} )$ is the $i^{th}$ source-term, where the source-time function $s ( t )$ is a Ricker wavelet with a frequency of $.1 \rm{Hz}$, $\delta$ is the 2-dimensional Dirac delta distribution, and $\bx_i^{\rm s}$ is the location of the $i^{th}$ source. 

To compute the wave-equation solutions, we use the finite-difference scheme implemented in the package Deepwave \cite{richardson_alan_2023}, which approximately solves \eqref{eq:Weq} in rectangular domains. Considering perturbations of the inhomogeneous P-wave speed $c$ from a reference speed $c_0 = 3000 m/s,$ we discretize the dimensionless parameter ${\delta c}/{c_0}$ onto a $200 \times 200$ grid, and represent it by a \emph{model parameter vector} $\bm \in \RR^p,$ with $p = 40,000.$ We record the output at $200$ time steps with $\Delta t =1s,$ though it should be noted that this is not necessarily the time step of the numerical solver, as Deepwave will take finer time steps when necessary to satisfy the CFL condition \cite{richardson_alan_2023}. We consider wave speed perturbations of up to $5\%,$ which gives a minimum speed of $2850m/s$ and a minimum of at least $11$ cells per wavelength, thus we can expect our simulations to remain numerically stable.

The mapping from a proposed model parameter vector to (preprocessed) synthetic seismograms at observed times $0, \Delta t,  \dots, (n_t - 1) \Delta t,$ for the $i$th event defines the $i$th parameter-to-observable mapping $G_i: \RR^p \to \RR^n.$ We can thus write $\bd_i = G_i ( \bm ) + \bseps_i,$ where $\bseps_i$ is the previously mentioned observational noise. For the purpose of these examples, we generate this observational noise as follows: Letting $\rF$ and $\rF^{-1}$ denote the FFT and its inverse applied over time, and letting $\bz_i \sim \cN_n ( \bszero, \bI )$ be a standard Gaussian vector, we set
$$
\bseps_i = \sigma\, {\rm Re} (
\rF^{-1} (
\bz_i \odot \rF ( G_i ( \bm ) ))),
$$
where $\sigma > 0$ is the noise level, and $\odot$ is the Hadamard product. This ensures that the noise concentrates on the same frequency bands as our observations, so that we avoid a total inverse crime; however, it should be noted that we generate the data for these inversions using the same computational grid, when it would be more appropriate to use a different computational grid to generate data that we use to invert it. We also employ a receiver-weighting approach common in the FWI community \cite{Ruan2019}, which for brevity we leave in Appendix \ref{app:additional_experimental_details}.

We regularize our problem using a smoothing Tikhonov-based regularizer of the form 
\begin{equation}
    R(\bm) = \medmath{\frac{1}{2}}\, \|\bD(\bm - \bm_0)\|^2,
\end{equation}
where $\bD$ is a finite difference discretization of the smoothing operator $\lambda ( \nu I - \Delta )$.
The effect of this regularization is a balance of smoothing and limiting the perturbation distance from $\bm_0$, and has  been useful in FWI~\cite{Bui-Thanh2013-zg, Bessel}. Furthermore, this regularizer has an invertible Hessian $\bD^\top \bD$ whose inverse may be used as smoothing operator, enabling it to both fit into our framework and also to be utilized in standard practice FWI methodology, which applies such a smoothing operator to updates as a form of regularization \cite{Modrak}.

Finally, we test our algorithms on the reconstruction of a low velocity anomaly shaped like a smiley face under two source configurations. Sources and receivers sampled from a uniform distribution give an ideal configuration (as they lead to more accurate recoveries of the target model); however, source-receiver configurations for real-world regional and global scale seismic imaging are constrained by the distribution of receivers, which is subject to geographic and budgetary constraints, and the distribution of earthquakes along fault lines. For this reason, we also consider a realistic source coverage that mimics this by including a region with a high density of receivers, similar to the west coast of the US, and a region with a low density of receivers, similar to the Pacific Ocean. To test the effect of a increased number of sources, we have also tried adding more sources to these configurations by randomly perturbing the source locations to obtain two more configurations with additional sources. All configurations are shown in Figure~\ref{fig:source_setup}, along with the target model.

\begin{figure}[t]
    \centering
    \begin{tabular}{cc}
        \includegraphics[width=0.35\linewidth]{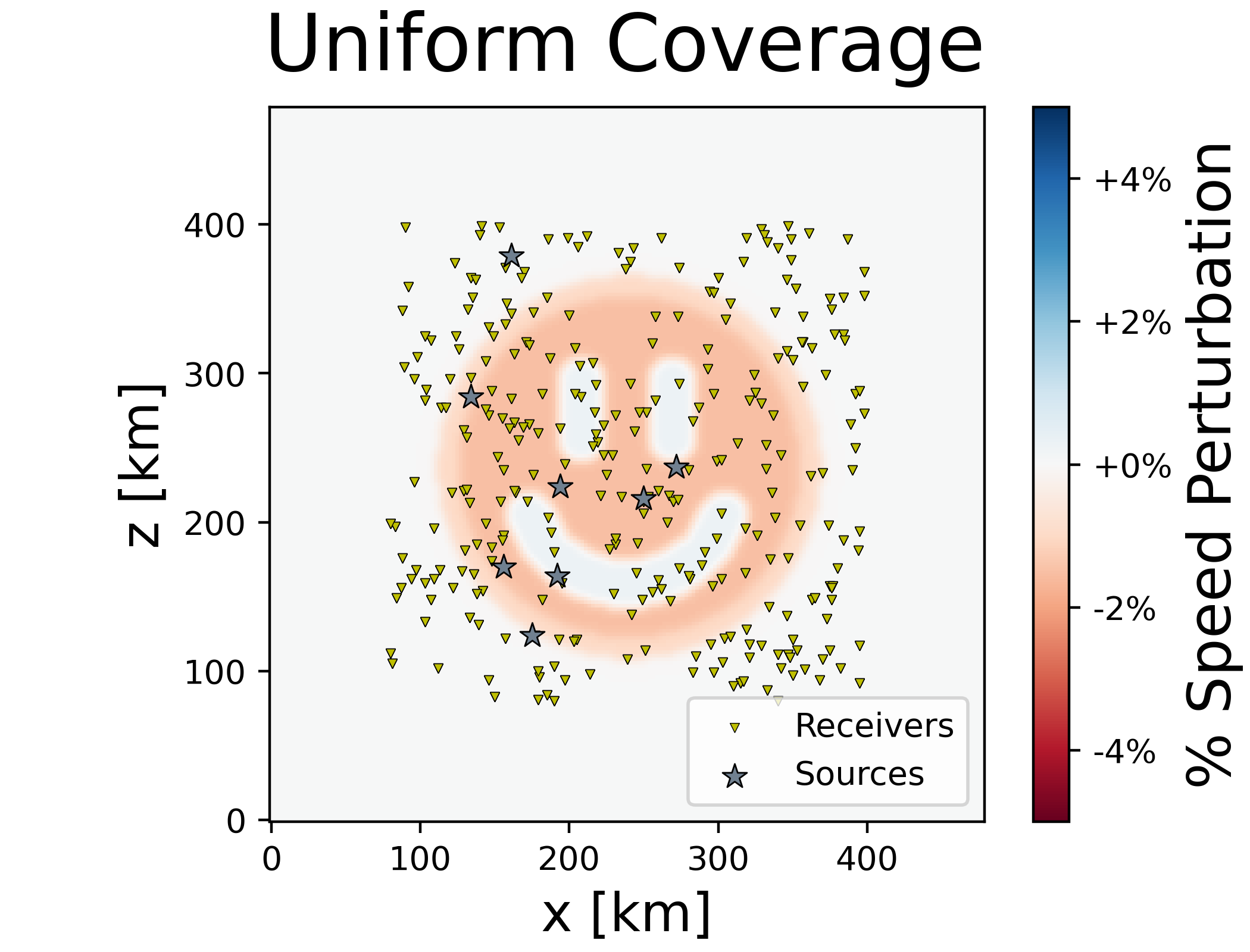}
        &
        \includegraphics[width=0.35\linewidth]{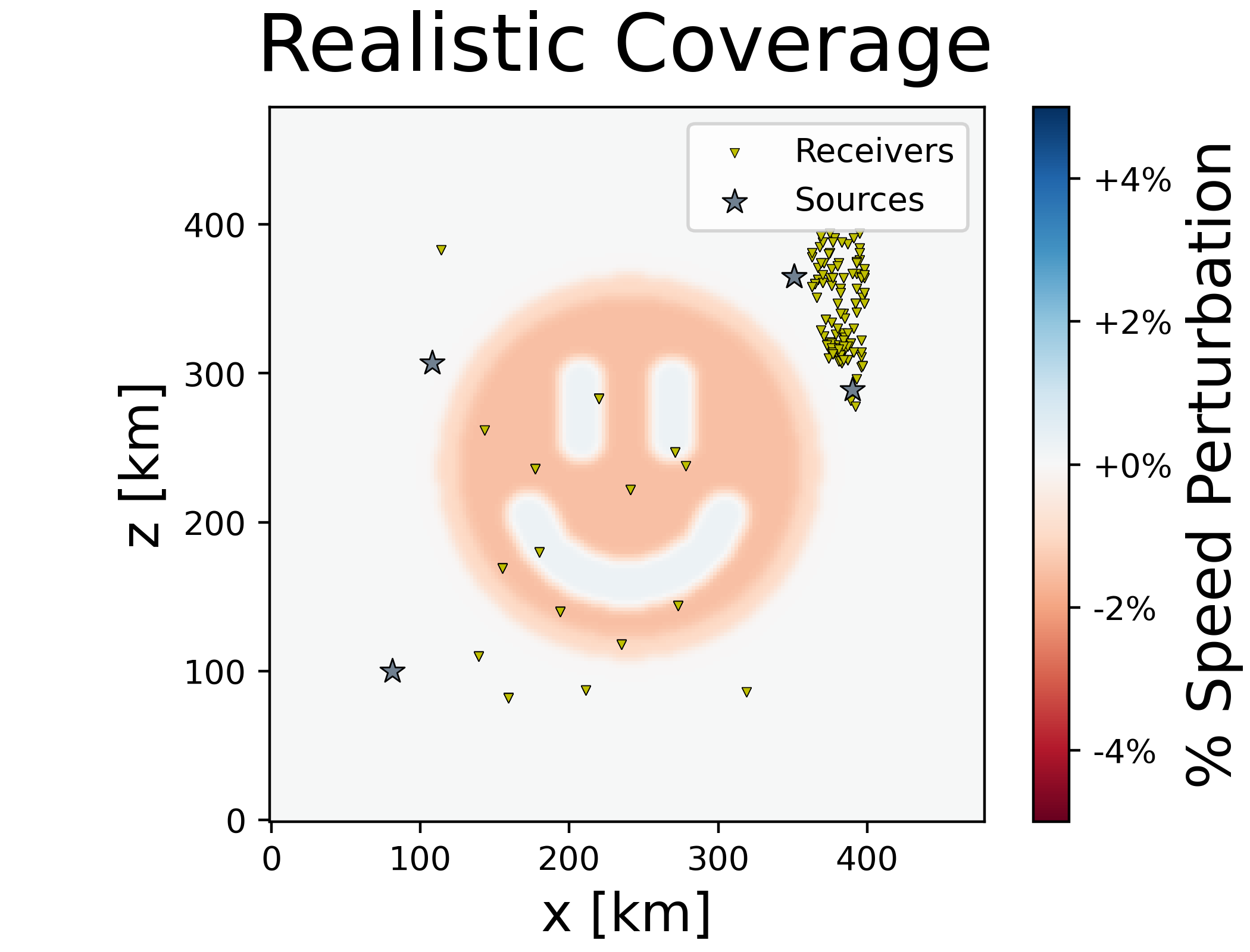} 
        \\
        \includegraphics[width=0.35\linewidth]{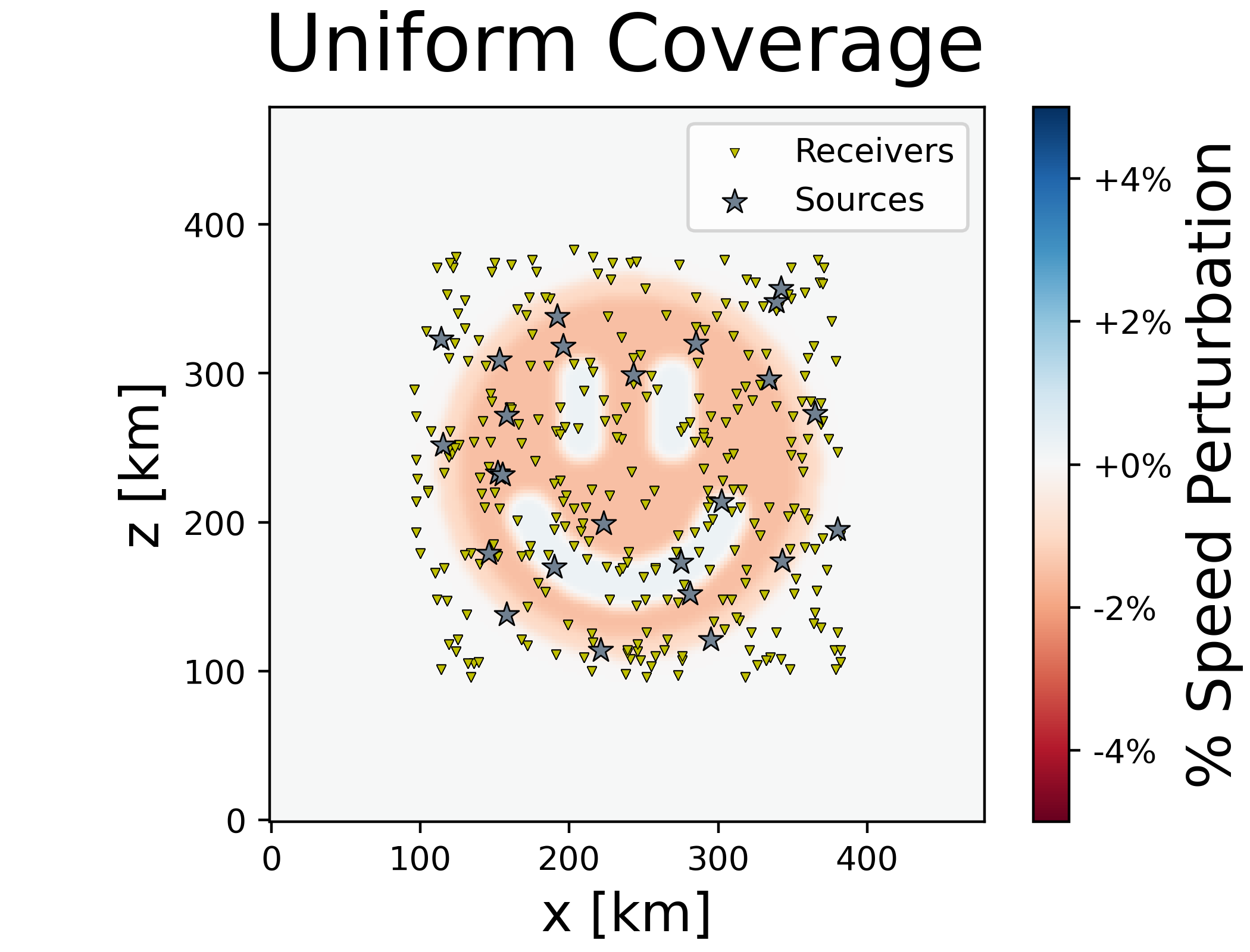}
        &
        \includegraphics[width=0.35\linewidth]{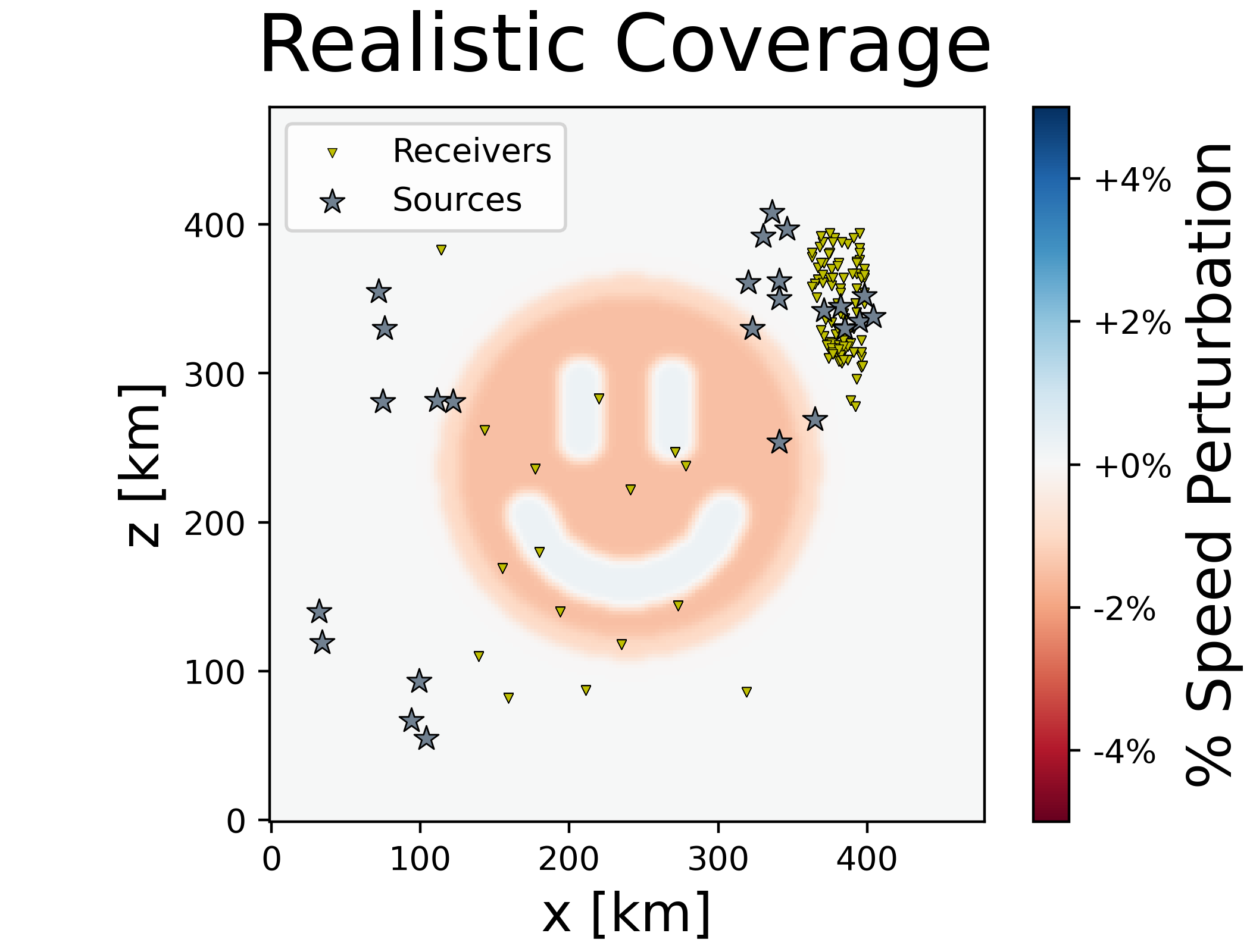} 
    \end{tabular}
    \caption{Our four source-receiver configurations displayed over the speed perturbation ${\delta c}/{c_0}$ for our target model 
    - on the top left, we sample $8$ sources and $300$ receivers from a uniform distribution over the square $[100{\rm km}, 400{\rm km}]^2,$ and on the top right, we transform $5$ source and $181$ receiver locations from around the  Pacific ocean to be contained in our domain, and configurations in the bottom row comprising versions of the top row with additional sources added (so that each has $25$ sources)}
    \label{fig:source_setup}
\end{figure}

\begin{figure}[t]
    \centering
    \begin{tabular}{ccc}
        \multicolumn{3}{c}{Uniform Coverage, $\sigma = 0.1$}
        \\
        \includegraphics[width=0.3\linewidth]{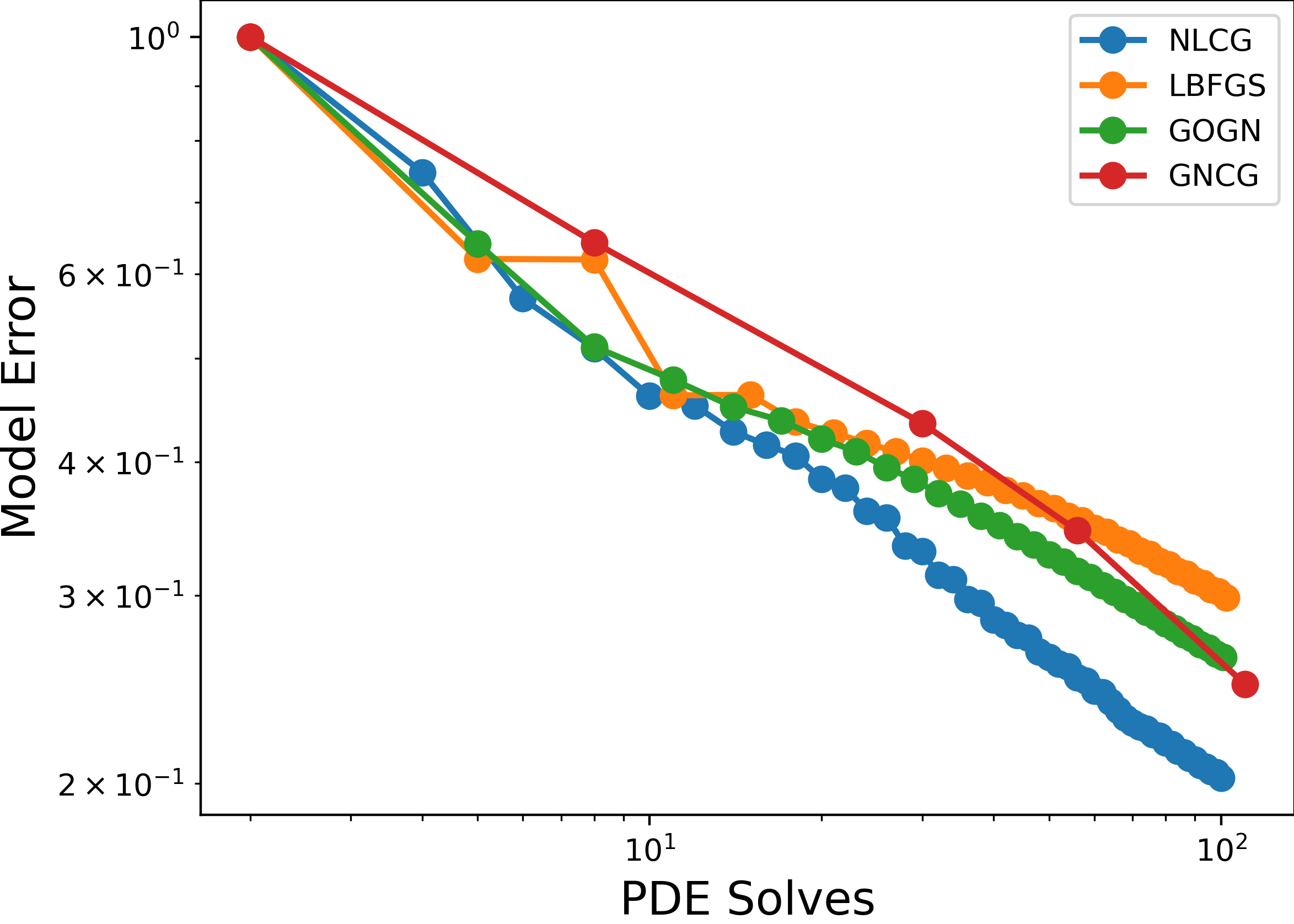}
        &
        \includegraphics[width=0.3\linewidth]{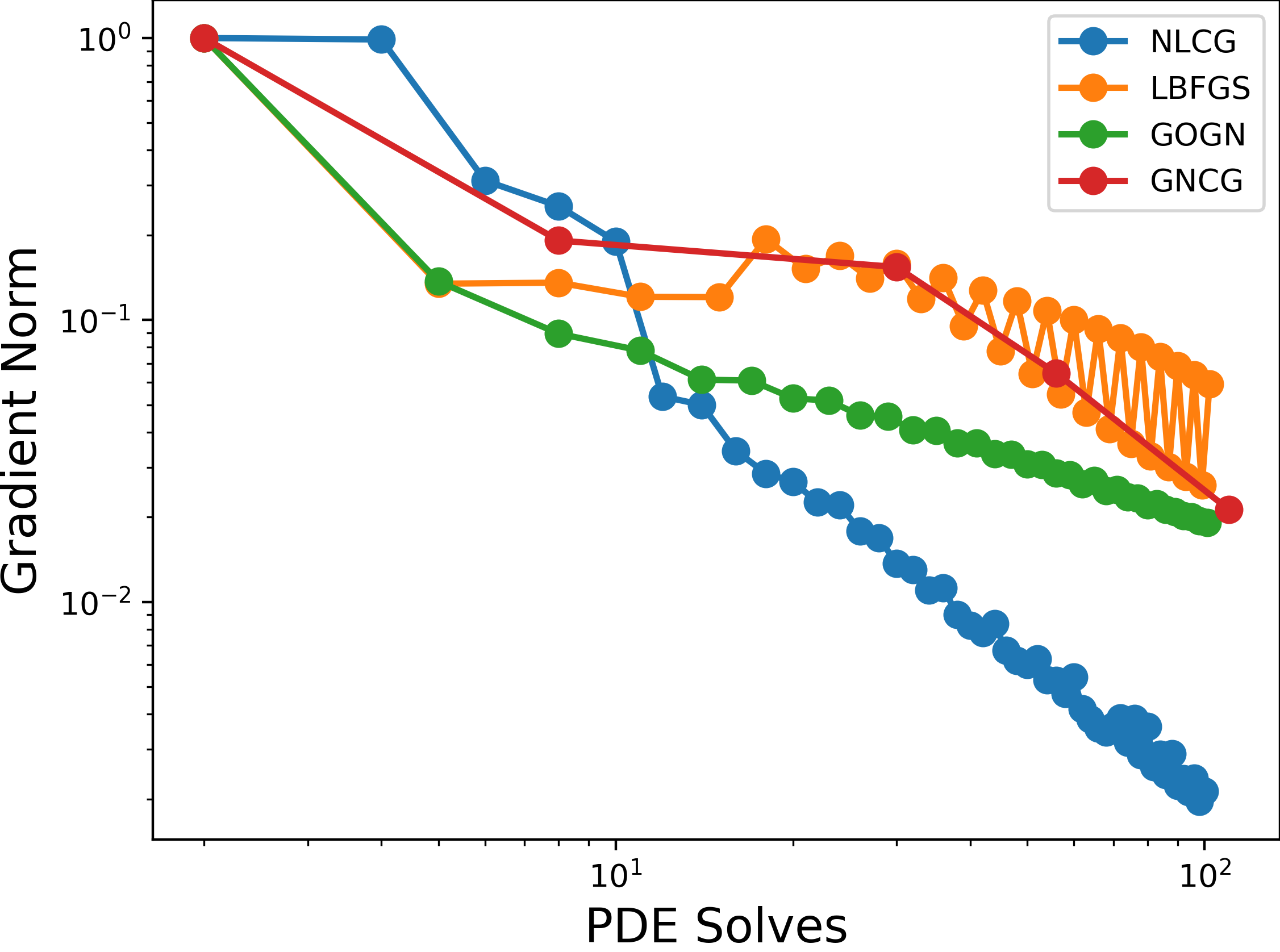}
        &
        \includegraphics[width=0.3\linewidth]{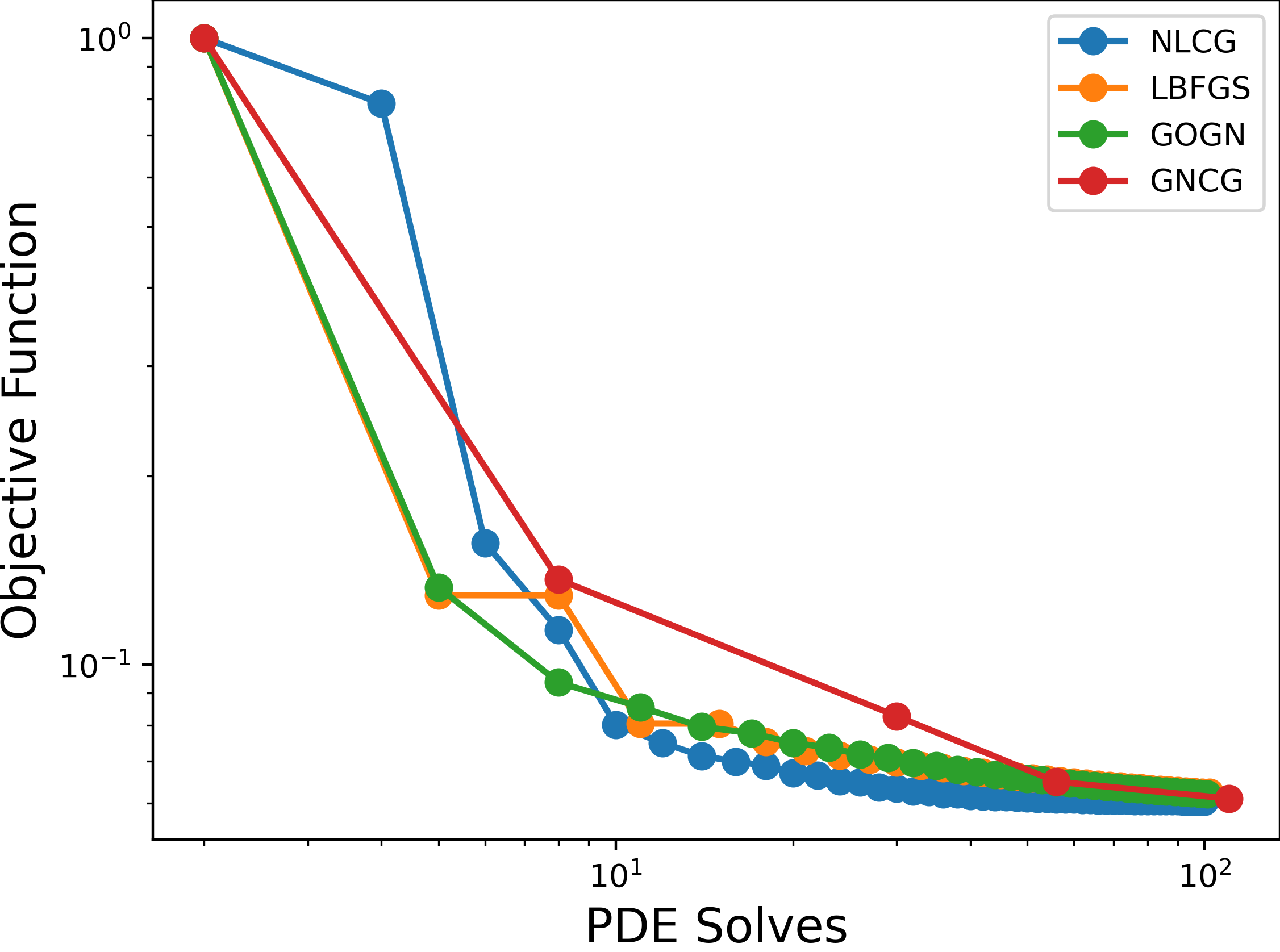}   
        \\
        \multicolumn{3}{c}{Realistic Coverage, $\sigma = 0.1$ }
        \\
        \includegraphics[width=0.3\linewidth]{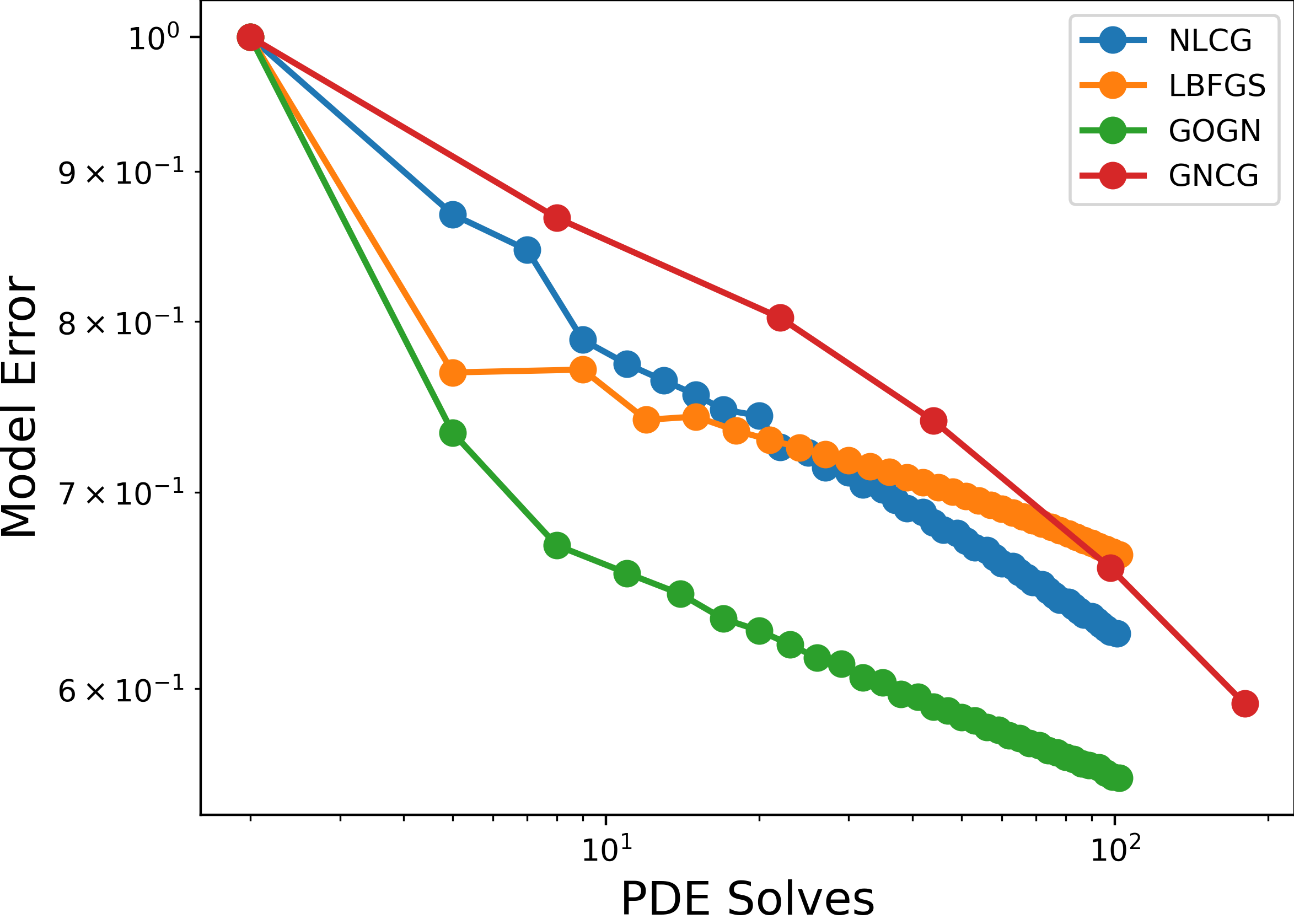}
        &
        \includegraphics[width=0.3\linewidth]{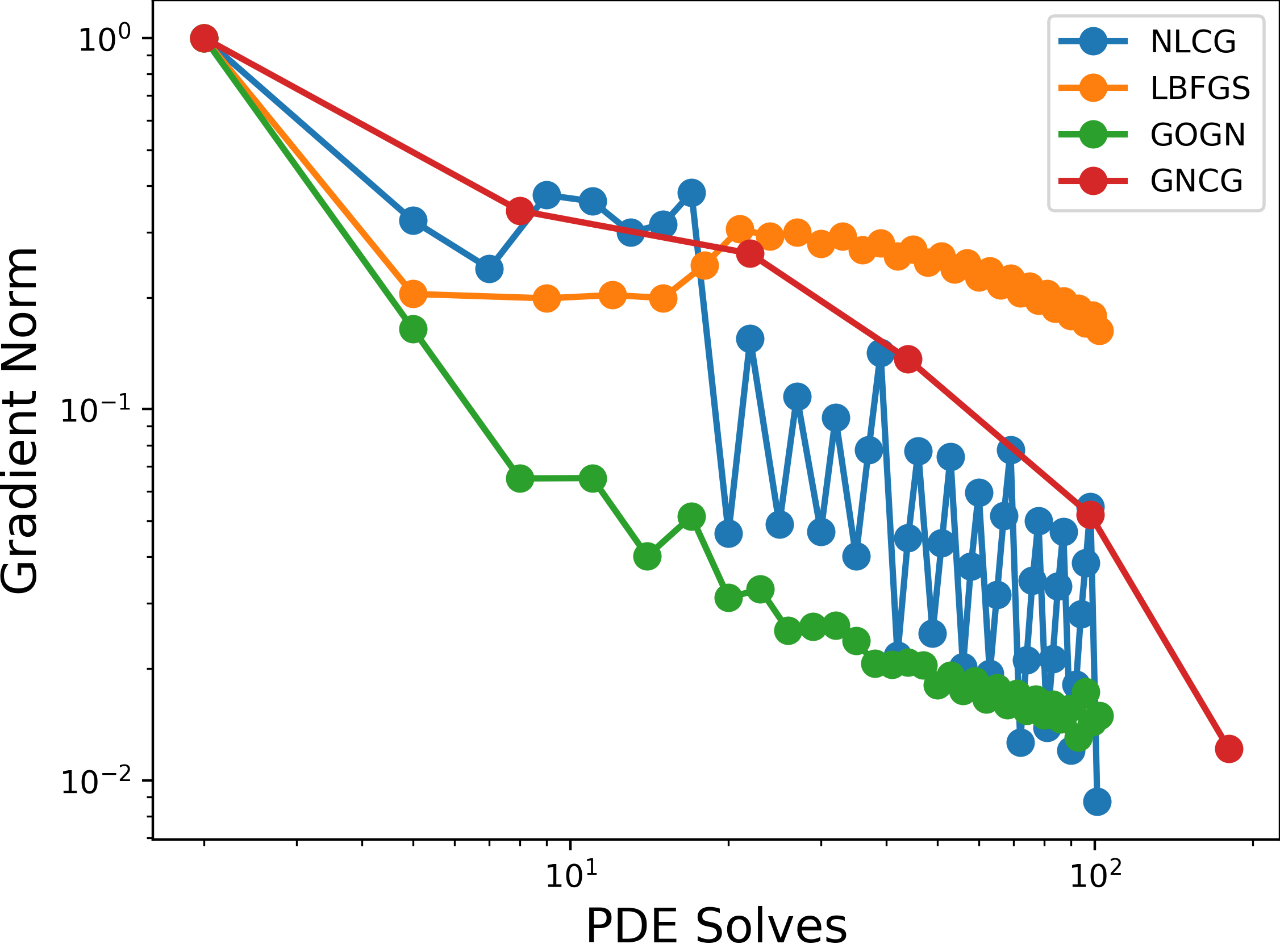}
        &
        \includegraphics[width=0.3\linewidth]{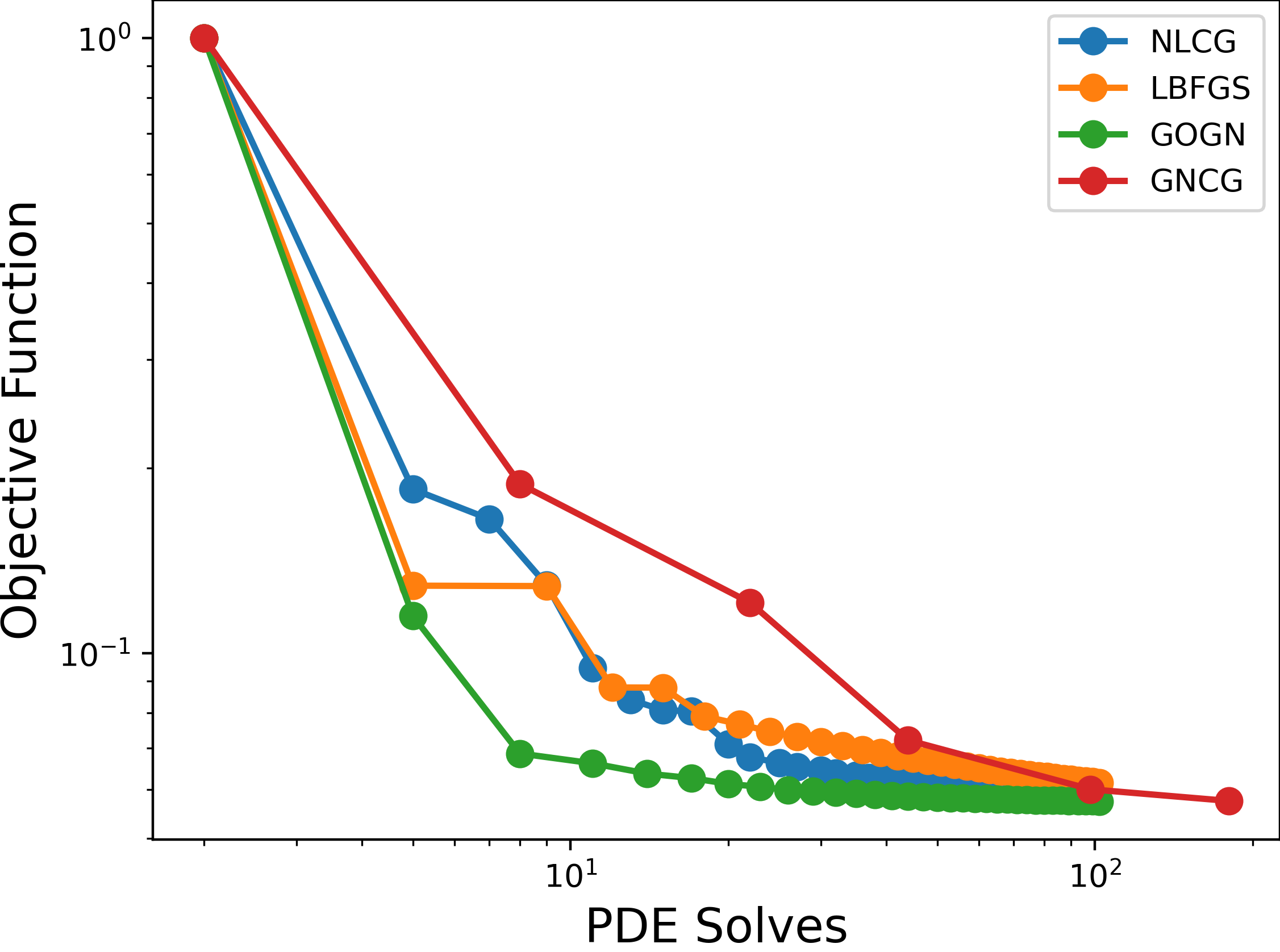}       
    \end{tabular}         
    \caption{Convergence plots for uniformly distributed receiver coverage (top row) and realistic coverage (bottom row) for nonlinear conjugate gradient (NLCG), limited-memory BFGS (LBFGS), gradient-only Gauss-Newton (GOGN) and Gauss-Newton CG (GNCG), with X-axes across all images denoting number of PDE solves during optimization, and Y-axes denoting model error (left), gradient norm (middle), and objective function values (right) for $8$ and $5$ sources, respectively, at a noise level of $\sigma = 0.1$ }
    \label{fig:convergence_plots_5sources}
\end{figure}
\begin{figure}[t]
    \centering
\includegraphics[width=0.4\linewidth]{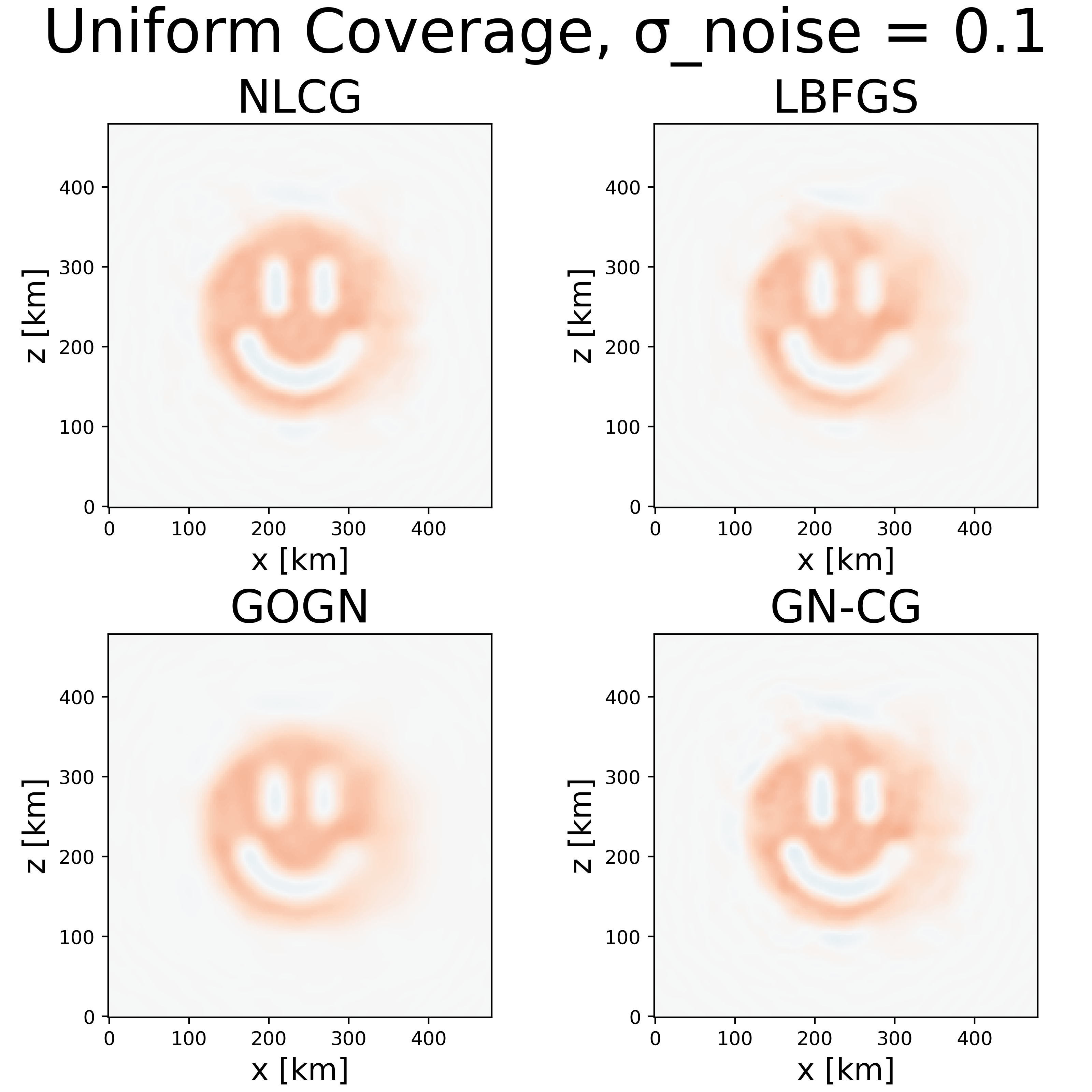}   
\includegraphics[width=0.4\linewidth]{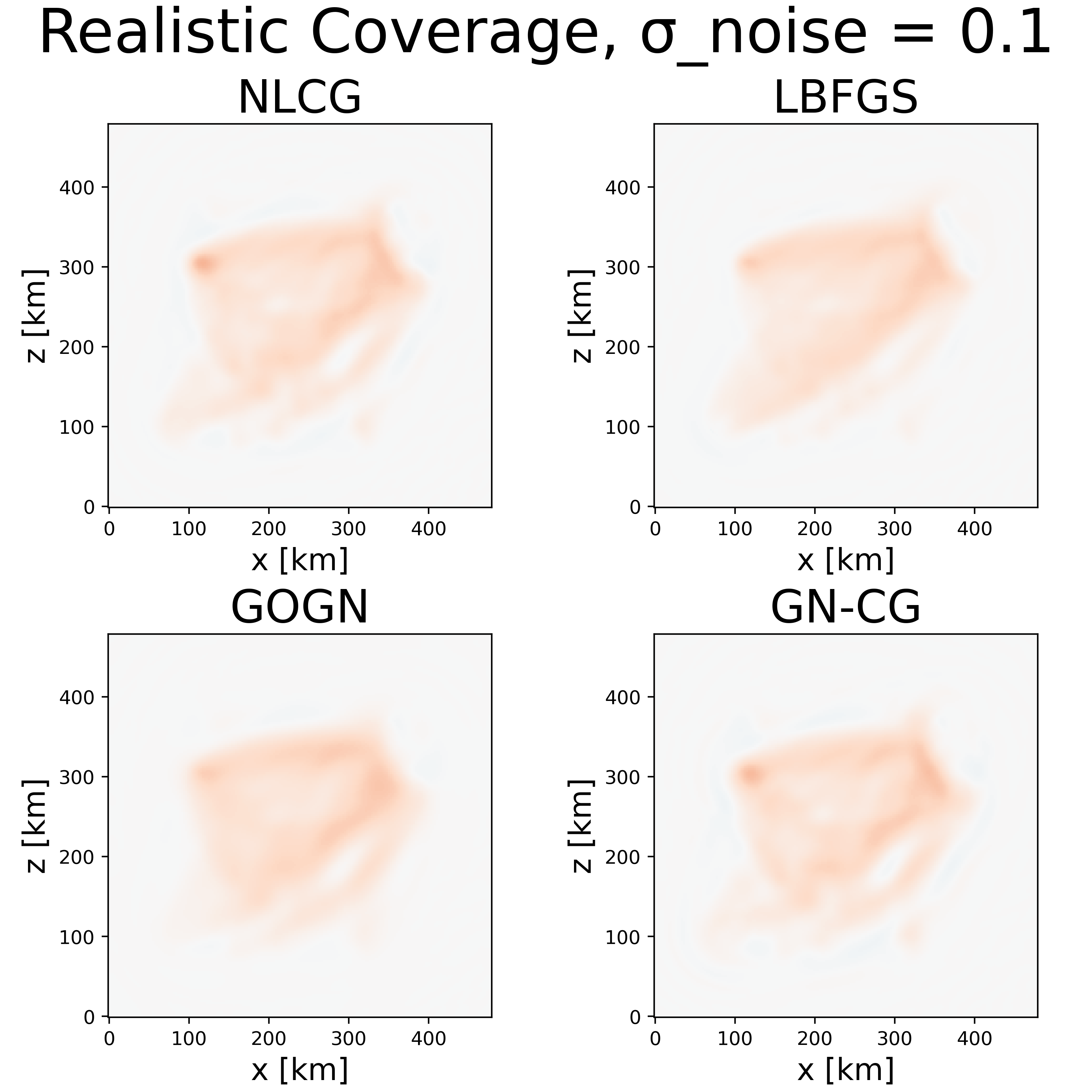} 
\includegraphics[width=.102\linewidth, trim=0cm 4cm 0cm 0cm, clip]{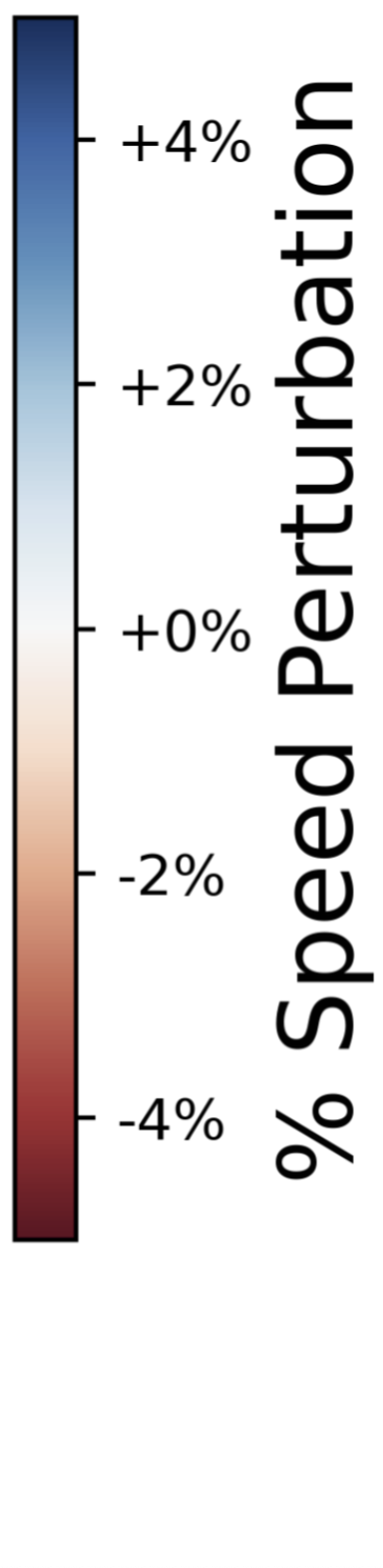}    
    \caption{Model reconstructions for noise level $0.1$ based on the experiments in Figure~\ref{fig:convergence_plots_5sources} }
    \label{fig:reconstructions_5sources}
\end{figure}
\subsection{Algorithmic Setup}
\label{subsec:alg_setup}
We compare the performance of the current ``best practice'' optimization algorithms for FWI, including  Nonlinear Conjugate Gradient (NLCG) and limited-memory BFGS (LBFGS) \cite{Modrak}, with GOGN and the Gauss-Newton-CG method developed in \cite{Epanomeritakis_2008} (GNCG). 
As our work is mainly measured in terms of number of PDE-solves (forward and adjoint), we stop our optimizations at the first iteration exceeding the computational cost of $100$ PDE solves. For each iteration we set a maximum of 10 linsearch iterations. We start with a maximum step size of $1$ for GNCG, and $.05 / \| \bp_k \|_\infty$ for other algorithms to keep the maximum magnitude of perturbations introduced in each iteration at or below $5\%$. For up to $5$ subsequent step lengths, we use quadratic interpolation based on the directional derivative and objective value at the current iterate, as well as the objective value at the last attempted step length. If this does not work we switch to Armijo backtracking. We accept the first step length that decreases the objective. 

For each algorithm, we employ the preconditioning (or regularization) strategies that yield the best performance in our experiments (and that have been found successful in the literature). 
For NLCG and LBFGS, it is common to incorporate a diagonal preconditioner approximating the diagonals of the Hessian at the initial iterate $\bm_0$, where we consider a diagonal approximation to the data-misfit Hessian of the form $\bH_0^{\rm diag} = {\rm diag} ( \nabla^2 \Phi ( \bm_0 ) \bsone),$ where $\bsone$ is a vector of ones (with safeguard to avoid negative values), and precondition our NLCG iterates with the Hessian approximation $\bH_0^{\rm diag} + \bD^\top \bD.$ This is similar to the smoothing considered in \cite{Bessel}, which was found successful for FWI, and will apply more smoothing in regions with less data coverage (similar to the smoothing strategy used in \cite{GladM15}), while at the same time correcting for imbalanced data coverage. 
Thus, in one step, the effect is similar to the separate application of preconditioning and smoothing developed in \cite{Modrak}. 

For LBFGS, we use the same Hessian approximation used in NLCG to initialize our inverse Hessian approximation, but must subsequently apply smoothing with the operator $(\bD^\top \bD)^{-1},$ following usual practice in FWI. Note we can compute a factorization of $\bD$ in advance in order to apply the smoothing operator efficiently.
As the GOGN update is given by the solution to $({\bJ^{\rm GO}}^\top \bJ^{\rm GO} + \bD^\top \bD) \bp_k = \nabla F ( \bm_k ),$ it is already smooth, and thus does not require any preconditioning or additional smoothing-- it is worth mentioning that we solve the this system with the formula
$$
\bp_k = 
(\bD^\top \bD)^{-1} {\bJ^{\rm GO}}^\top (
\bI +
\bJ^{\rm GO} (\bD^\top \bD)^{-1} {\bJ^{\rm GO}}^\top
)^{-1} (\bJ^{\rm GO} (\bm_k - \bm_0) - \bsrho_k) - ( \bm_k - \bm_0 ),
$$
which requires the computation of $(\bD^\top \bD)^{-1} {\bJ^{\rm GO}}^\top$ using $N$ products with $(\bD^\top \bD)^{-1}$ and the inversion of an $N \times N$ matrix $\bI +
\bJ^{\rm GO} (\bD^\top \bD)^{-1} {\bJ^{\rm GO}}^\top.$ For GNCG, we use the strategy developed by \cite{Epanomeritakis_2008}, which uses a history of products with previous iterations' Gauss-Newton Hessians to form a quasi-Newton preconditioner \cite{qnprec}. We initialize this preconditioner with $300$ Richardson iterations applied to $\bH_0^{\rm diag} + \bD^\top \bD$, as it is required that the preconditioner used for the CG solution be a linear operator (as opposed to initializing with CG iterations, which converge faster than Richardson iterations but act as a nonlinear operator, and thus can not be used).

\begin{figure}[t]
    \centering
    \begin{tabular}{ccc}
        \multicolumn{3}{c}{Uniform Coverage, $\sigma = 0.1$}
        \\
        \includegraphics[width=0.3\linewidth]{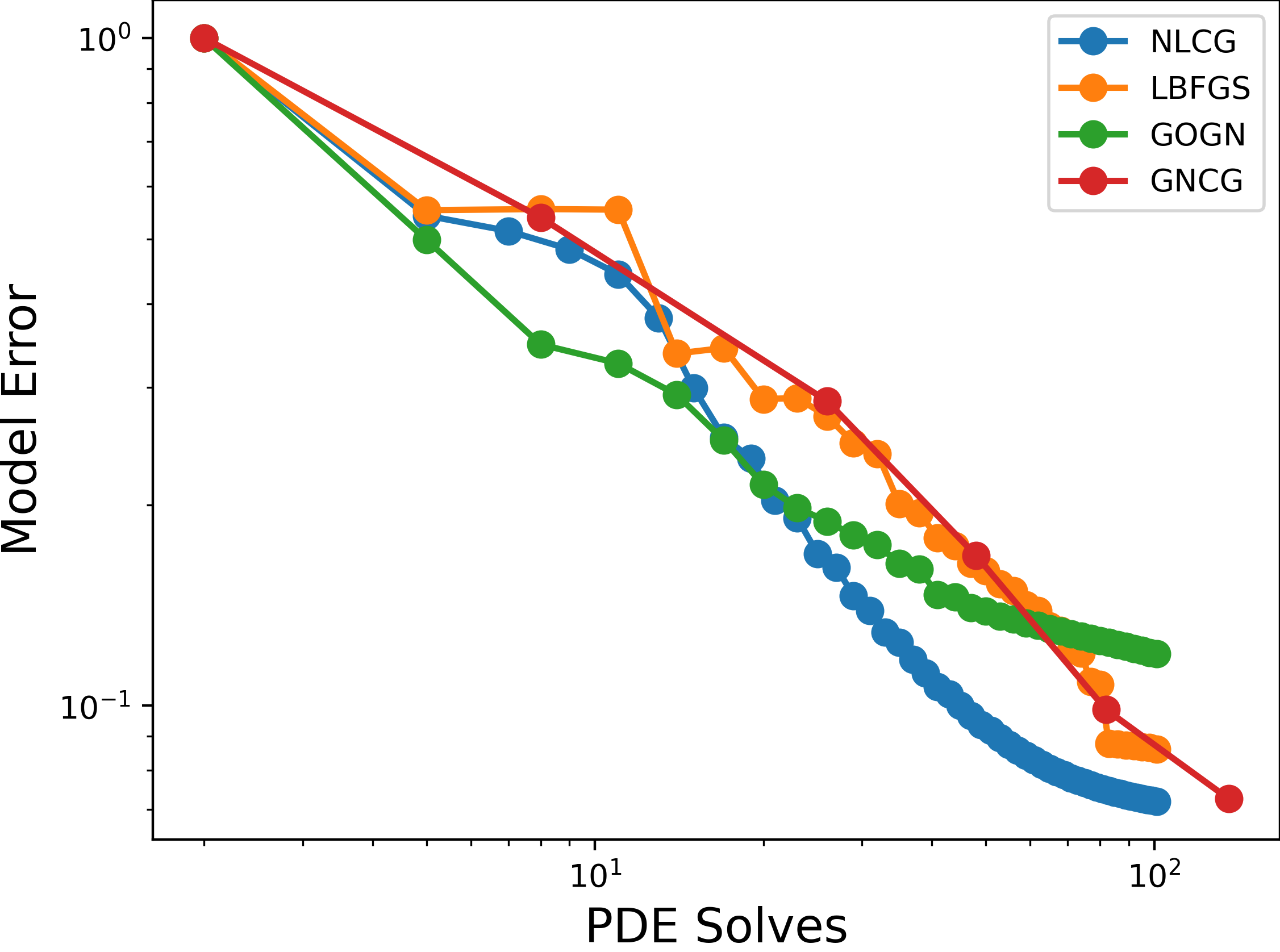}
        &
        \includegraphics[width=0.3\linewidth]{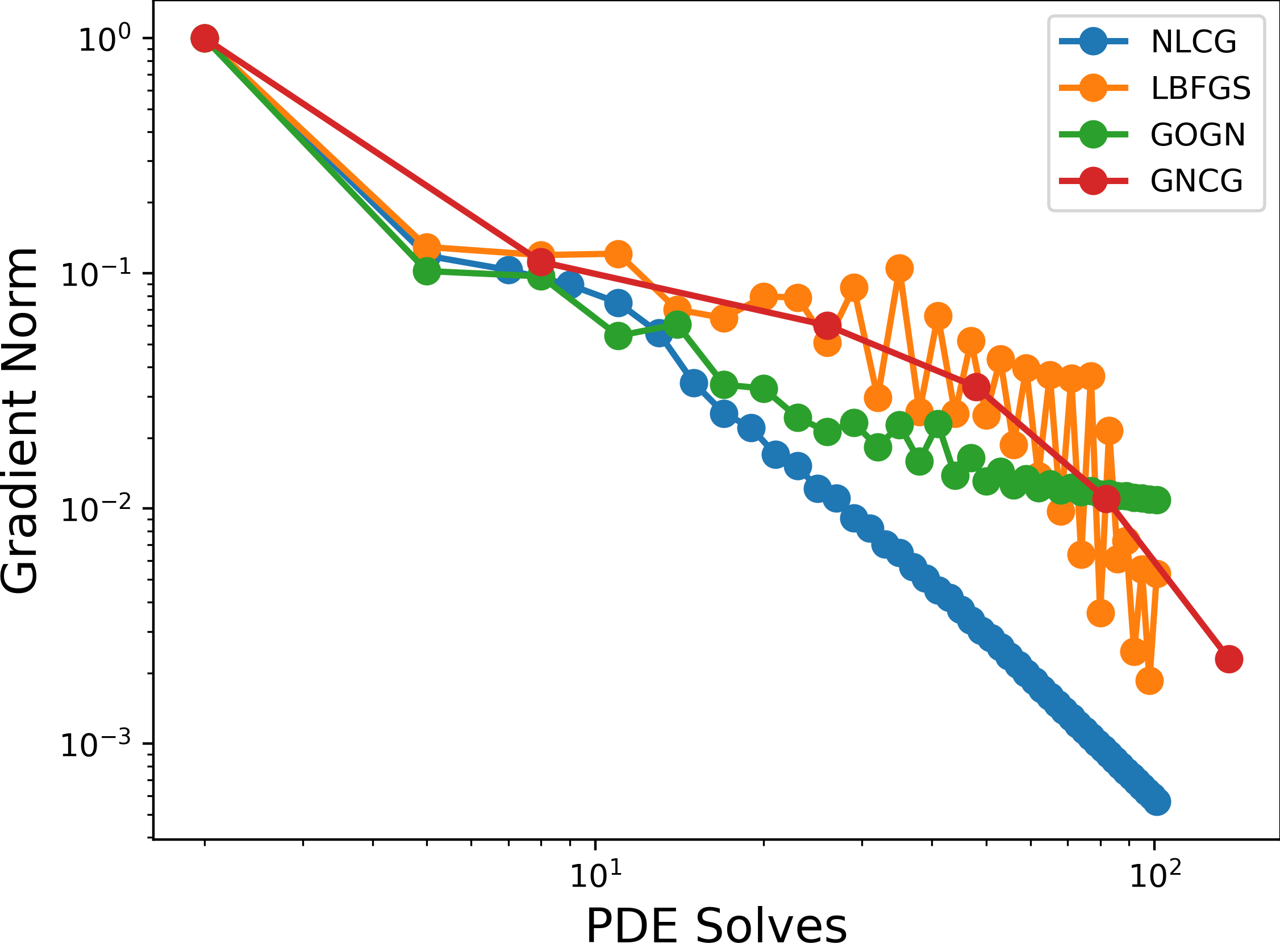}
        &
        \includegraphics[width=0.3\linewidth]{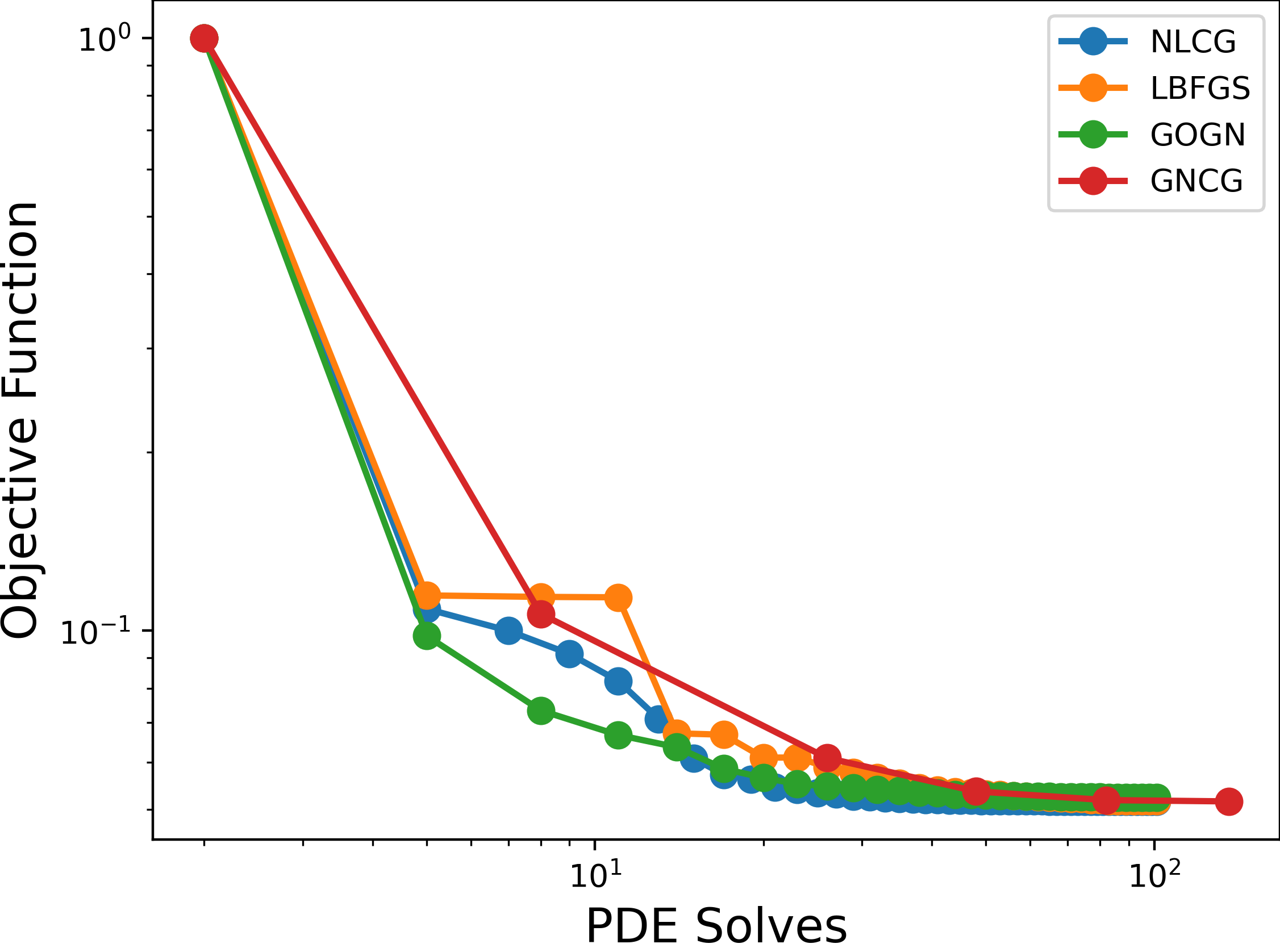}   
        \\
        \multicolumn{3}{c}{Realistic Coverage, $\sigma = 0.1$ }
        \\
        \includegraphics[width=0.3\linewidth]{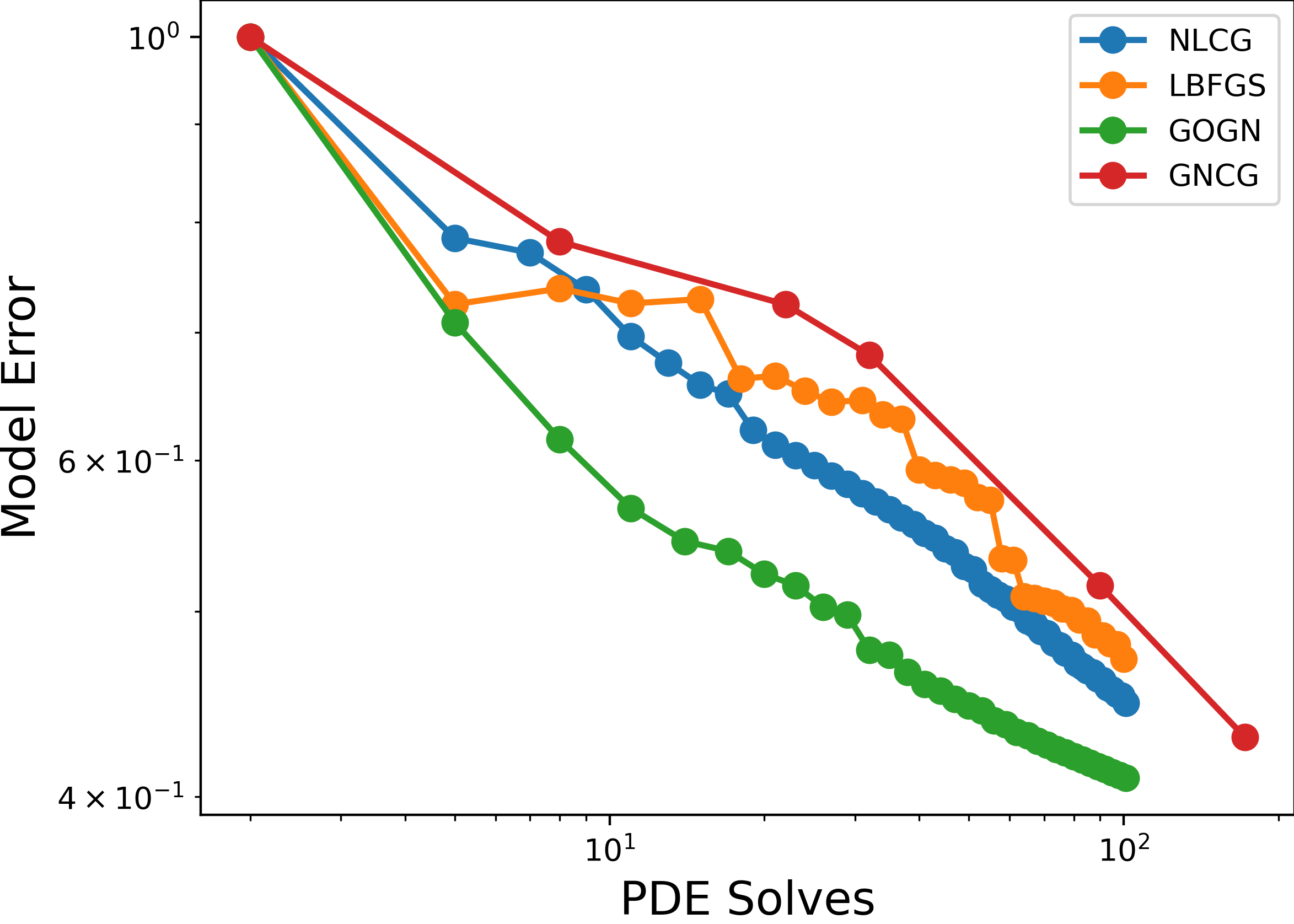}
        &
        \includegraphics[width=0.3\linewidth]{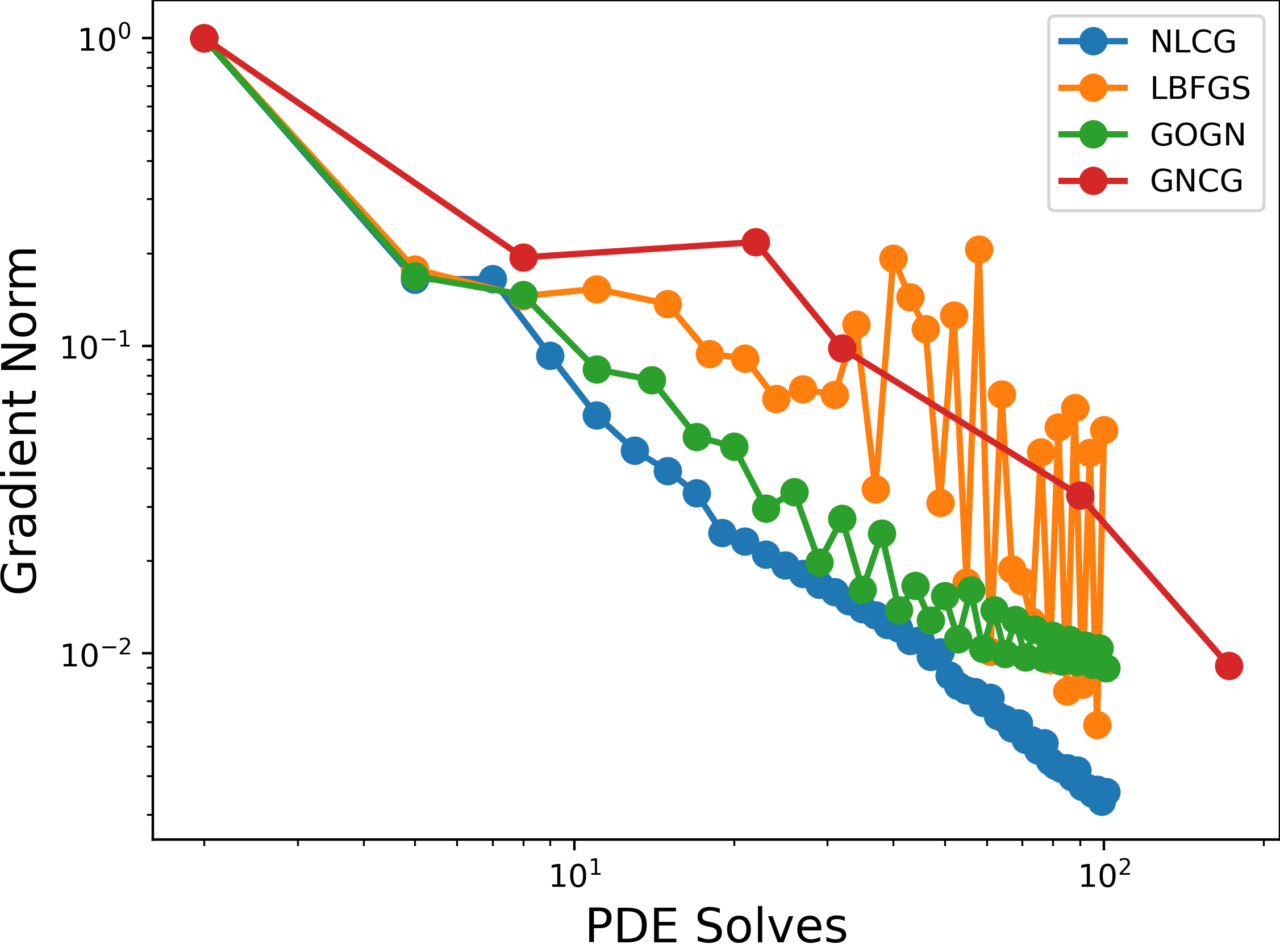}
        &
        \includegraphics[width=0.3\linewidth]{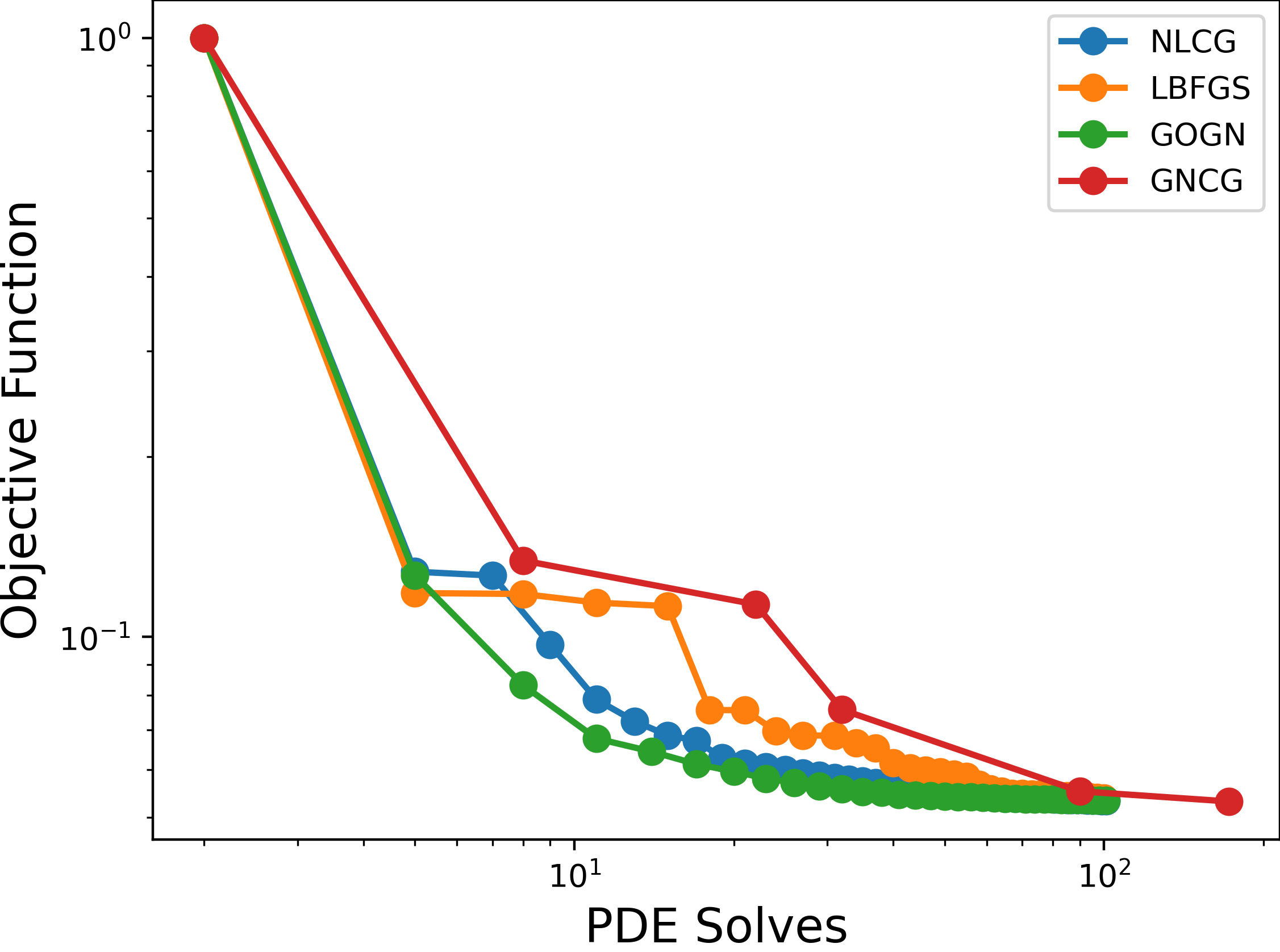}       
    \end{tabular} 
    \caption{Convergence plots for uniformly distributed receiver coverage (top row) and realistic coverage (bottom row), with X-axes across all images denoting number of PDE solves during optimization, and Y-axes denoting model error (left), gradient norm (middle), and objective function values (right) for $25$ sources at a noise level of $\sigma = 0.1$}
    \label{fig:convergence_plots_25sources}
\end{figure}

\begin{figure}[t]
    \centering
\includegraphics[width=0.4\linewidth]{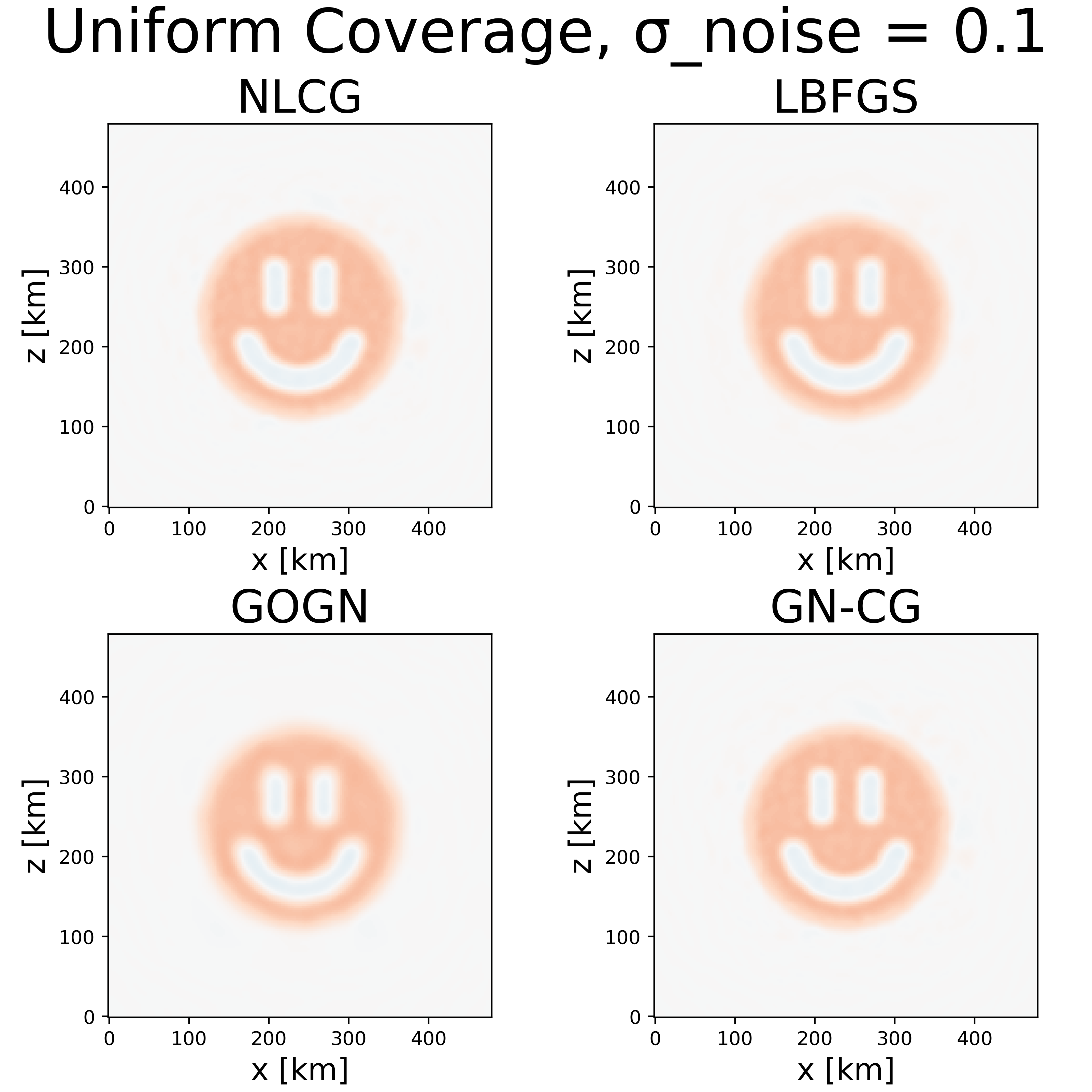}   
\includegraphics[width=0.4\linewidth]{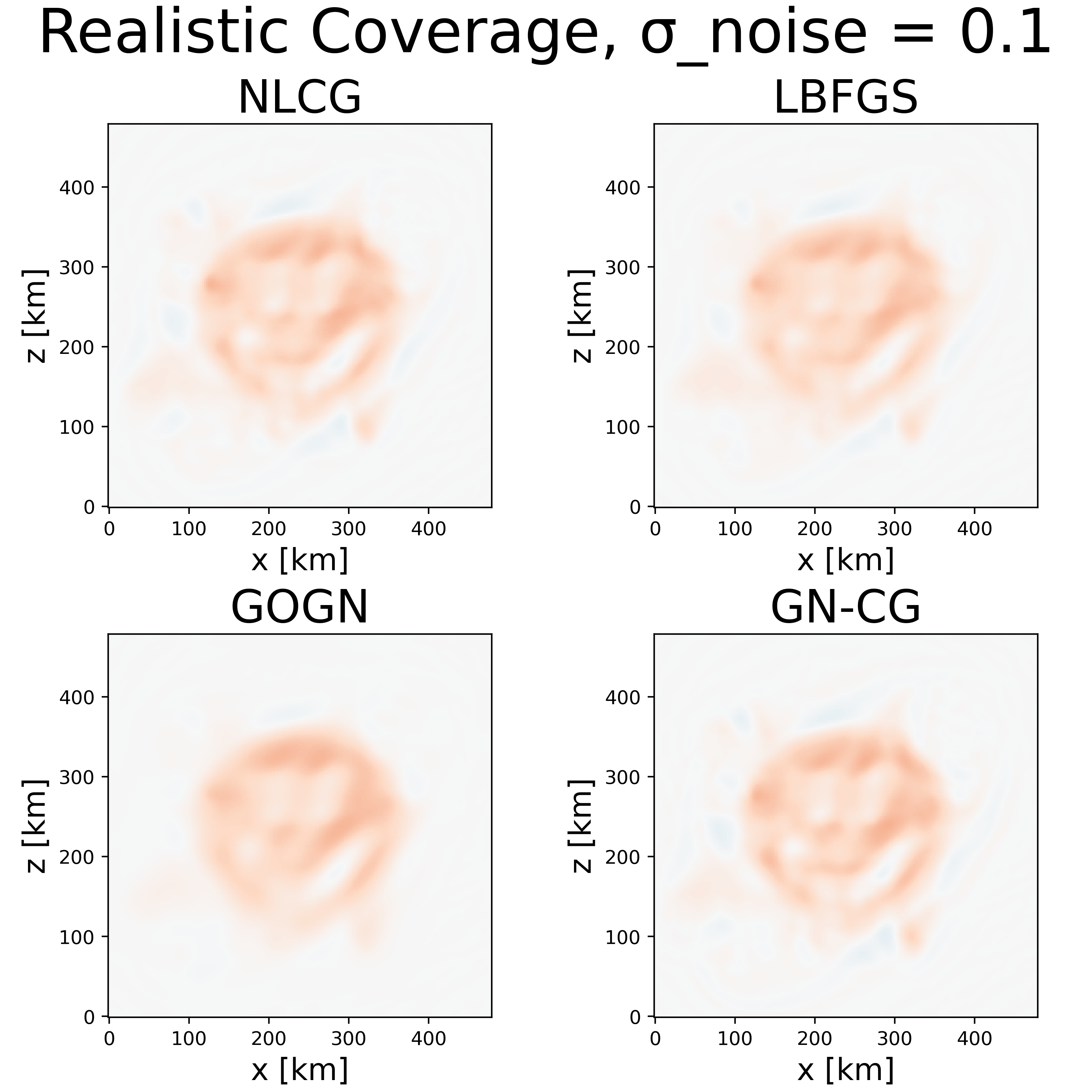}   
\includegraphics[width=.102\linewidth, trim=0cm 4cm 0cm 0cm, clip]{figures/cbar.png}    
    \caption{Final reconstructions from the experiments in Fig.~\ref{fig:convergence_plots_25sources} }
    \label{fig:reconstructions_25sources}
\end{figure}

\subsection{Results and Discussion}
\label{subsec:results}
In Figure \ref{fig:convergence_plots_5sources}, we show convergence plots for model error $\| \bm_{\rm true} - \bm_k \|$ (left column), gradient norm (middle column), and objective function values (right column). These are shown for uniform receiver coverage with $8$ sources (top row) and realistic receiver coverage with $5$ sources (bottom row), and are all run with noise level $\sigma = 0.1$. To make a fair comparison, we plot these results over the number of PDE solves in the x-axis, which is the primary computational budget in large-scale PDE constrained optimization problems. 
In this plots, we observe that GOGN outperforms all other methods when using realistic receiver coverage but remains competitive for the uniformly distributed receiver coverage. The corresponding reconstructions are shown in Figure~\ref{fig:reconstructions_5sources}. Similar results are shown in the case of 25 sources as shown in Figure~\ref{fig:convergence_plots_25sources} for convergence and Figure~\ref{fig:reconstructions_25sources}. Additional experiments for different noise levels are provided in Appendix~\ref{app:additional_experiments}.

These results indicate that GOGN is a competitive algorithm, offering improved efficiency in terms of PDE solves. While its performance is comparable to existing methods under uniformly distributed receiver coverage, GOGN appears particularly promising for more realistic receiver configurations. The mechanisms underlying this behavior are not yet fully understood, but we suspect that this is connected to greater ill-conditioning of the matrix $\bI +
\bJ^{\rm GO} (\bD^\top \bD)^{-1} {\bJ^{\rm GO}}^\top,$ which corresponds with a more difficult optimization landscape whose ill-conditioning is likely better captured by the GOGN Hessian than by CG or LBFGS.

Because GOGN requires fewer PDE solves to achieve meaningful model improvements, a natural strategy is to employ GOGN during the early stages of inversion to take advantage of its rapid initial convergence, and then transition to CG or GNCG for a small number of final iterations (potentially only one iteration of GNCG) to refine the inversion result, exploiting their superior long-term convergence properties. We expect this strategy may be especially useful when the source–receiver geometry is more favorable. This strategy is motivated by the observation that the effectiveness of the GOGN update depends on the rank and conditioning of the matrix $\tilde{J}$. When $\tilde{J}$ is severely rank-deficient, the amount of informative curvature captured by GOGN is limited; in the extreme case where $\tilde{J}$ has rank one, the method effectively reduces to standard gradient descent. 

It is also worth considering the comparison with LBFGS. Unlike GOGN, LBFGS does not require the solution of an $N \times N$ system to compute an update, but instead directly computes an update using dot products with vectors based on the update history of the optimization. Furthermore for LBFGS, only one vector, the model update, must be smoothed by the application of $\bD^\top \bD,$ instead of the $N$ separate gradients which must be smoothed to apply GOGN. However, in taking on these small additional costs to implement GOGN instead of LBFGS, we achieve better iterations early in the optimization, and a reconstruction that is more robust to observational noise. The former benefit is easily explained- LBFGS requires an update history to build its best approximation to the inverse Hessian, while GOGN approximates the Hessian using only information from the current iteration. In this sense, GOGN has a ``head-start'' in comparison with LBFGS. The latter observation is more difficult to explain, and further research is required to see if this robustness to observational noise is problem or implementation-dependent.

\section{Conclusion}

We proposed GOGN, a computationally efficient Gauss--Newton–type method for large-scale, regularized PDE-constrained optimization. By rewriting a sum of objectives as a sum of squares, we obtain a structured approximation to the Hessian that can be assembled from gradients of the individual objectives. When these objectives correspond to distinct physical states, this approximation incurs no additional PDE solves beyond those required to compute the gradient of the full objective.

Through numerical experiments in full-waveform inversion, we have shown that GOGN is a competitive optimization method, particularly in terms of efficiency measured by PDE solves. While its performance is comparable to standard approaches such as CG and GNCG under idealized, uniformly distributed receiver coverage, GOGN shows particular promise for more realistic receiver coverage used in regional and global scale FWI. The effectiveness of GOGN is influenced by the rank and conditioning of the approximate Jacobian matrix, which governs how much curvature information can be captured. In regimes where this matrix is severely rank-deficient, the method naturally degenerates toward gradient descent, limiting the attainable model improvement.

These observations suggest a practical hybrid strategy in which GOGN is used during the early stages of inversion to exploit its rapid initial convergence, followed by a transition to CG or GNCG to leverage their superior long-term convergence properties when the source–receiver configuration is more favorable. Understanding the precise relationship between acquisition geometry, the structure of approximate Jacobian matrix, and the performance of GOGN remains an important direction for future work.

\section*{Acknowledgments}
This study is supported by the National Science Foundation projects with grant numbers EAR-1945565 and OAC-2103621. 


%
%

\bibliographystyle{spmpsci}      
\bibliography{thesis}   


\appendix
\section{Appendix}
\subsection{Additional Experimental Details} \label{app:additional_experimental_details}
We employ a receiver-weighting approach \cite{Ruan2019}, where the observations are weighted based on source-receiver geometry. Rewriting the unweighted data misfit $\Phi ( \bm ) = \frac 12 \| \br_i ( \bm ) \|^2$ as a sum over misfits of individual seismograms
$$
\Phi ( \bm ) = \frac 12 
\sum_{i=1}^{N} \sum_{j=1}^{n_r} 
\| \bs_{ij} ( \bm ) - \bs_{ij}^{\rm obs} \|^2,
$$
where $N$ is the number of sources, $n_r$ is the number of receivers, $\bs_{ij}^{\rm obs}$ is the recorded seismogram for this source-receiver pair, and $\bs_{ij} ( \bm )$ is the corresponding simulated seismogram, we can consider a weighted data misfit
$$
\Phi ( \bm ) = \frac 12 
\sum_{i=1}^{N} \sum_{j=1}^{n_r} w_{ij}^2 
\| \bs_{ij} ( \bm ) - \bs_{ij}^{\rm obs} \|^2.
$$
In particular, the observation associated with the $i$th source and $j$th receiver is weighted by the factor
$$
w_{ij} = \frac{1}{\| \bs_{ij}^{\rm obs} \| \sqrt{n_r^{-1} \sum_{\ell=1}^{n_r} k ( \| \bx^r_i - \bx^r_\ell \| ) }},
$$
where $\bx_i^r$ is the $i$th receiver location, and $k$ is a Gaussian function
$$
k ( r ) = 
\frac{1}{(2\pi)^{1/2}\sigma_k} \re^{-{r^2}/{2 \sigma_k^2}},
$$
with chosen value  $\sigma_k = 100$. Dividing by the norm of the observed waveform $\| \bs^{\rm obs}_{ij} \|$ adjusts for the decay of amplitudes as waveforms travel away from the source, while dividing by the factor $\sqrt{n_r^{-1} \sum_{\ell=1}^{n_r} k ( \| \bx^r_i - \bx^r_\ell \| ) }$ adjusts for receiver density. Without the former weighting factor then, measurements with short distances from the source to receiver are weighted more heavily in the gradient, and without the latter, stronger updates will be made in regions with more receivers, leading to spatially imbalanced updates.

%
%



%

\subsection{Additional Experiments}
\label{app:additional_experiments}

We summarize the results of  a few additional experiments to show the impact of changing the noise level. We consider the same four source-receiver configurations as in Figure \ref{fig:source_setup}, which showcase realistic and uniformly distributed coverage with different numbers of sources. In addition to the noise level $\sigma=0.1$ consider in previously, we consider the two more noise levels $\sigma = 0.01$ and $\sigma = 1.0,$ and compare the performance of each algorithm.

(Relative) optimization statistics as a function of the number of PDE solves are displayed in Figures \ref{fig:optstats_simple2}, \ref{fig:optstats_simple252}, \ref{fig:optstats_simple0}, and \ref{fig:optstats_simple250}, while the final reconstructions are displayed in  Figures \ref{fig:rec_simple2}, \ref{fig:rec_simple252}, \ref{fig:rec_simple0}, and \ref{fig:rec_simple250}. The general trend in these experiments is that, for the low noise level $\sigma = 0.01$, a similar reconstruction is reached by each algorithm, with GOGN outperforming or matching the performance of other algorithms after a small number of simulations in terms of model error and objective function value. These advantages are clear with realistic data coverage, but more subtle with uniform coverage. When the noise level is increased to $\sigma = 1.0$, GOGN's advantages at optimizing the objective function disappear, however, it still generally outperforms other algorithms in terms of model error, and it can be visually confirmed that the reproductions are less corrupted by the noise. This evidence suggests that GOGN is more robust to the difficulties of realistic source-receiver configurations and observational error than CG, LBFGS, and GN-CG. 


\FloatBarrier

\begin{figure}[t]
    \centering
    Uniform Coverage, $\sigma = 0.01$ \newline 
    \includegraphics[width=0.3\linewidth]{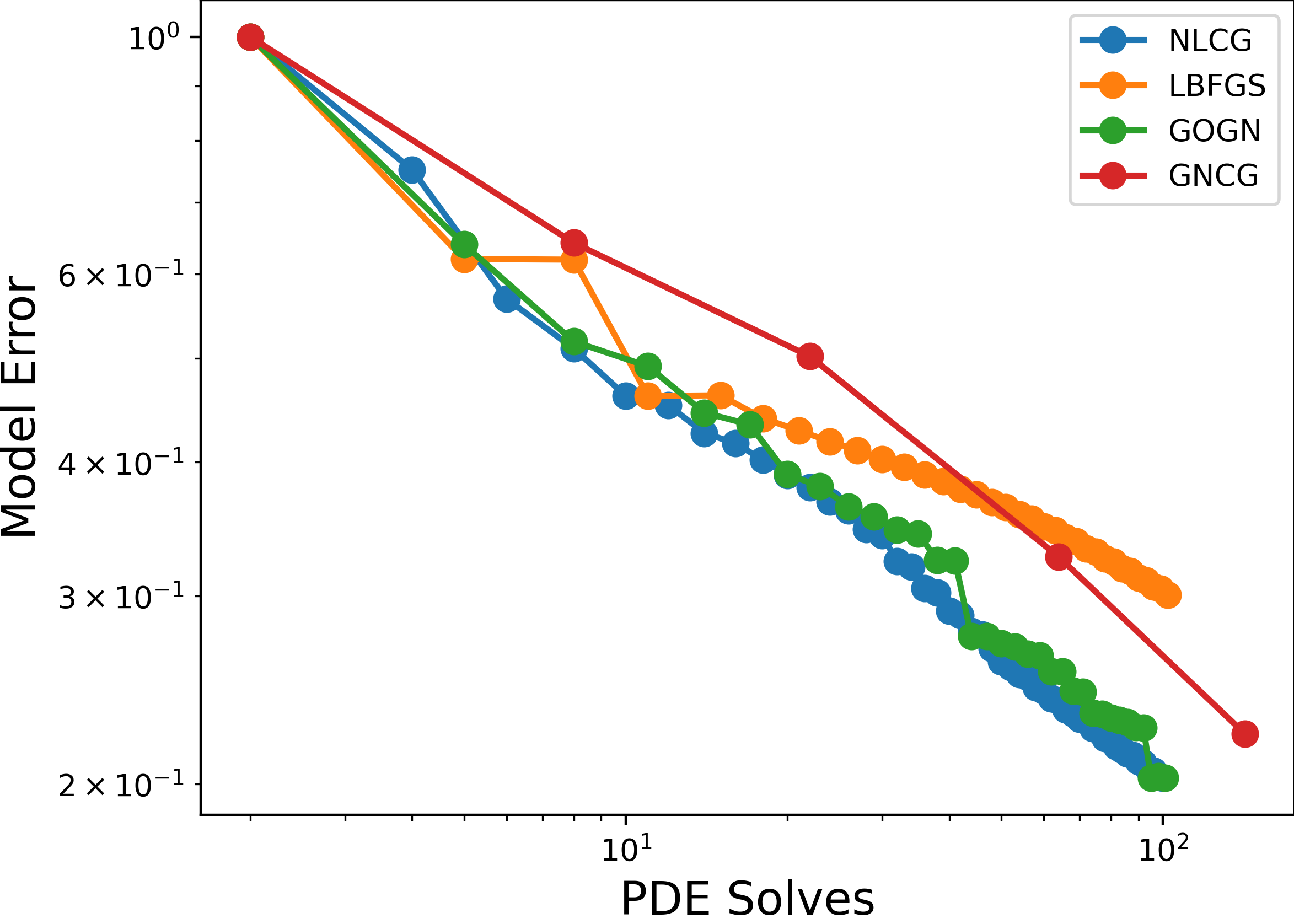}
    \includegraphics[width=0.3\linewidth]{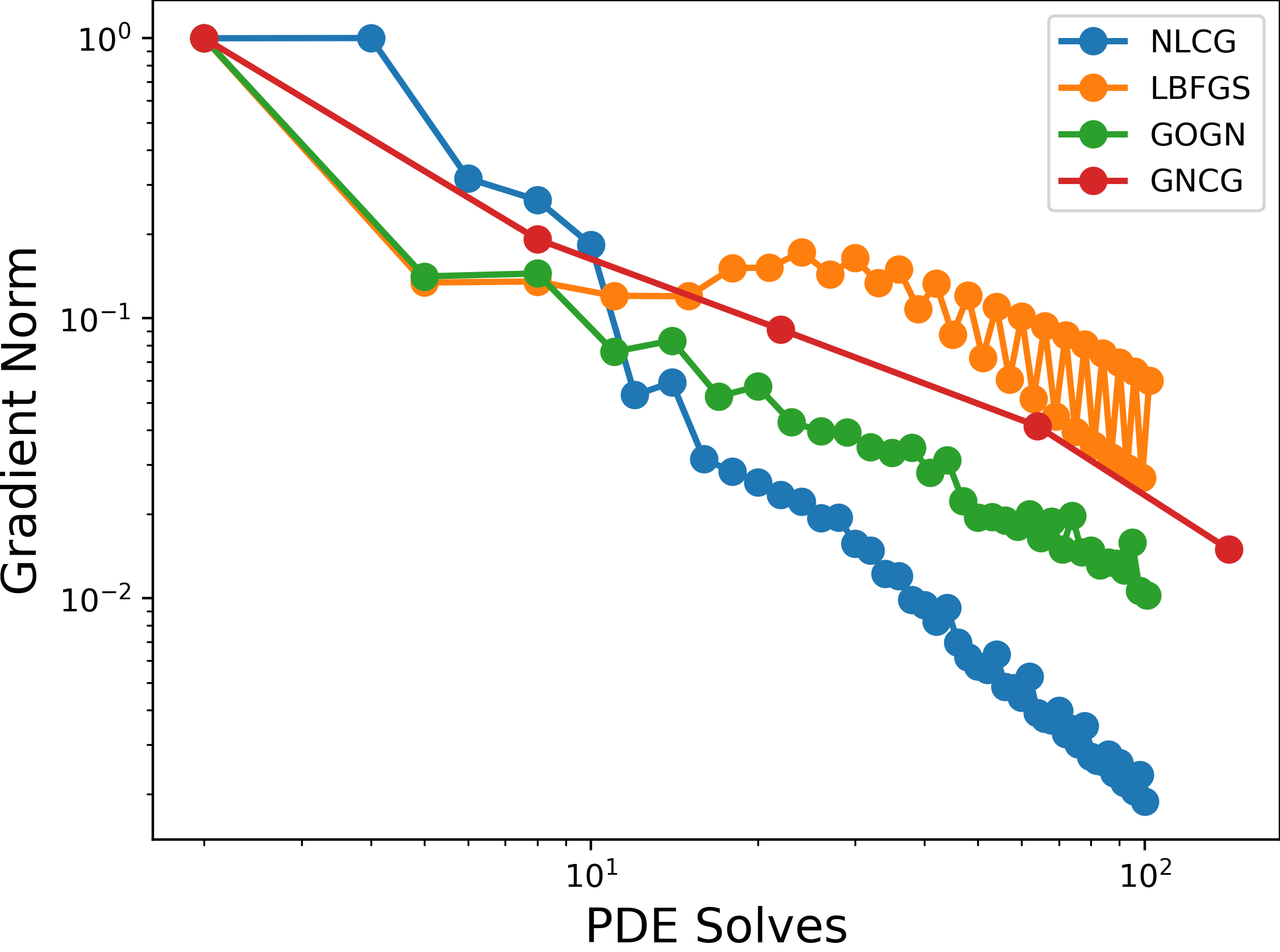}
    \includegraphics[width=0.3\linewidth]{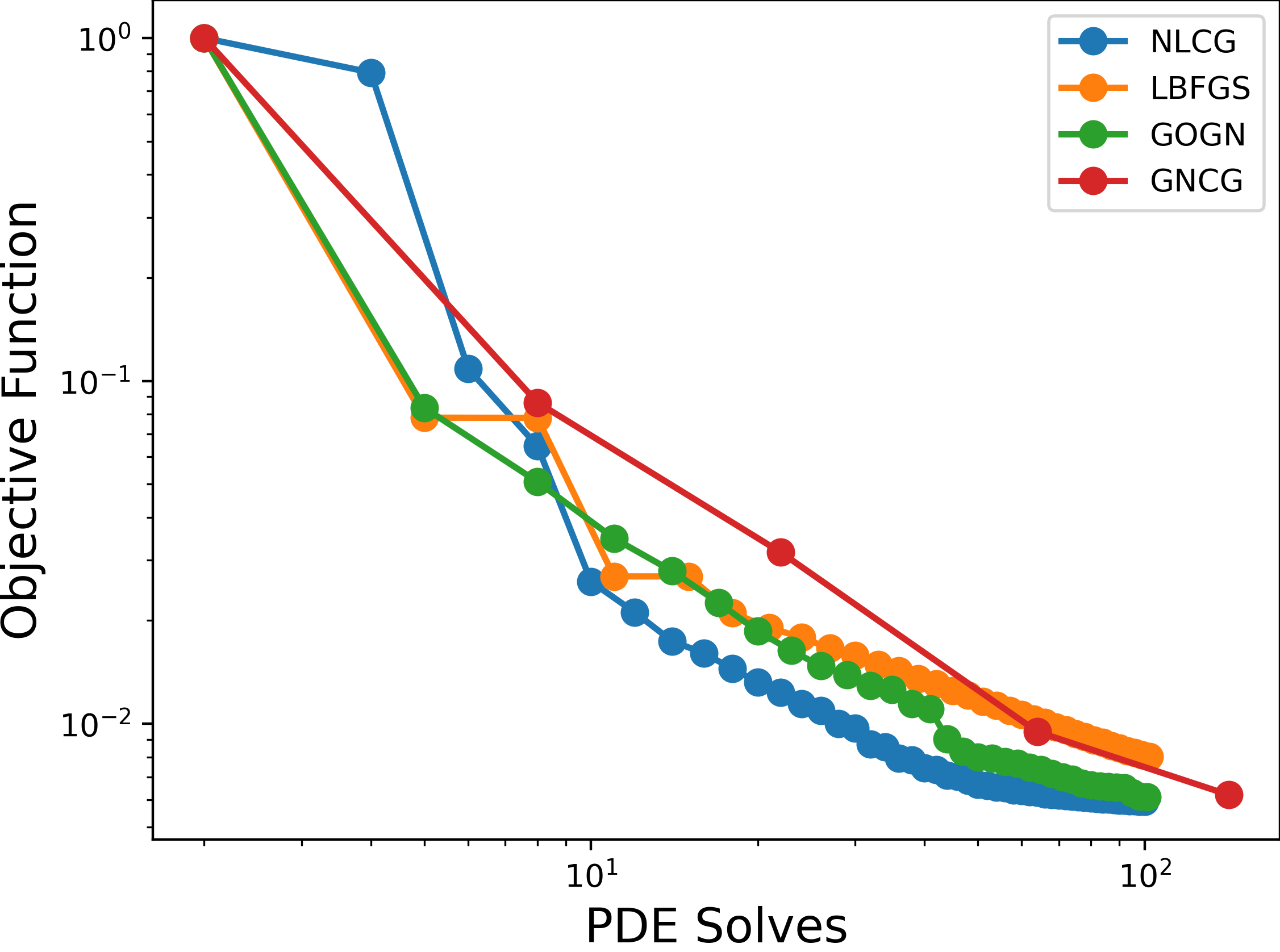}   
    \newline Realistic Coverage, $\sigma = 0.01$ \newline 
    \includegraphics[width=0.3\linewidth]{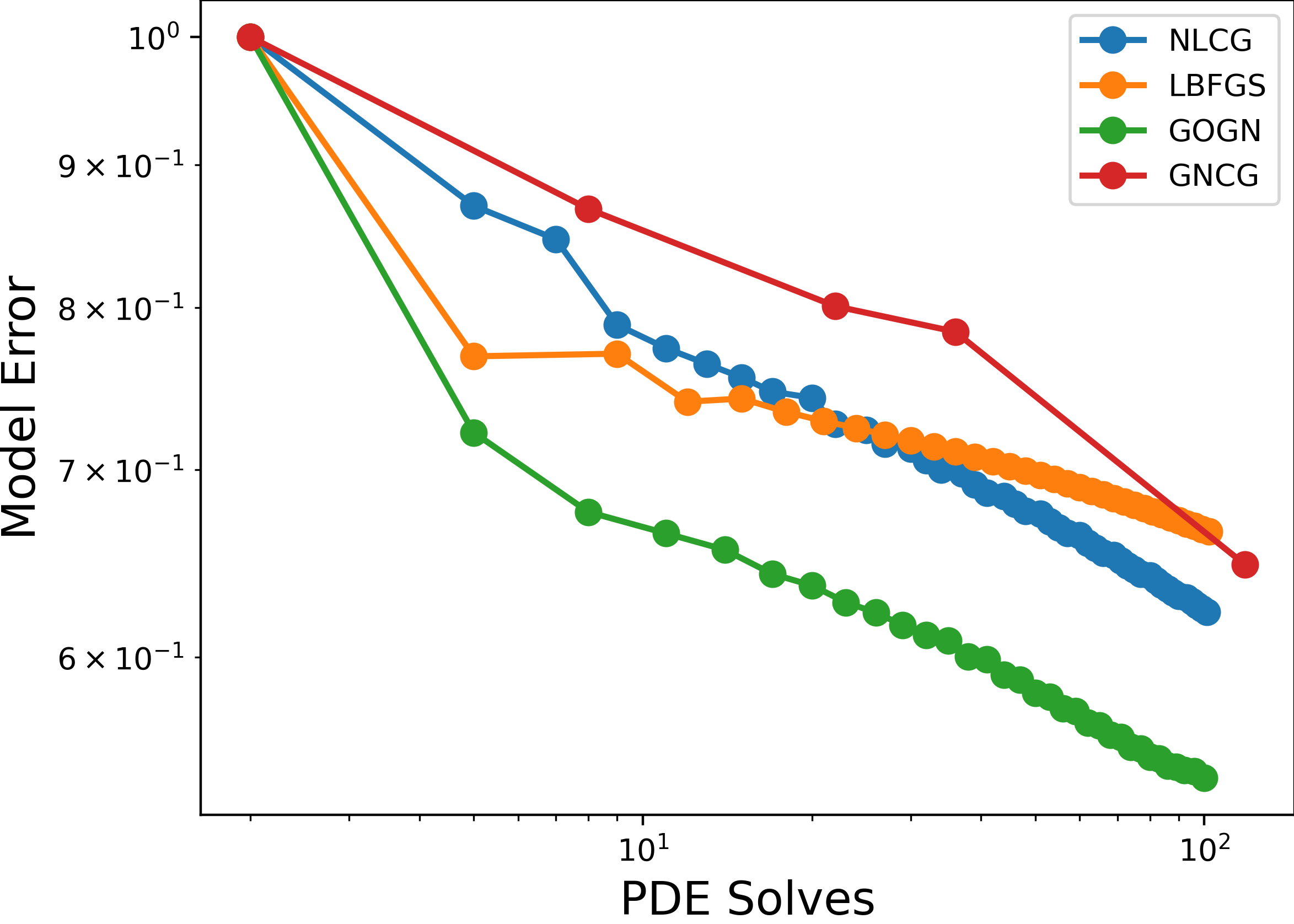}
    \includegraphics[width=0.3\linewidth]{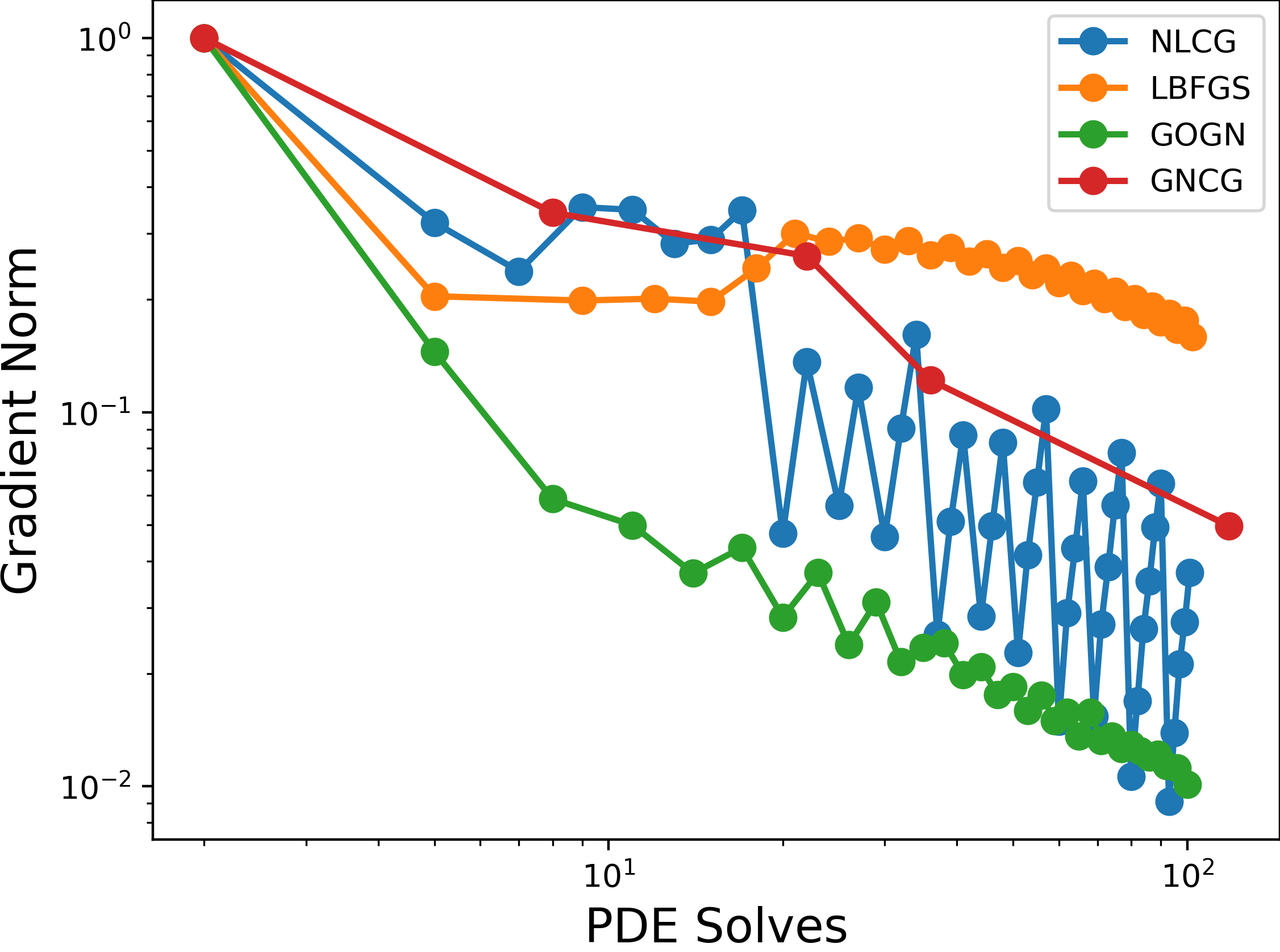}
    \includegraphics[width=0.3\linewidth]{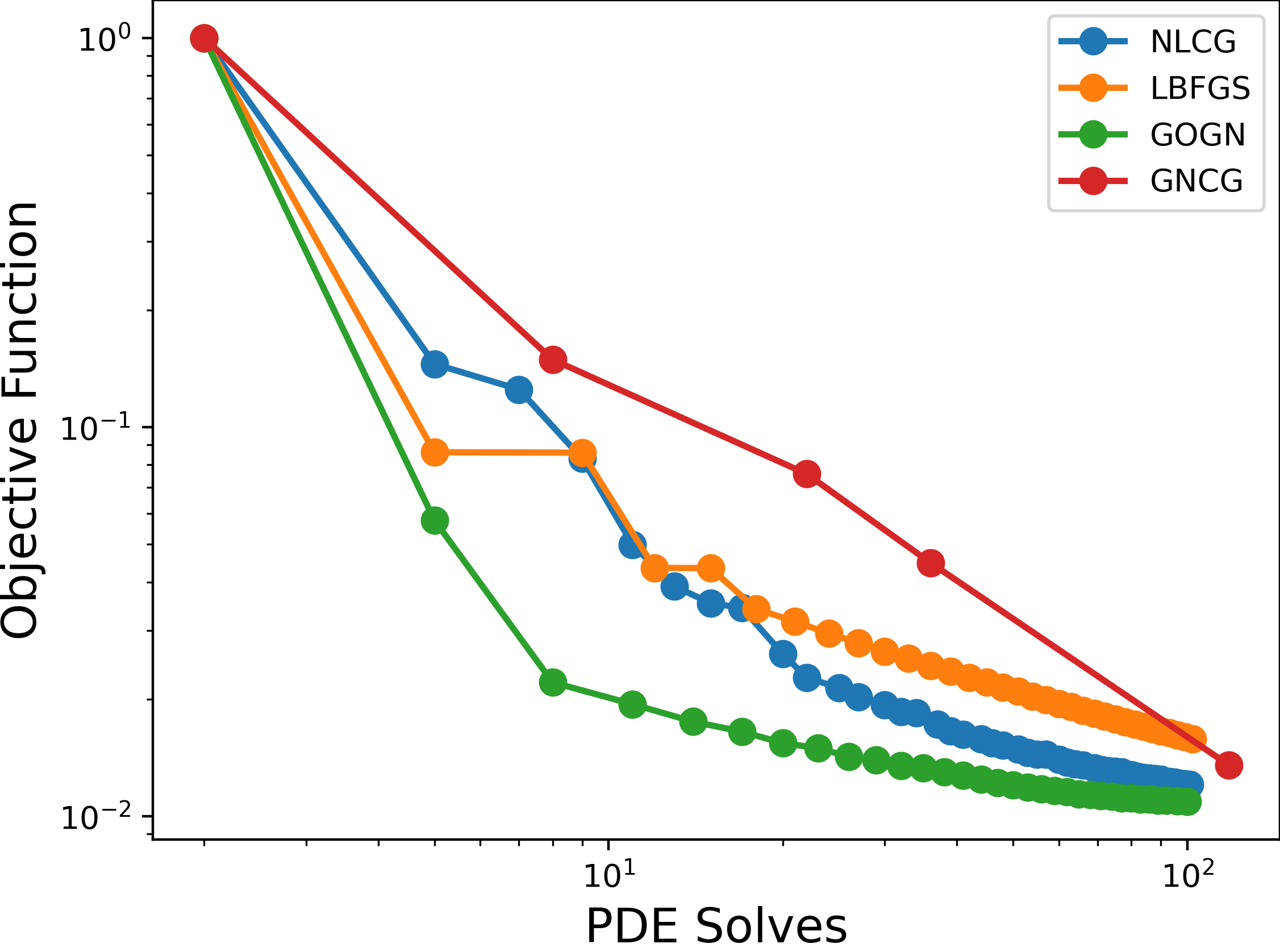}      
    \caption{Convergence plots for uniformly distributed receiver coverage (top row) and realistic coverage (bottom row), with X-axes across all images denoting number of PDE solves during optimization and Y-axes denoting model error (left), gradient norm (middle), and objective function values (right) for $8$ and $5$ sources, respectively, at a noise level of $\sigma = 0.01$}
    \label{fig:optstats_simple2}
\end{figure}

\begin{figure}[t]
    \centering
\includegraphics[width=0.4\linewidth]{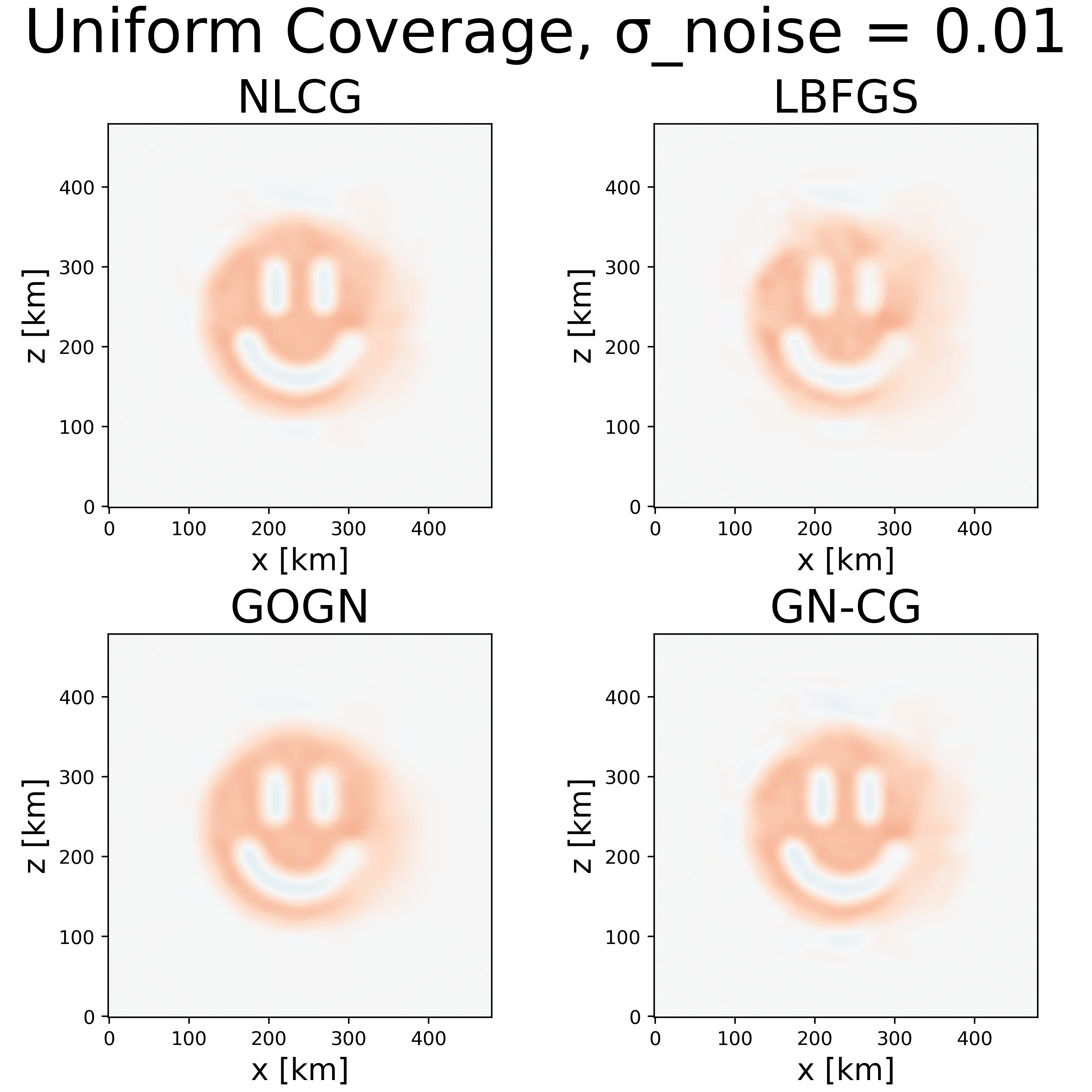}   
\includegraphics[width=0.4\linewidth]{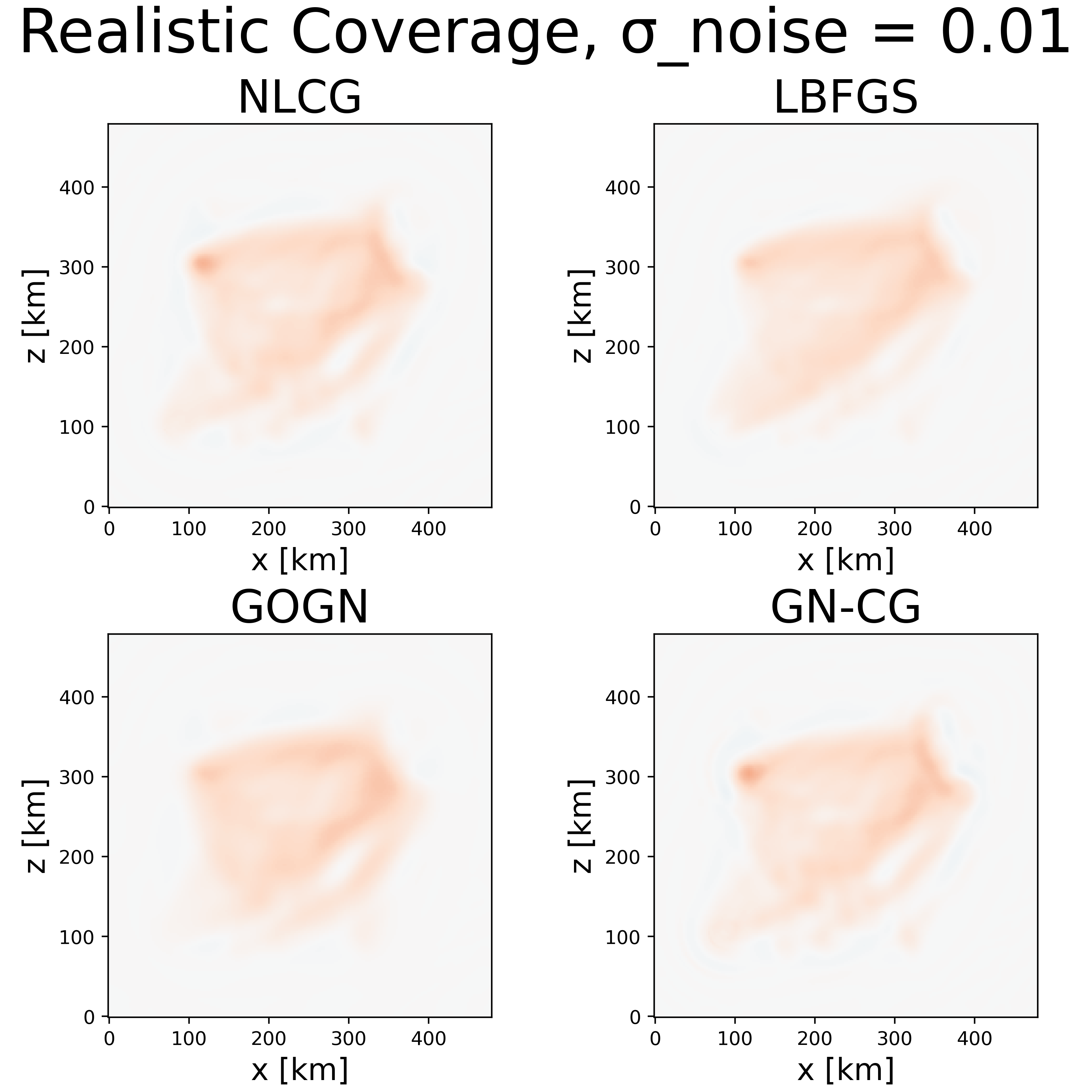}   
\includegraphics[width=.102\linewidth, trim=0cm 4cm 0cm 0cm, clip]{figures/cbar.png}    
    \caption{Final reconstructions from the experiments in Figure~\ref{fig:optstats_simple2} }
    \label{fig:rec_simple2}
\end{figure}

\begin{figure}[t]
    \centering
    Uniform Coverage, $\sigma = 0.01$ \newline 
    \includegraphics[width=0.3\linewidth]{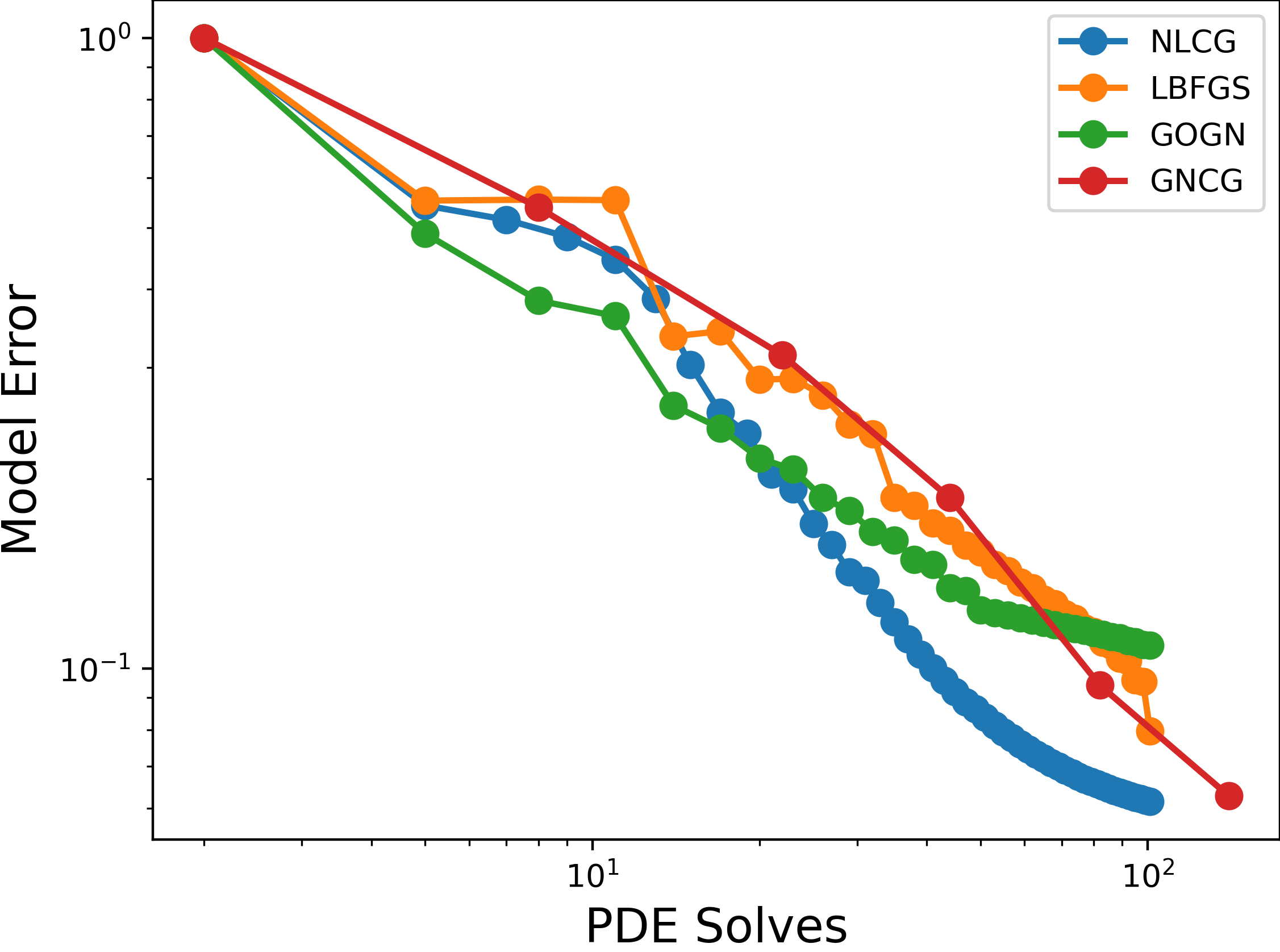}
    \includegraphics[width=0.3\linewidth]{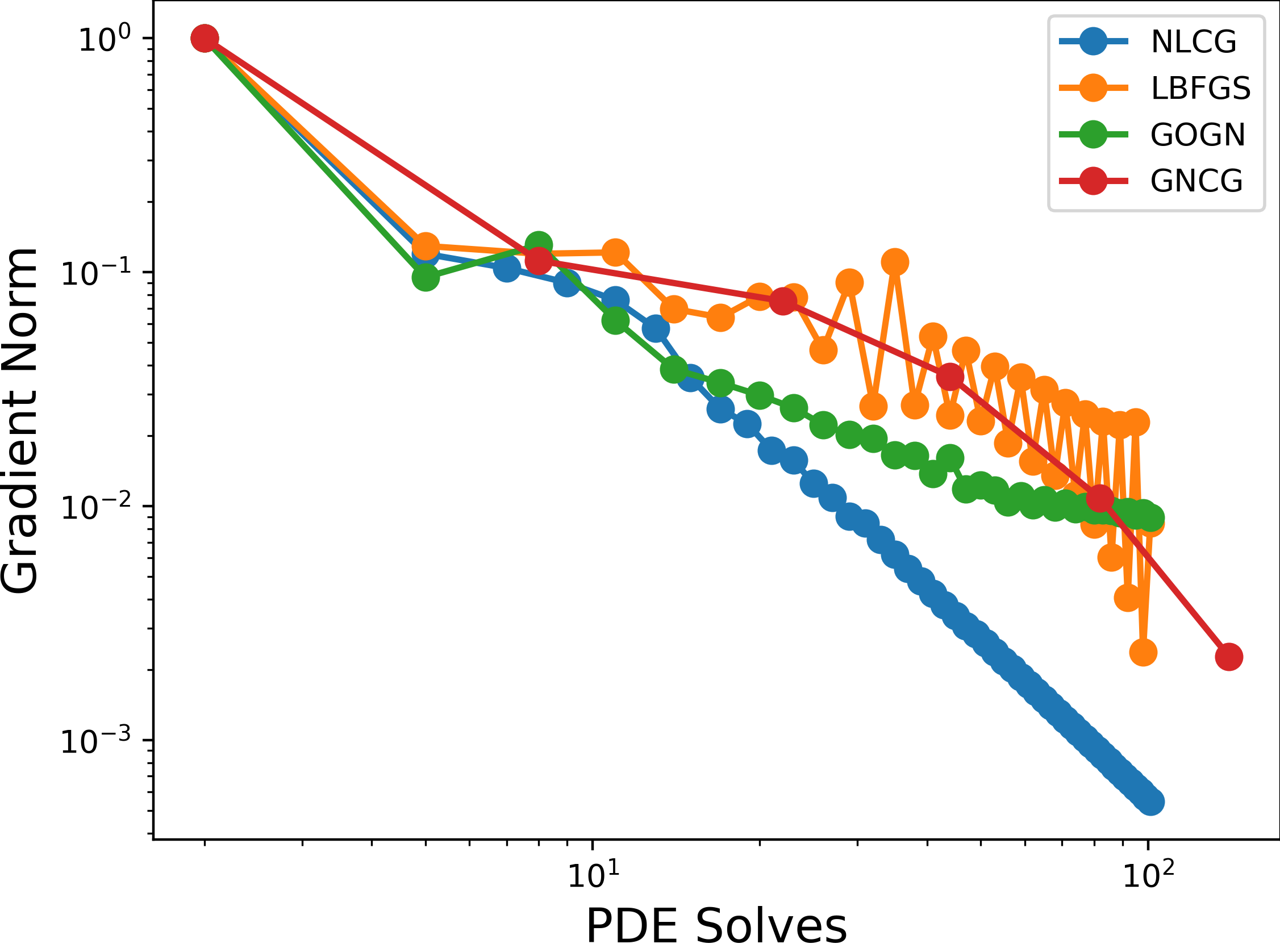}
    \includegraphics[width=0.3\linewidth]{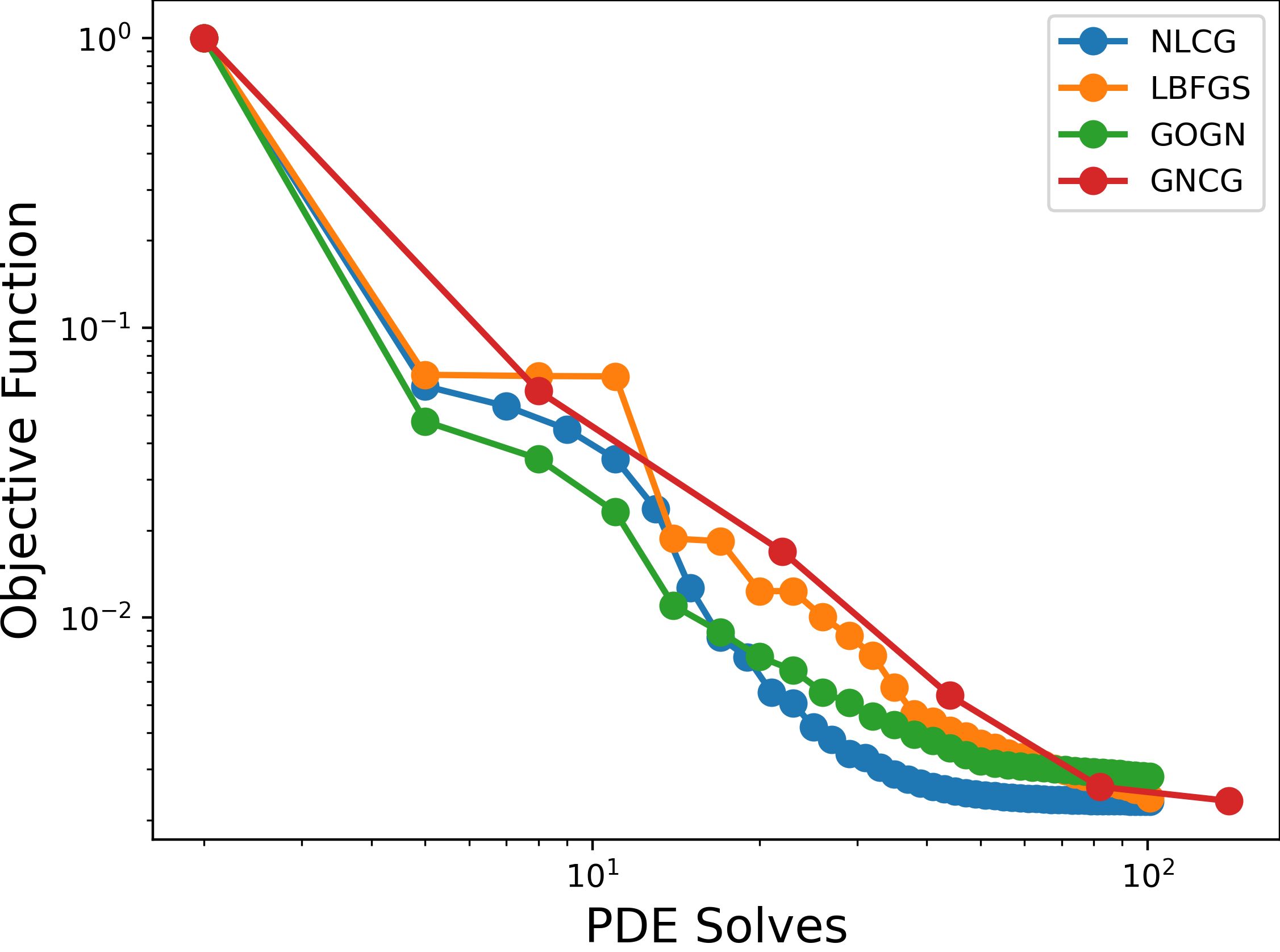}   
    \newline Realistic Coverage, $\sigma = 0.01$ \newline 
    \includegraphics[width=0.3\linewidth]{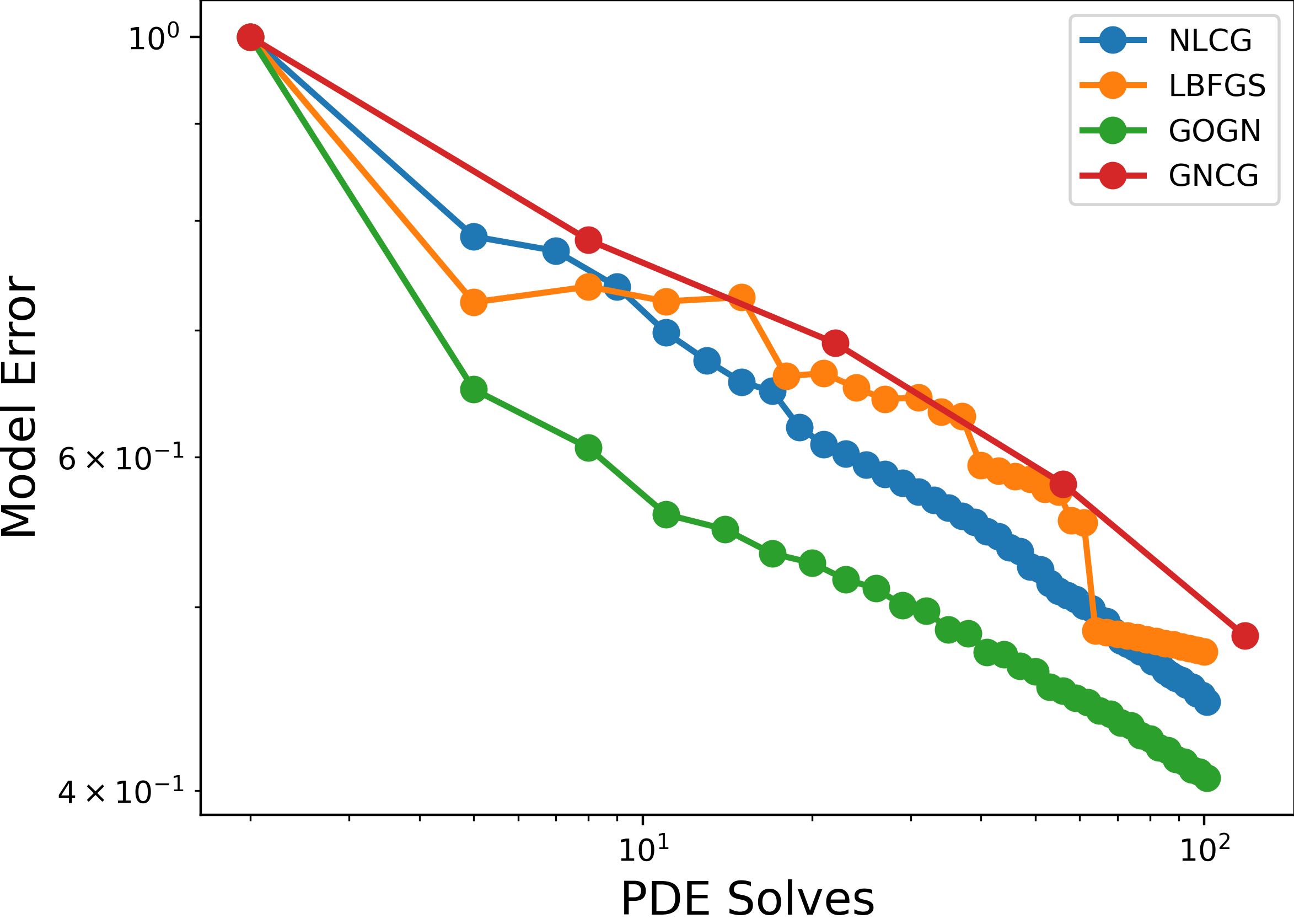}
    \includegraphics[width=0.3\linewidth]{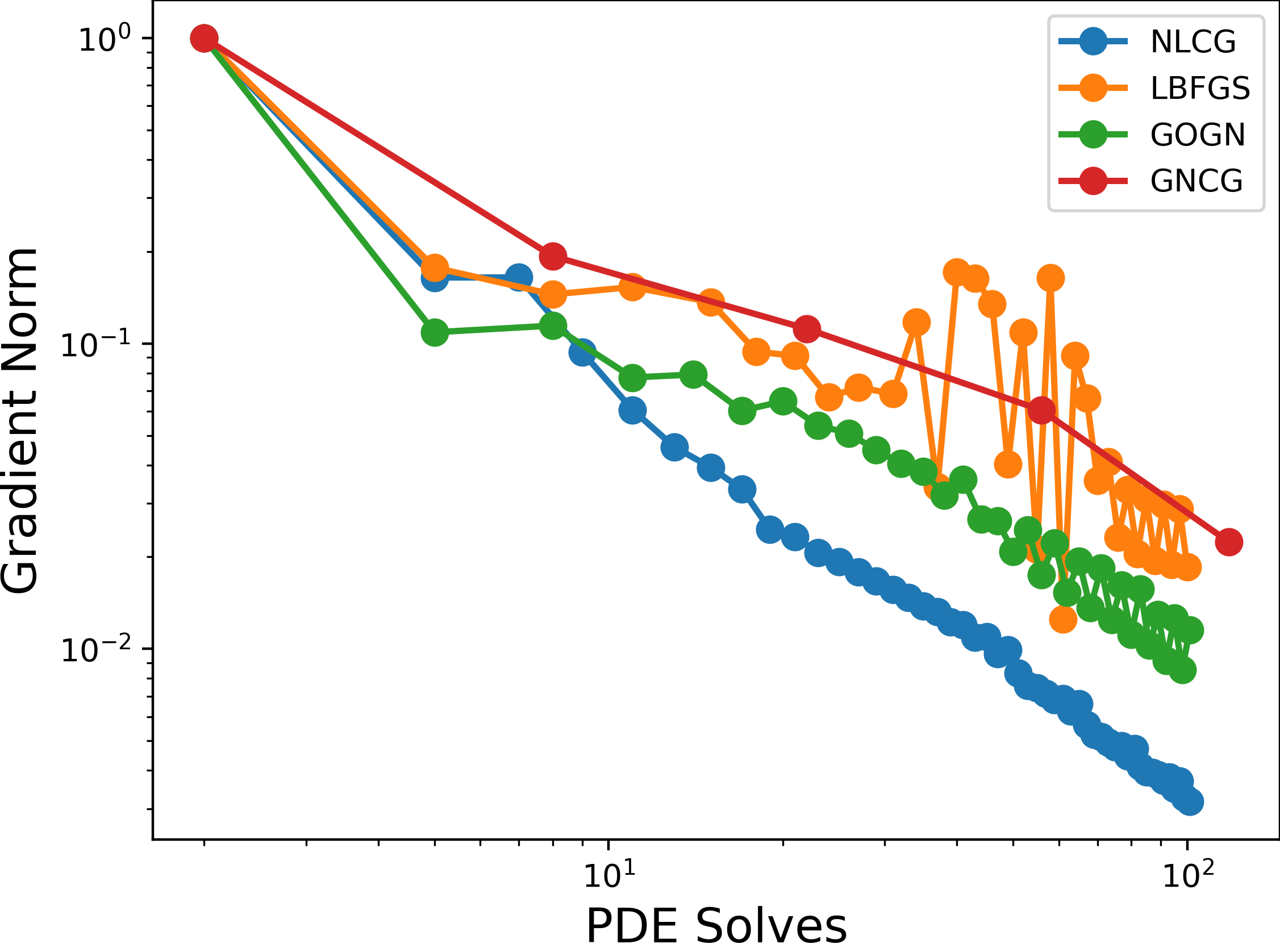}
    \includegraphics[width=0.3\linewidth]{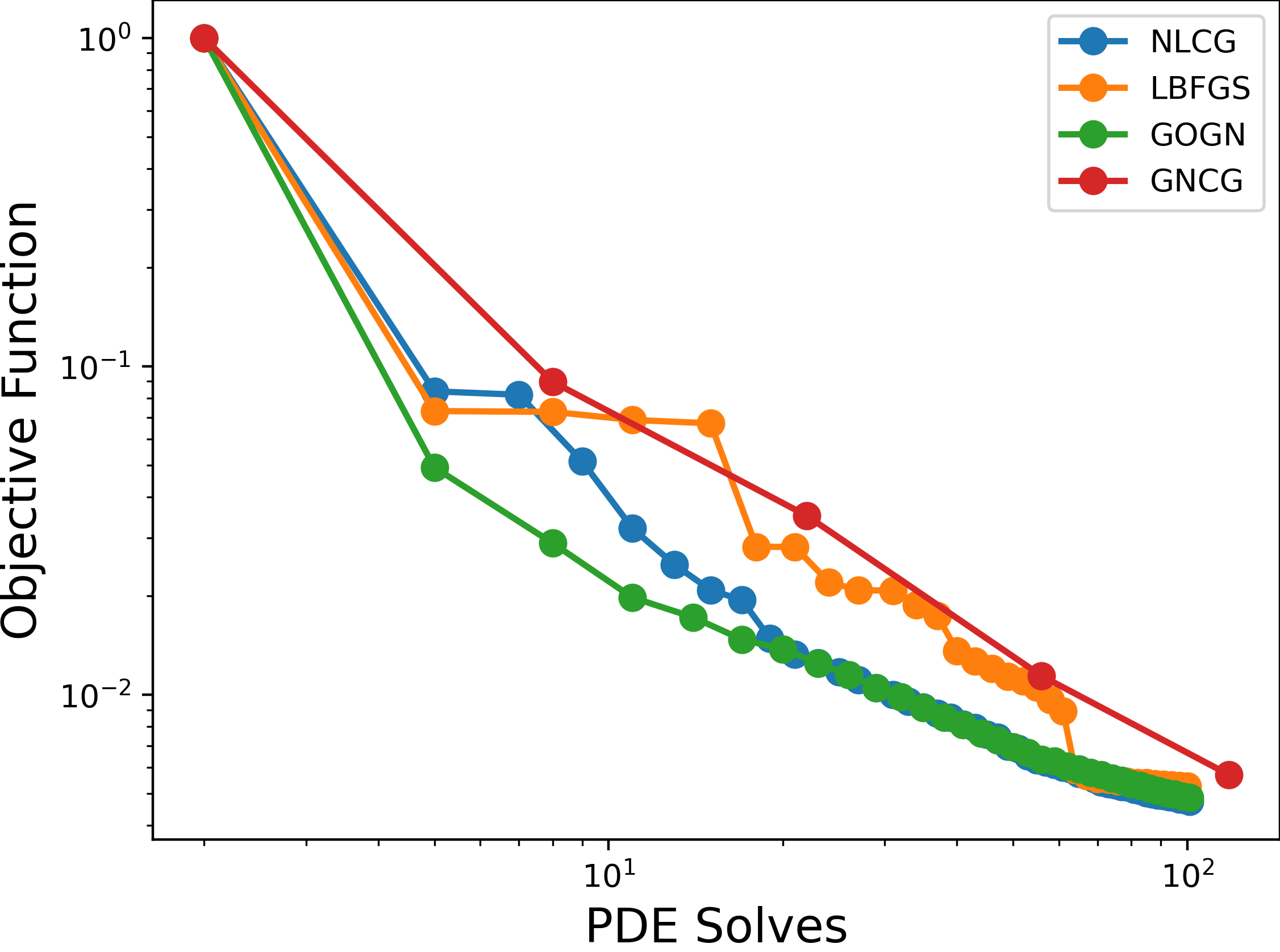}       
    \caption{Convergence plots for uniformly distributed receiver coverage (top row) and realistic coverage (bottom row), with X-axes across all images denoting number of PDE solves during optimization, and Y-axes denote model error (left), gradient norm (middle), and objective function values (right) for $25$ sources at a noise level of $\sigma = 0.01$}
    \label{fig:optstats_simple252}
\end{figure}

\begin{figure}[t]
    \centering
\includegraphics[width=0.4\linewidth]{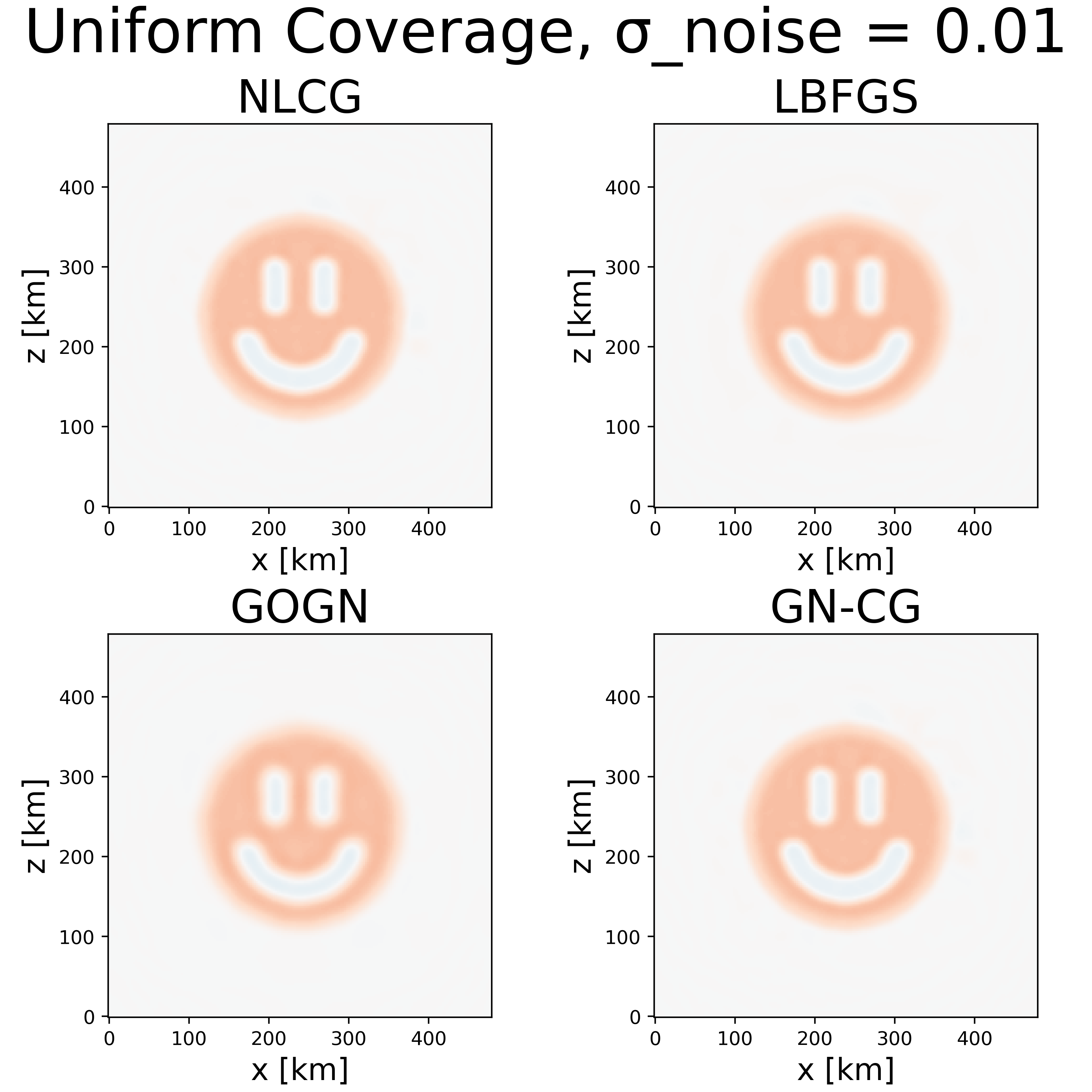}   
\includegraphics[width=0.4\linewidth]{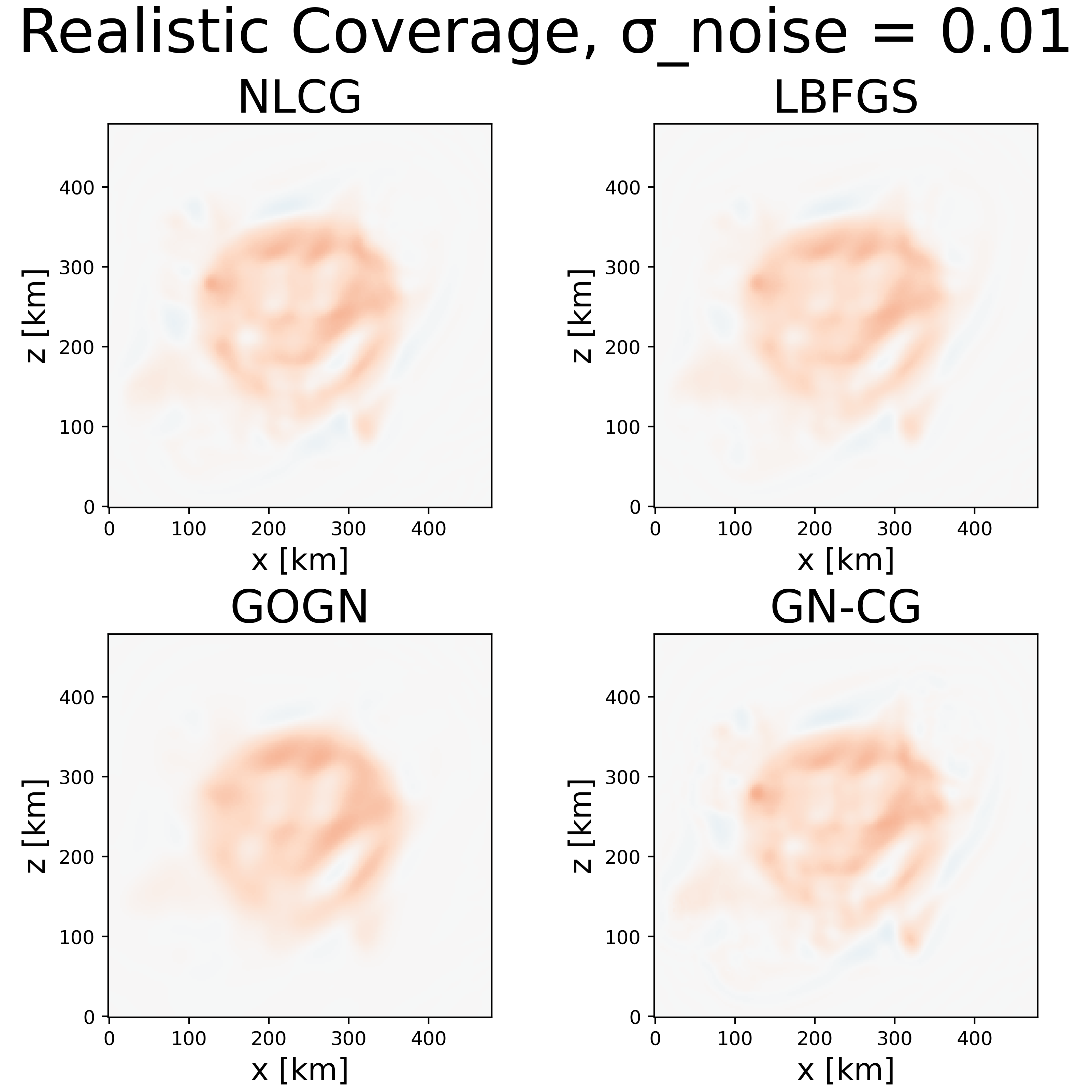}   
\includegraphics[width=.102\linewidth, trim=0cm 4cm 0cm 0cm, clip]{figures/cbar.png}    
    \caption{Final reconstructions from the experiments in Figure~\ref{fig:optstats_simple252} }
    \label{fig:rec_simple252}
\end{figure}

\begin{figure}[t]
    \centering
    Uniform Coverage, $\sigma = 1.0$ \newline 
    \includegraphics[width=0.3\linewidth]{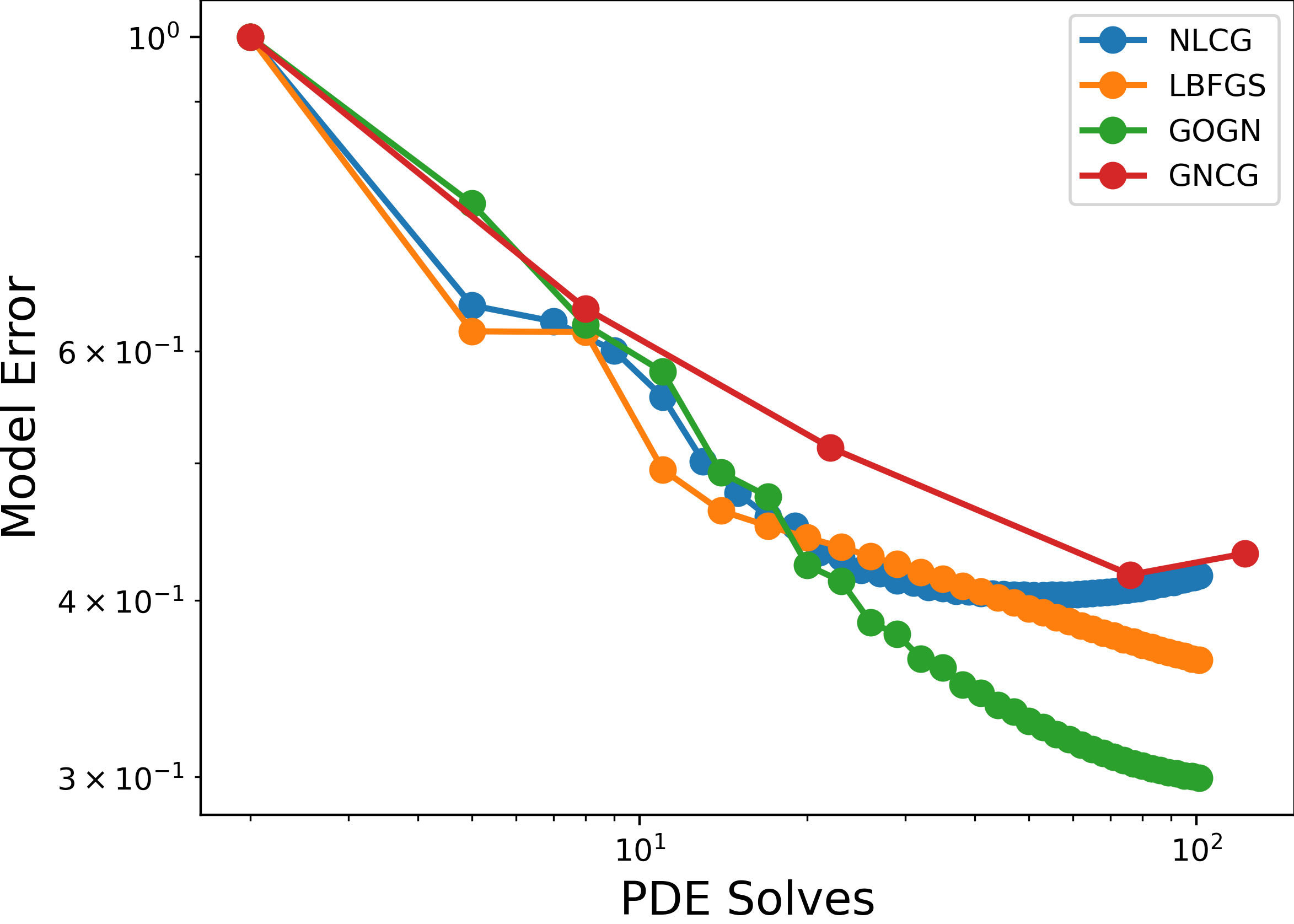}
    \includegraphics[width=0.3\linewidth]{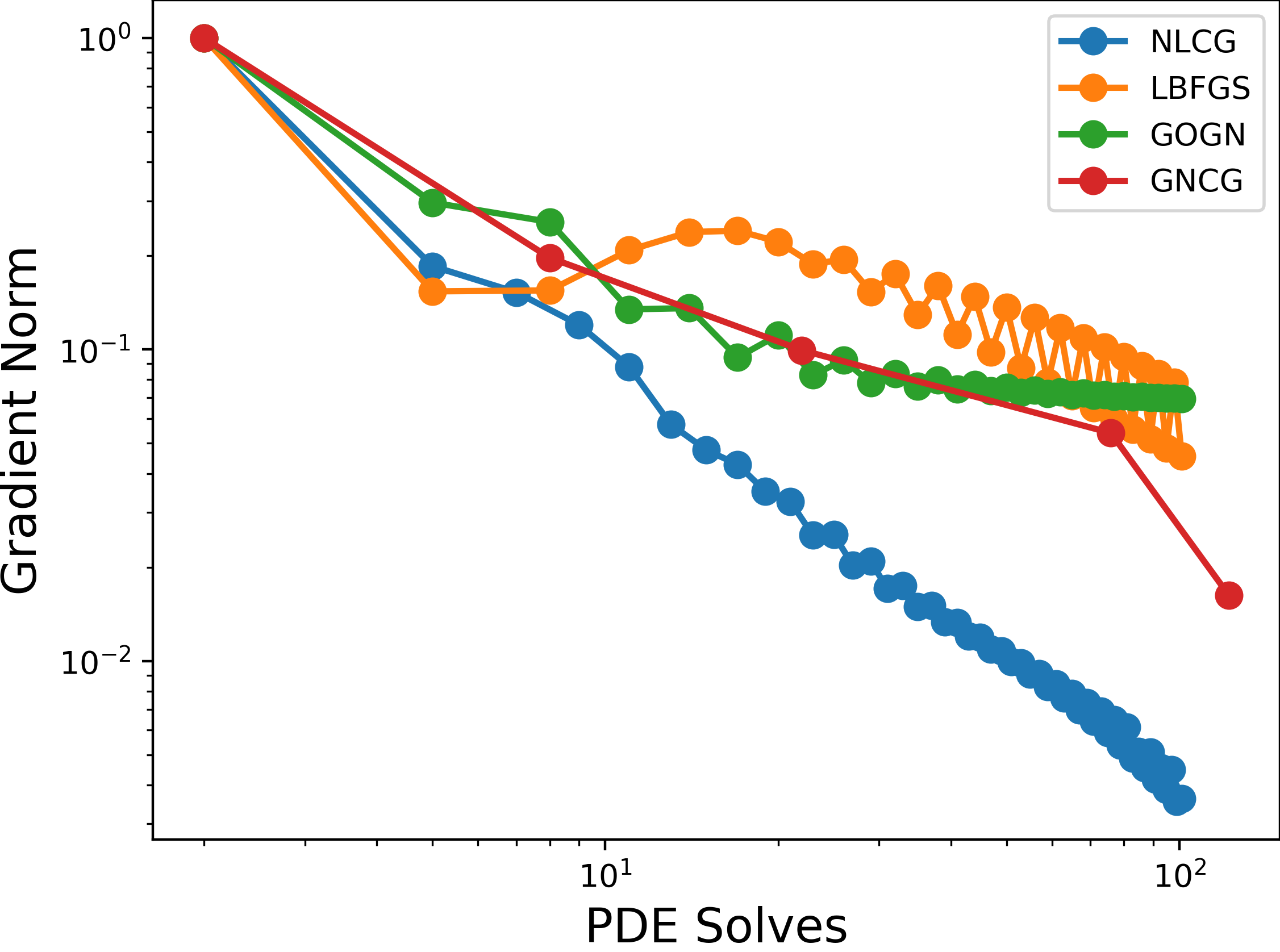}
    \includegraphics[width=0.3\linewidth]{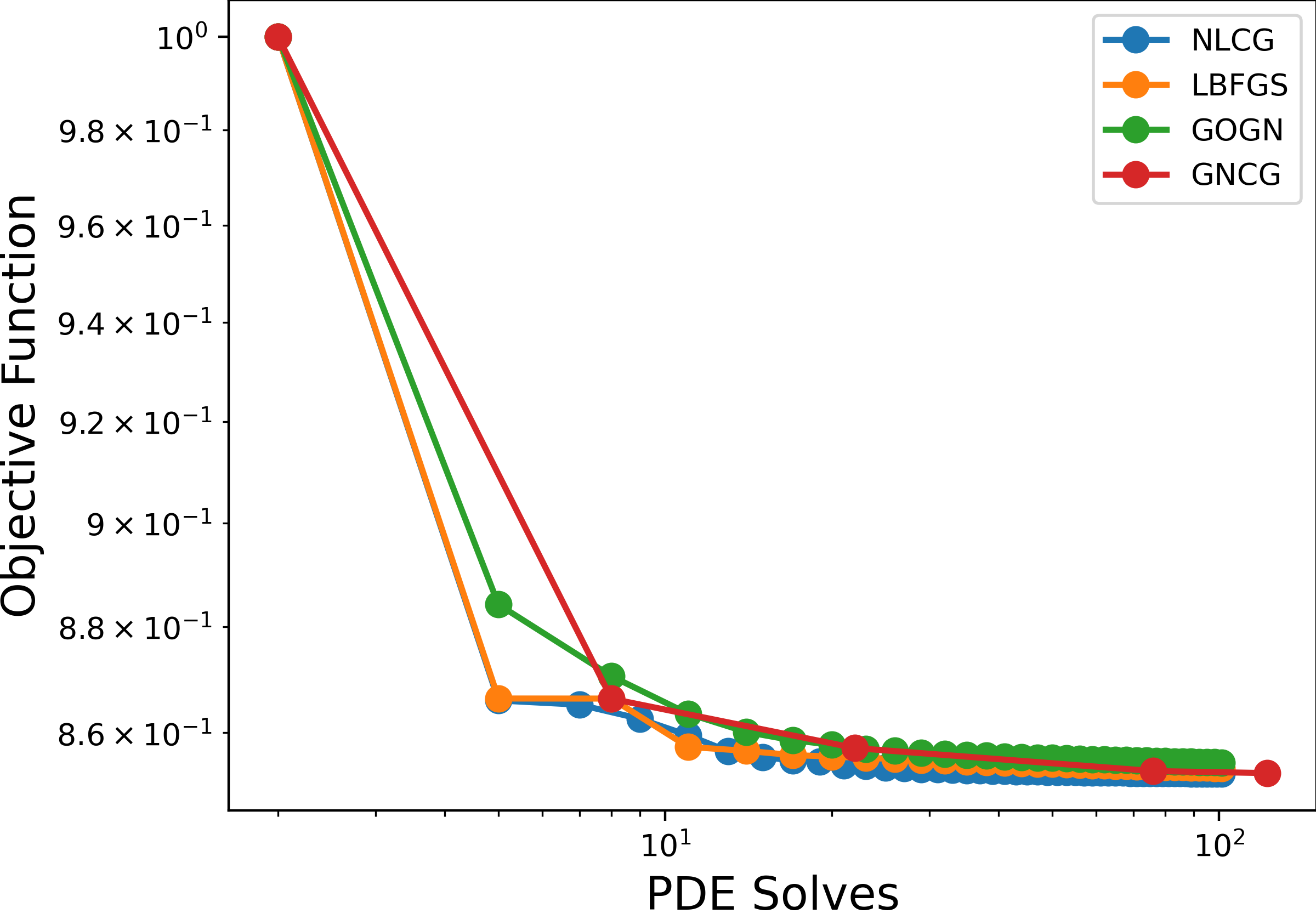}   
    \newline Realistic Coverage, $\sigma = 1.0$ \newline 
    \includegraphics[width=0.3\linewidth]{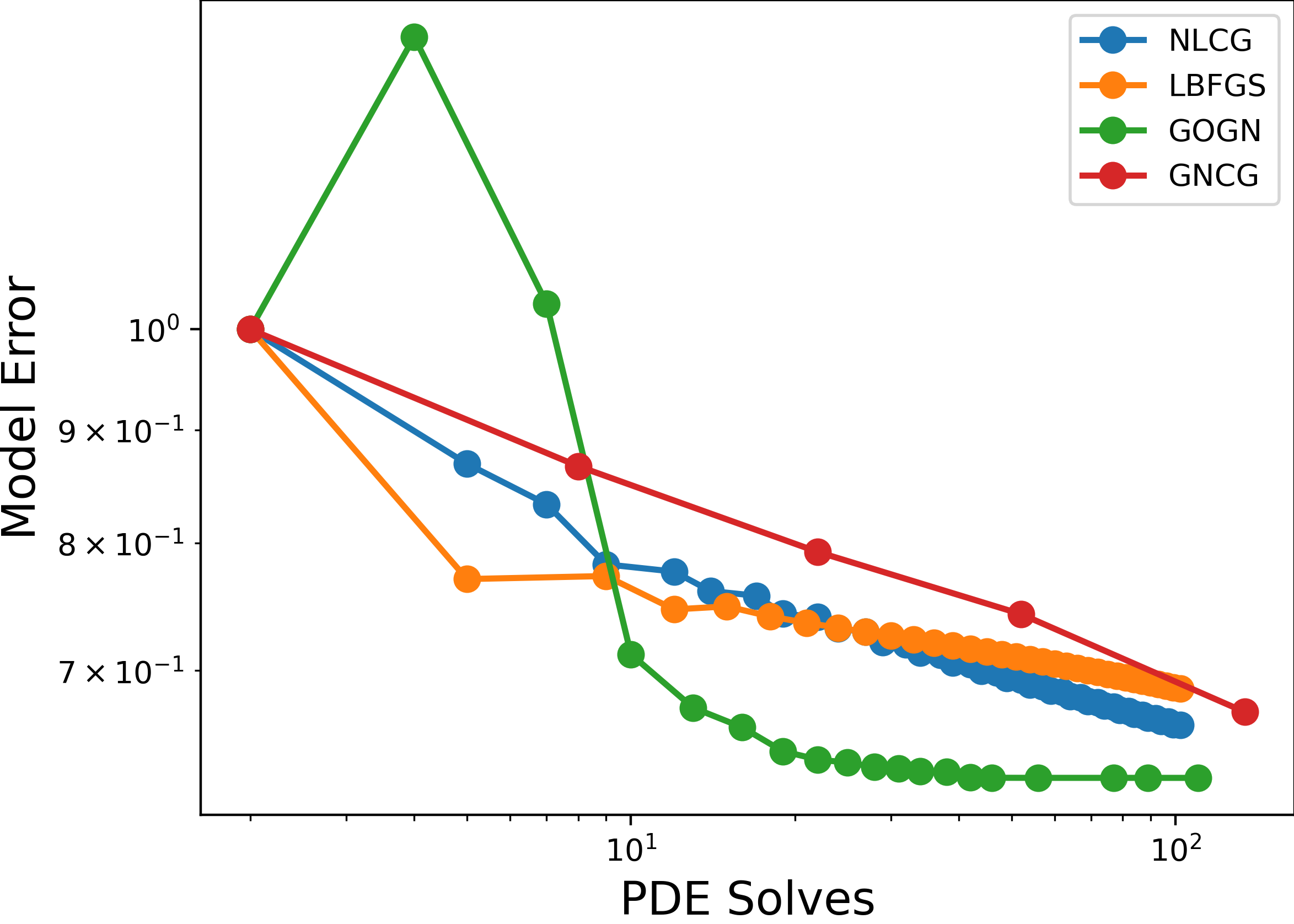}
    \includegraphics[width=0.3\linewidth]{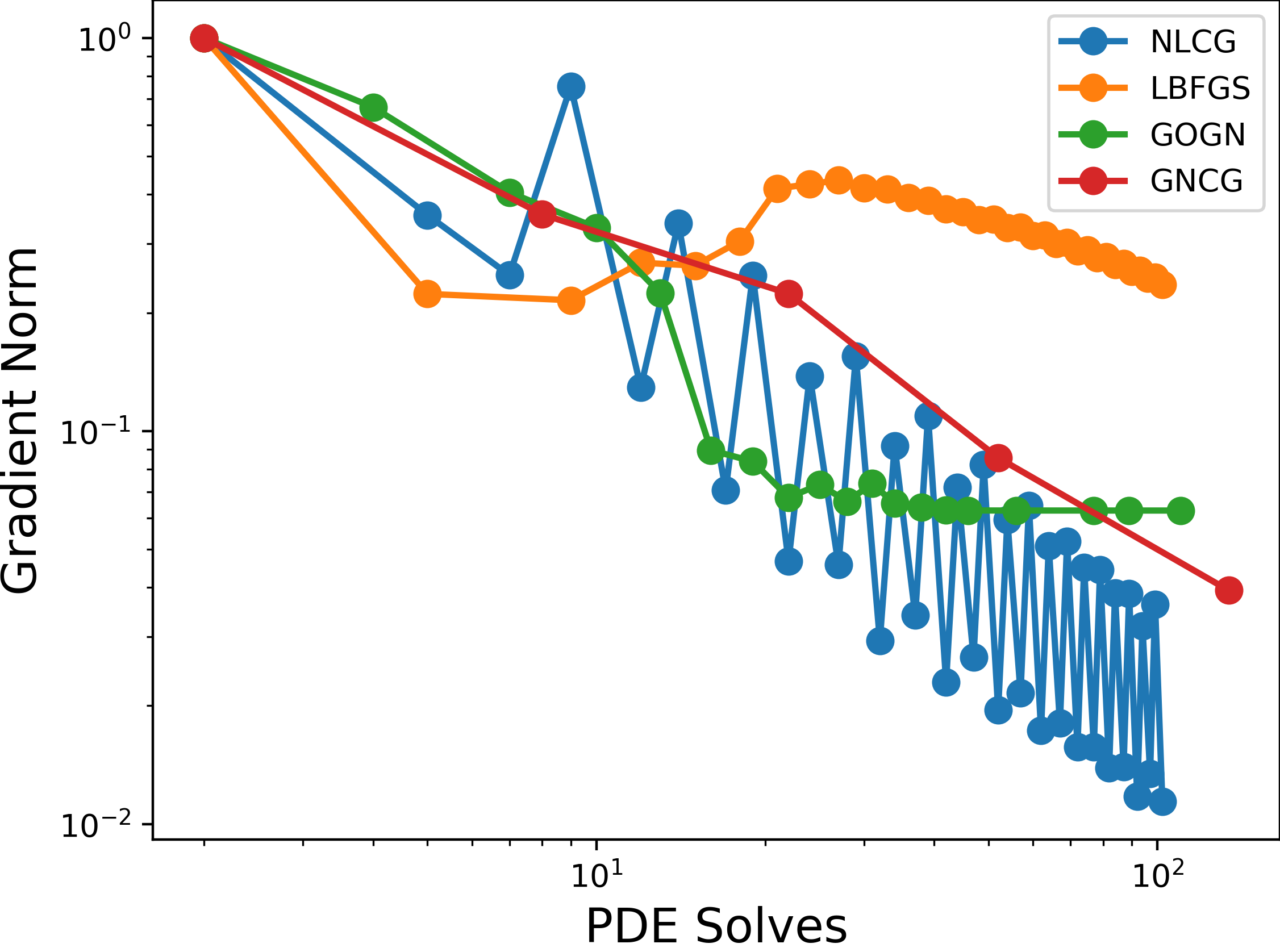}
    \includegraphics[width=0.3\linewidth]{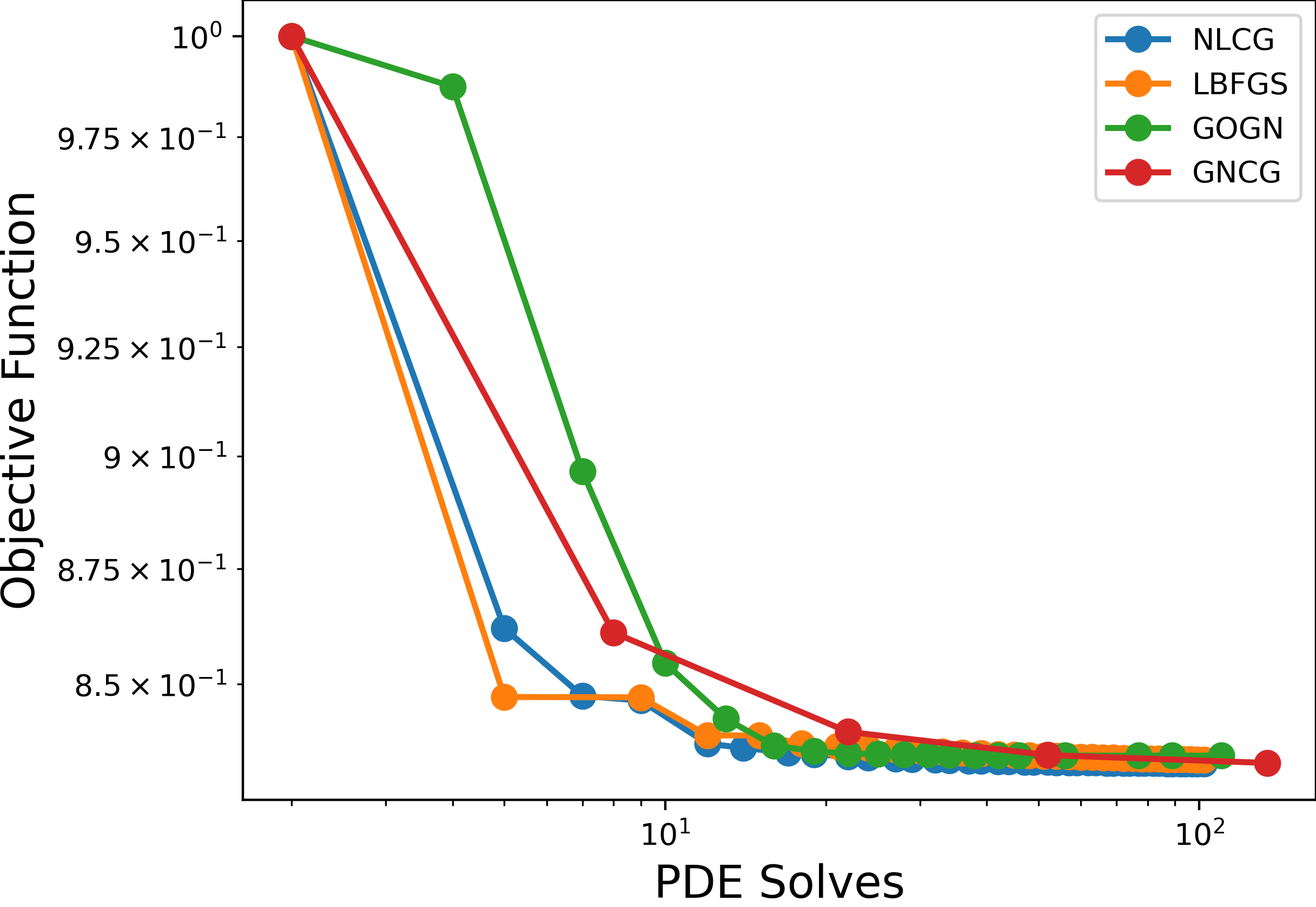}       
    \caption{Convergence plots for uniformly distributed receiver coverage (top row) and realistic coverage (bottom row). X-axes across all images denotes number of PDE solves during optimization. Y-axes denote model error (left), gradient norm (middle), and objective function values (right) for $8$ and $5$ sources, respectively, at a noise level of $\sigma = 1.0$}
    \label{fig:optstats_simple0}
\end{figure}

\begin{figure}[t]
    \centering
\includegraphics[width=0.4\linewidth]{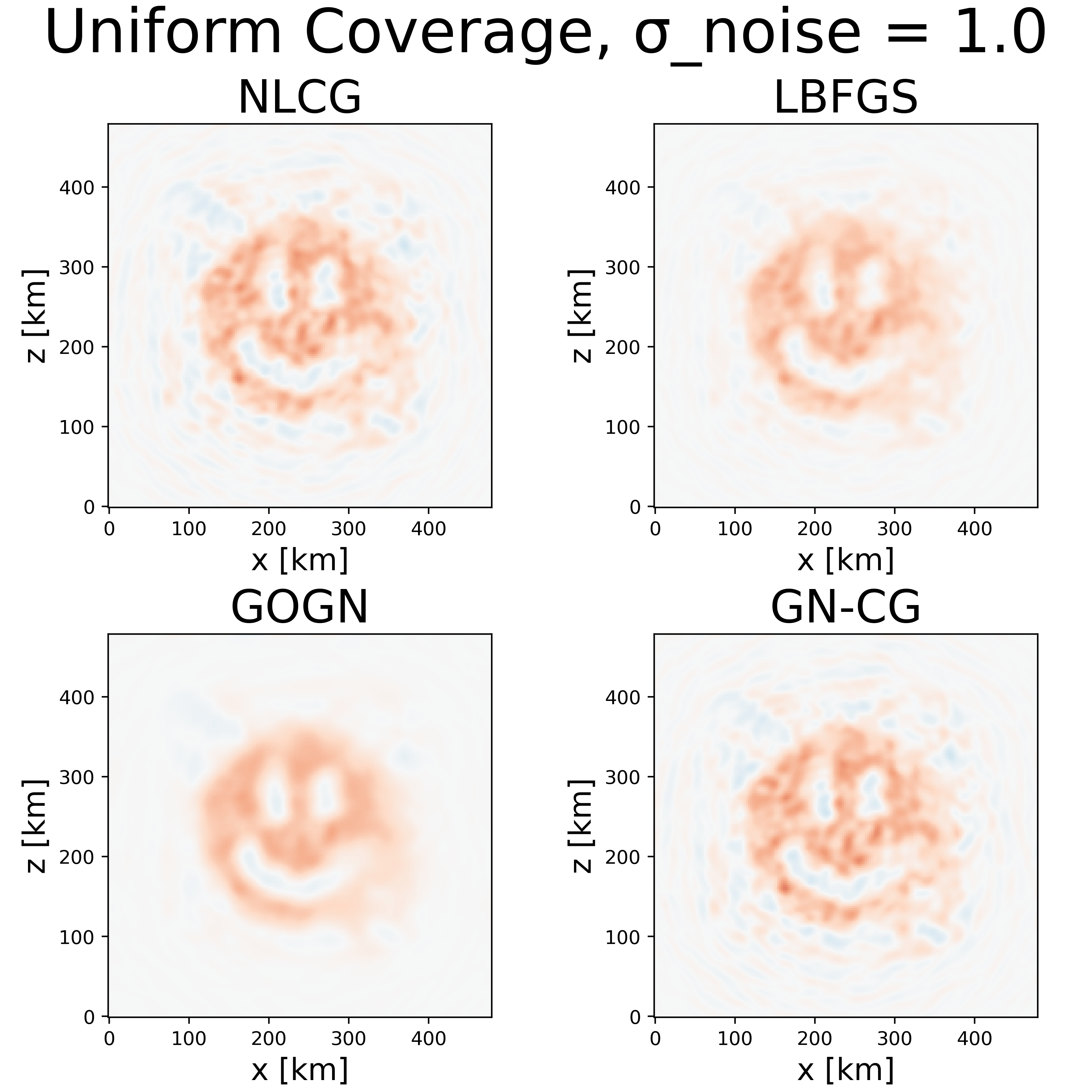}   
\includegraphics[width=0.4\linewidth]{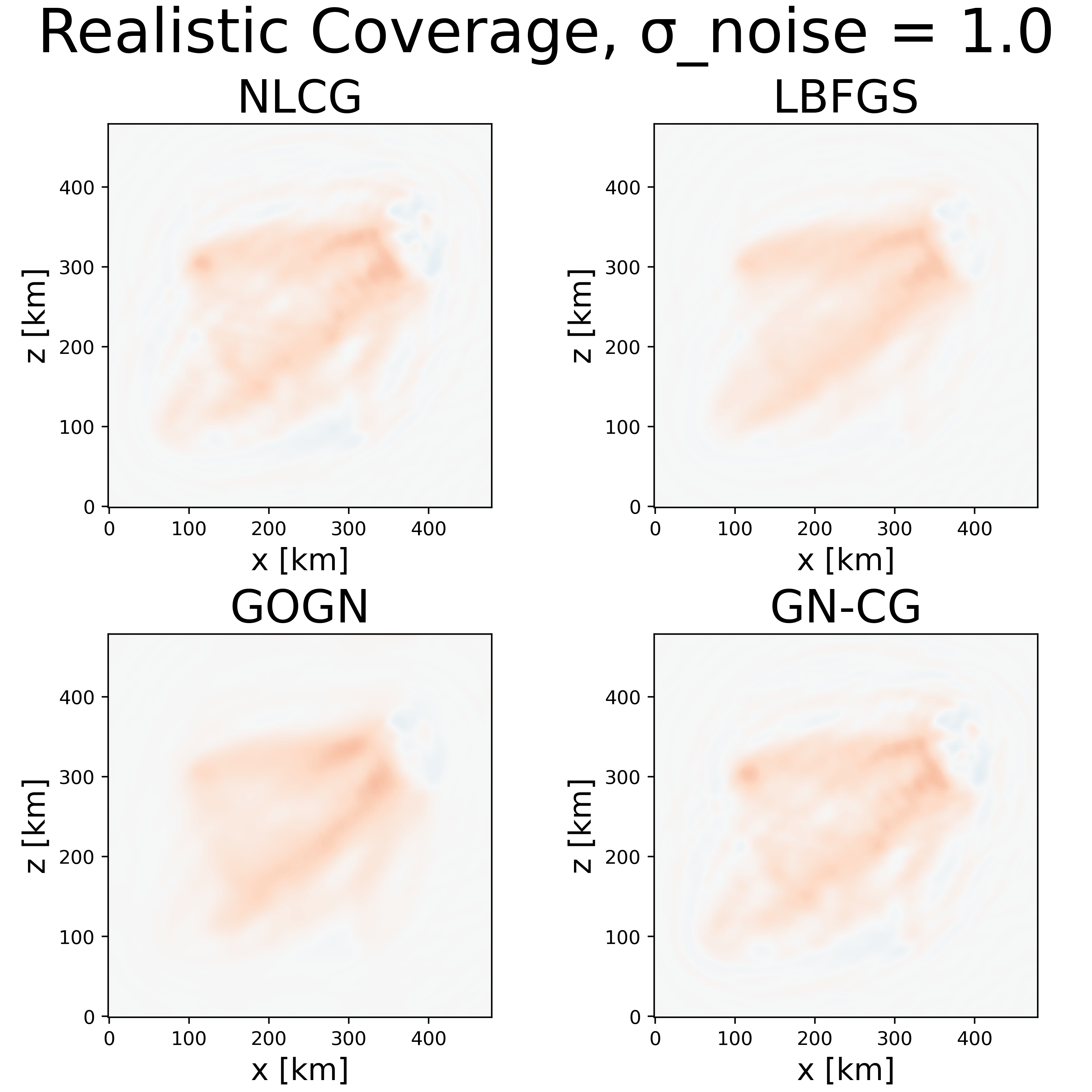}   
\includegraphics[width=.102\linewidth, trim=0cm 4cm 0cm 0cm, clip]{figures/cbar.png}    
    \caption{Final reconstructions from the experiments in Figure~\ref{fig:optstats_simple0} }
    \label{fig:rec_simple0}
\end{figure}

\begin{figure}[t]
    \centering
    Uniform Coverage, $\sigma = 1.0$ \newline 
    \includegraphics[width=0.3\linewidth]{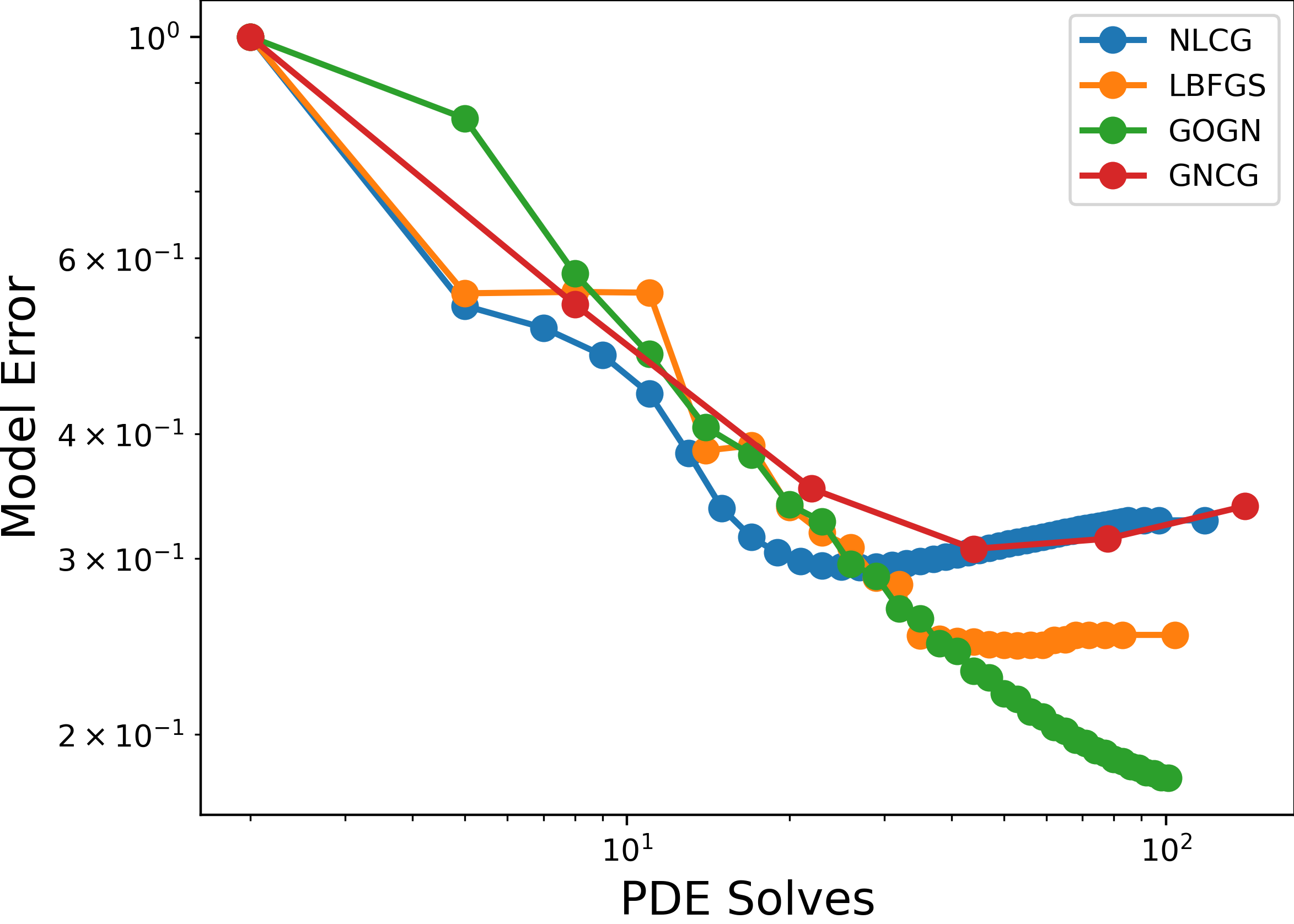}
    \includegraphics[width=0.3\linewidth]{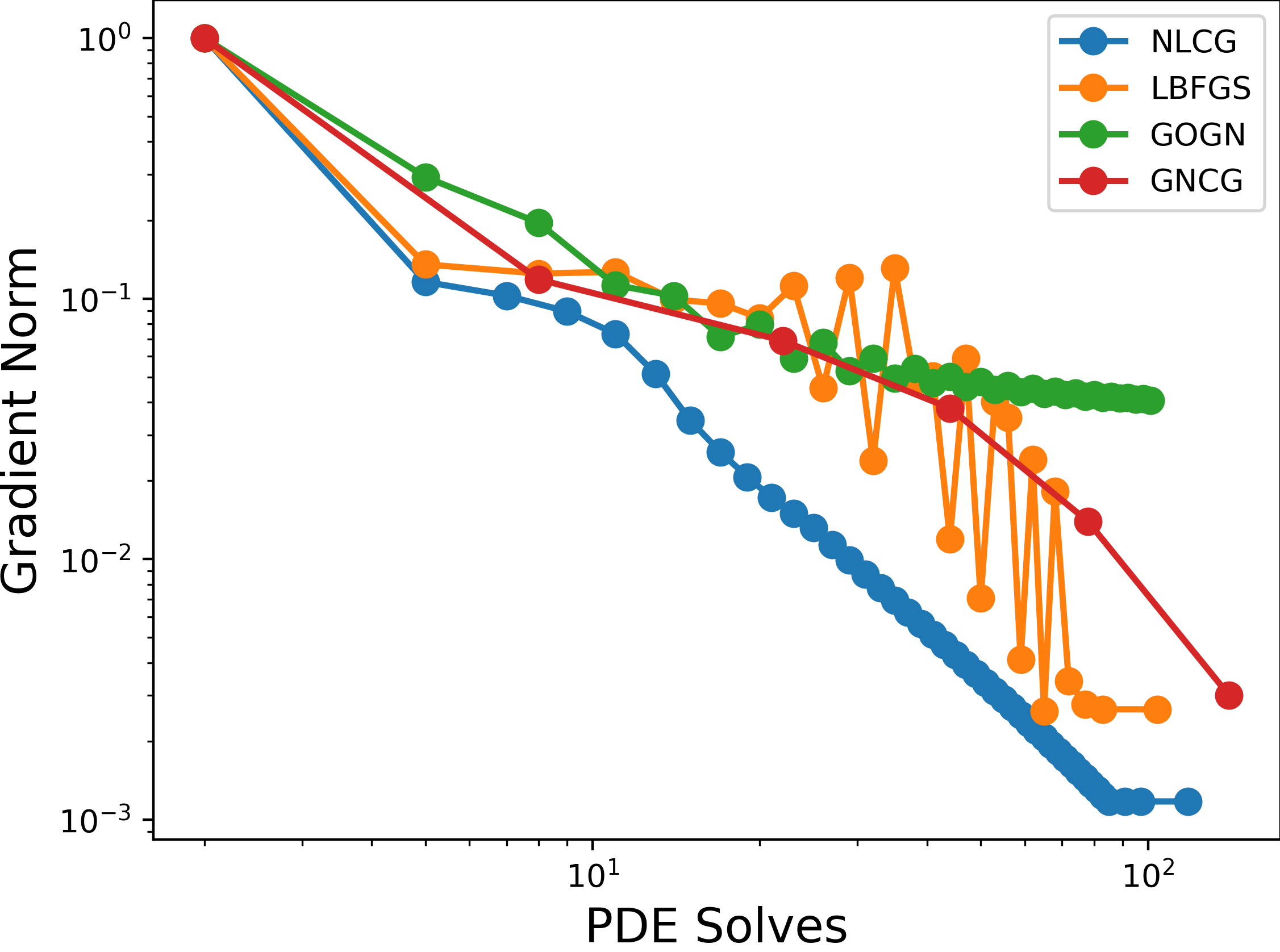}
    \includegraphics[width=0.3\linewidth]{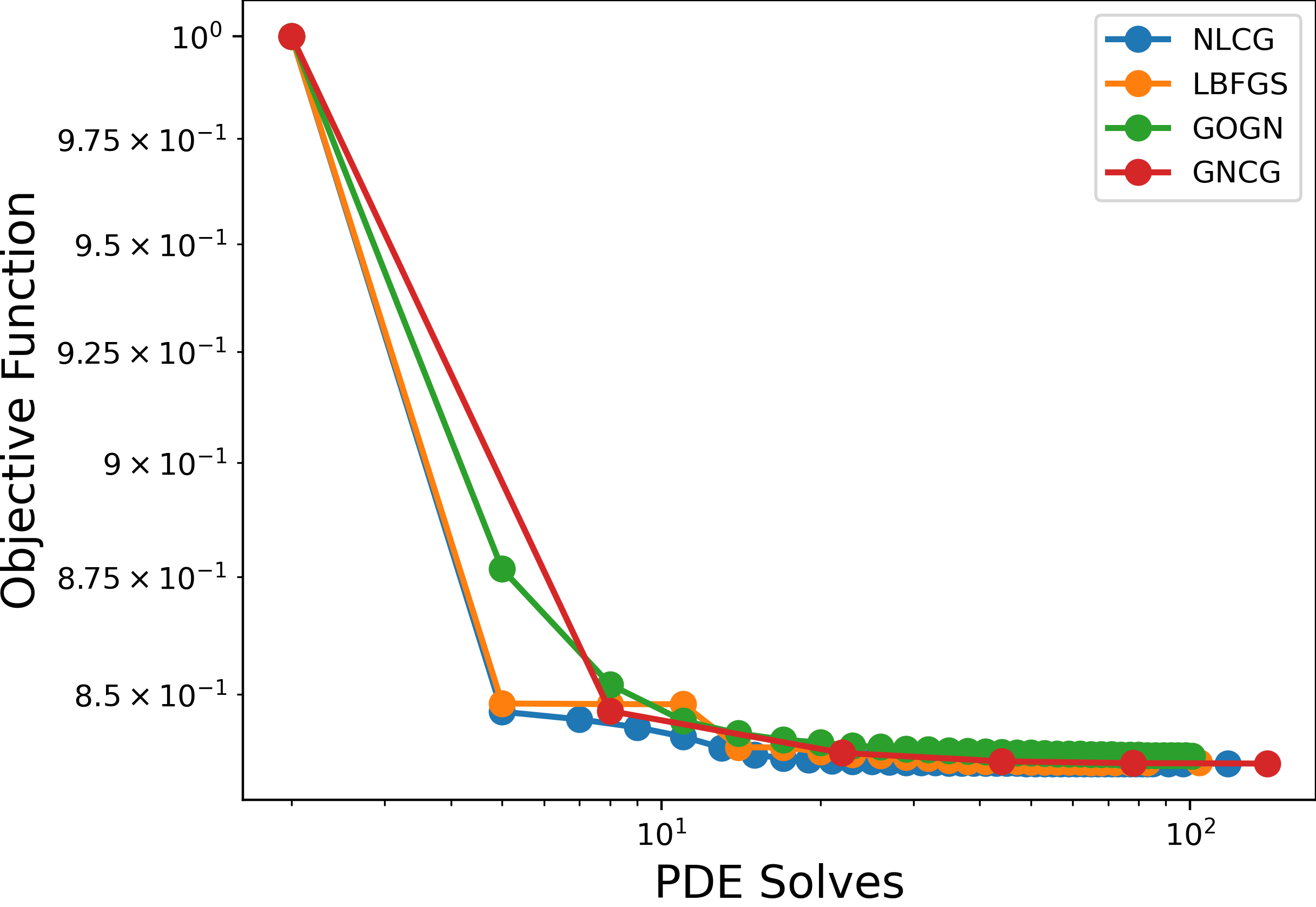}   
    \newline Realistic Coverage, $\sigma = 1.0$ \newline 
    \includegraphics[width=0.3\linewidth]{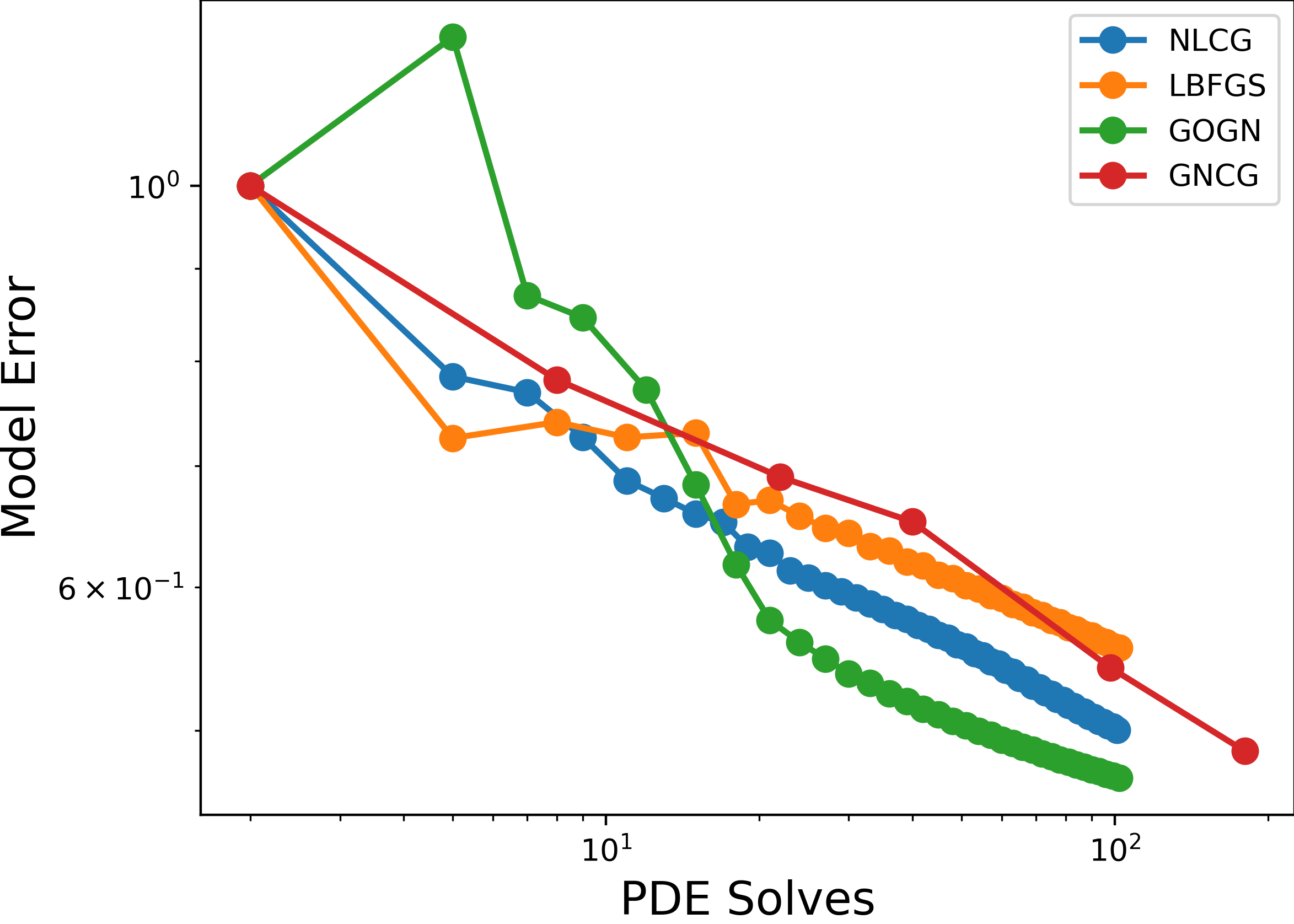}
    \includegraphics[width=0.3\linewidth]{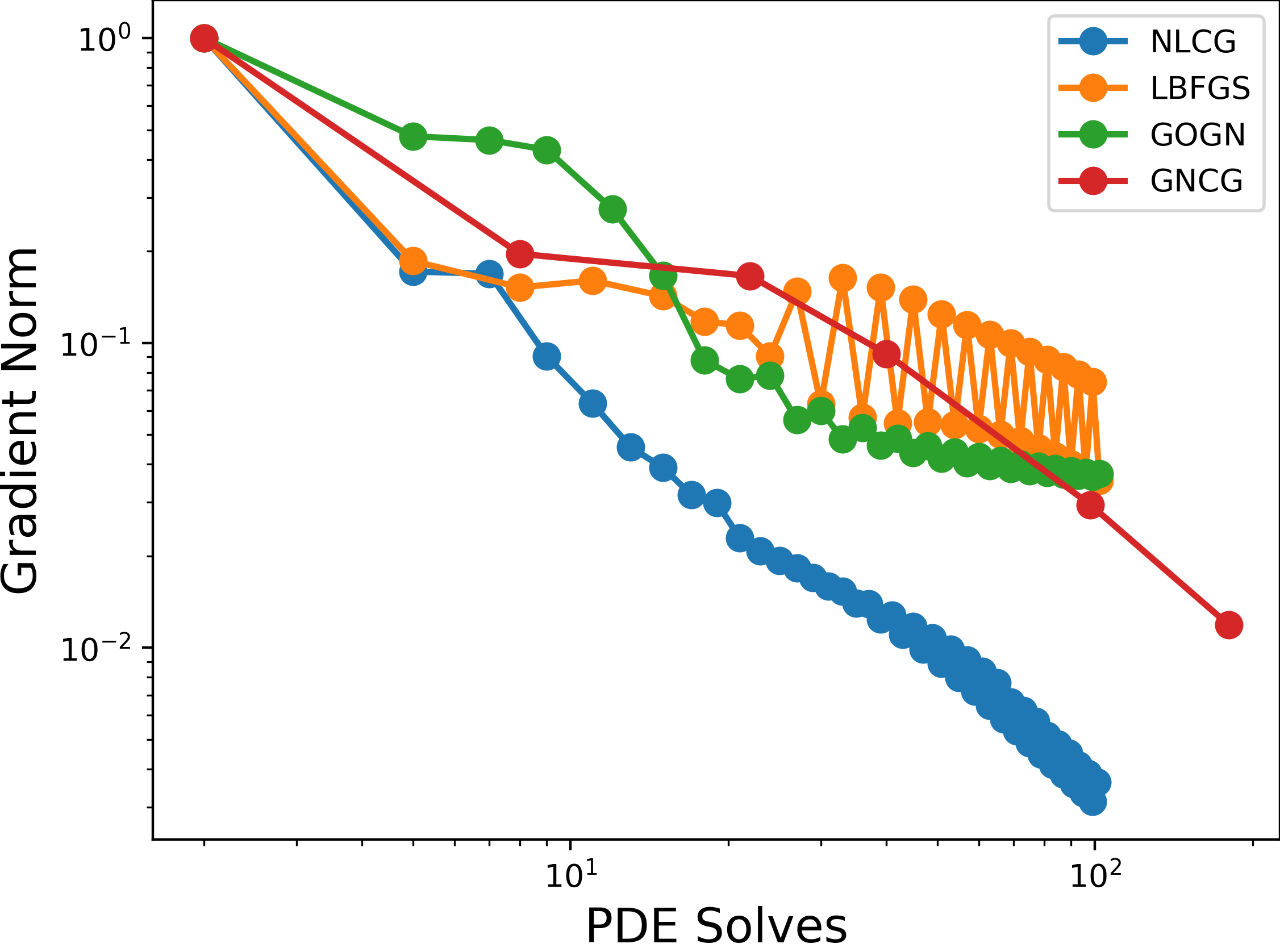}
    \includegraphics[width=0.3\linewidth]{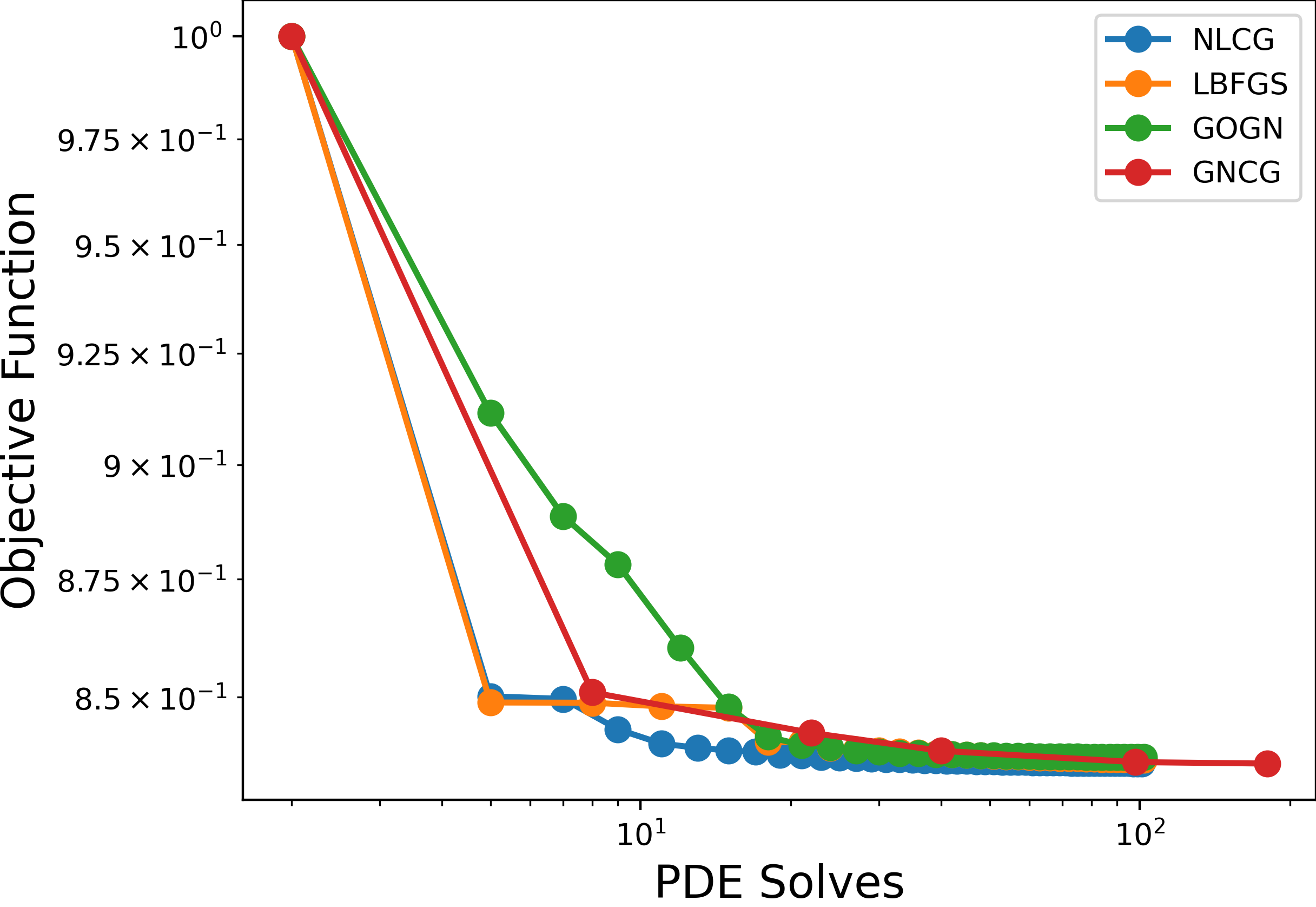}       
    \caption{Convergence plots for uniformly distributed receiver coverage (top row) and realistic coverage (bottom row), with X-axes across all images denoting number of PDE solves during optimization, and Y-axes denoting model error (left), gradient norm (middle), and objective function values (right) for $25$ sources at a noise level of $\sigma = 1.0$}
    \label{fig:optstats_simple250}
\end{figure}

\begin{figure}[t]
    \centering
\includegraphics[width=0.4\linewidth]{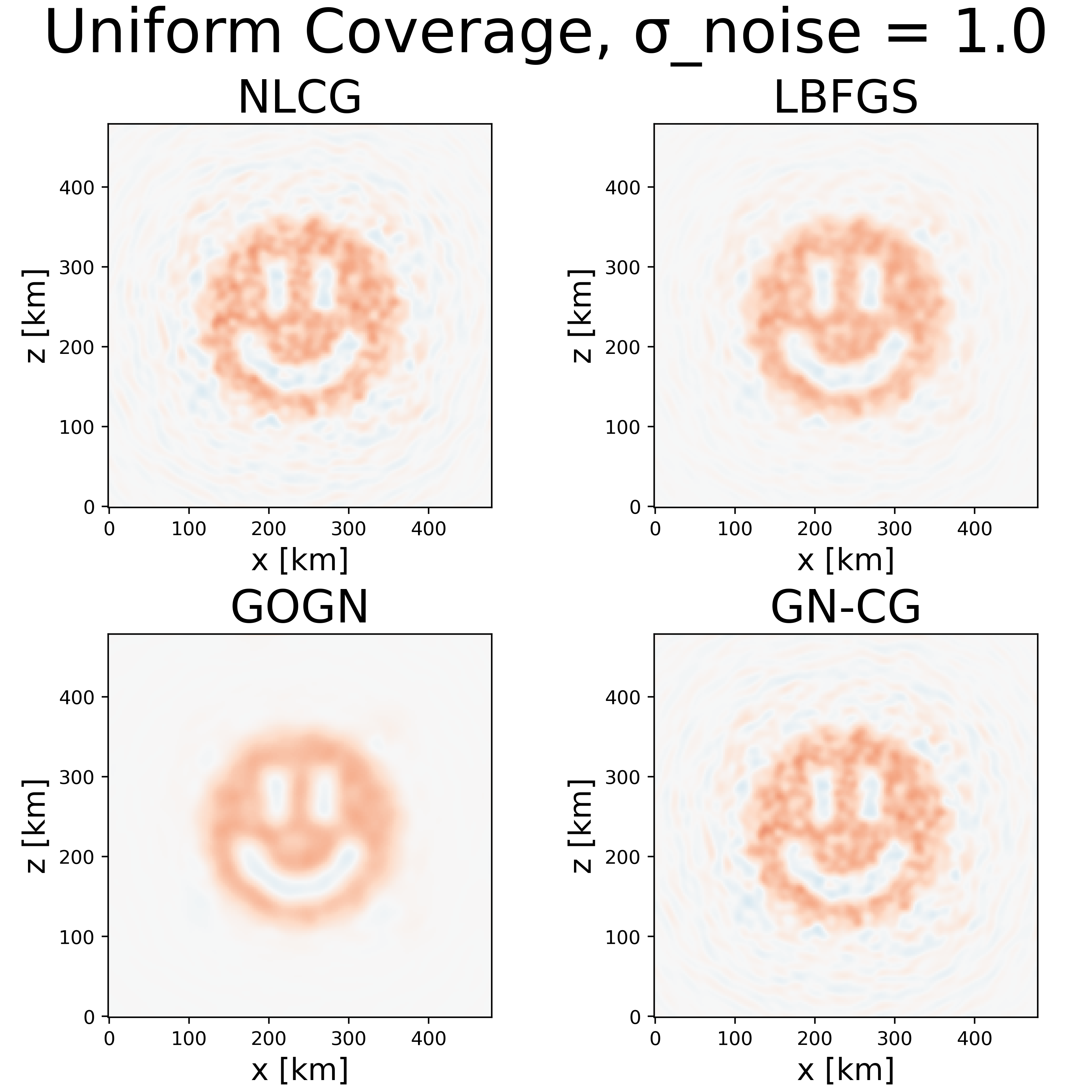}   
\includegraphics[width=0.4\linewidth]{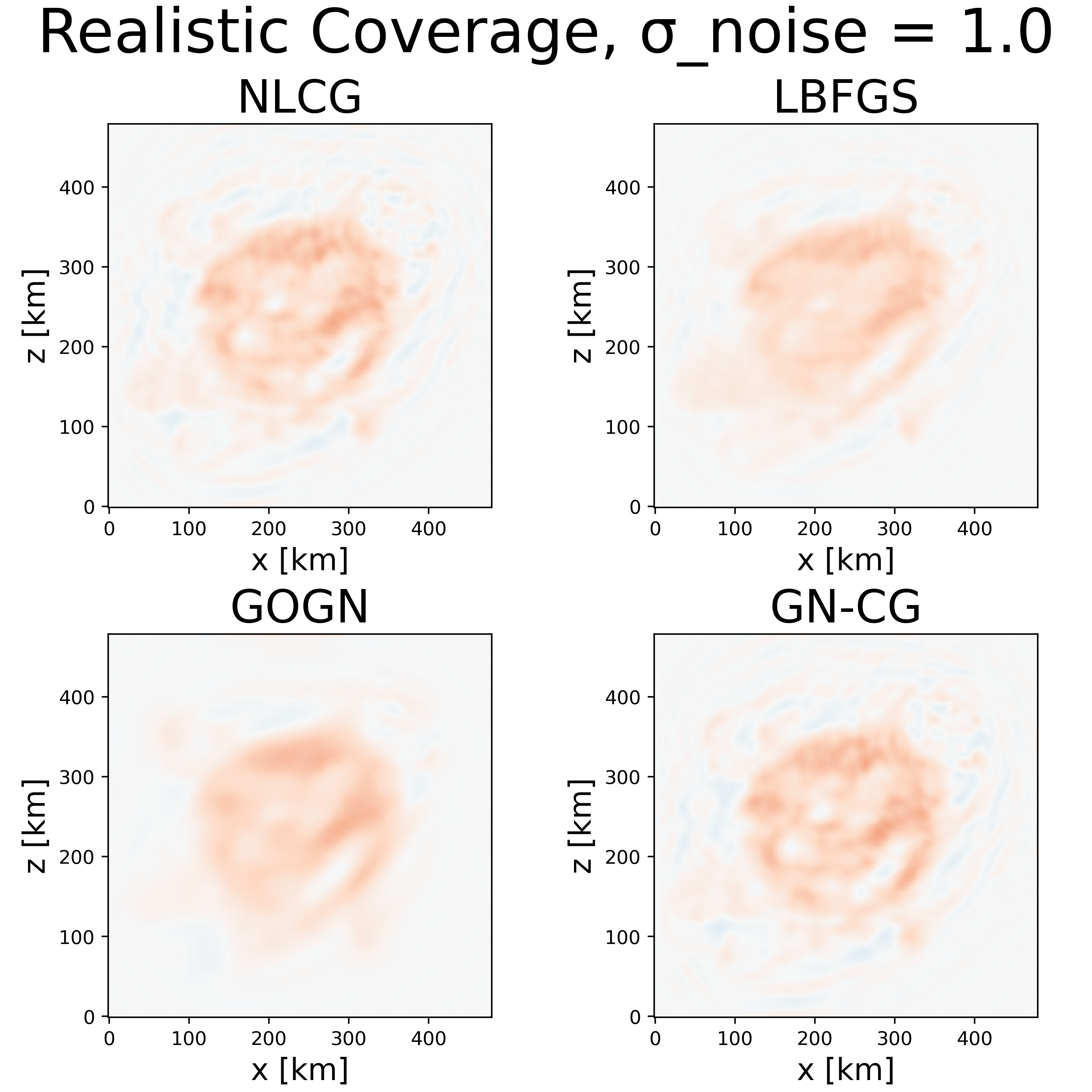}   
\includegraphics[width=.102\linewidth, trim=0cm 4cm 0cm 0cm, clip]{figures/cbar.png}    
    \caption{Final reconstructions from the experiments in Figure~\ref{fig:optstats_simple250} }
    \label{fig:rec_simple250}
\end{figure}

\end{document}